\documentclass[a4paper,10pt,reqno]{amsart}
\usepackage{amscd,amssymb,graphicx,color,epigraph}
\usepackage[shortalphabetic,initials]{amsrefs}

\usepackage{mathtools}
\usepackage{amsmath, amsthm, amssymb}
\usepackage{mathrsfs}
\usepackage[colorlinks,allcolors=blue]{hyperref}
\usepackage[all,cmtip]{xy}
\usepackage{tikz}
\usepackage{tikz-cd}
\usepackage{float}
\usepackage[shortlabels]{enumitem}
\setlist[itemize]{leftmargin=18pt}
\setlist[enumerate]{leftmargin=18pt}

\usetikzlibrary{arrows}
\usetikzlibrary{shapes.geometric}
\usetikzlibrary{decorations.markings}
\usetikzlibrary{decorations.pathreplacing}
\usepackage{pgfmath}

\usepackage{amscd}

\usepackage{ulem}
\usepackage{comment}

\let\amsamp=&

\theoremstyle{plain}

\newtheorem{thm}{Theorem}[section]
\newtheorem{lem}[thm]{Lemma}
\newtheorem{prop}[thm]{Proposition}
\newtheorem{cor}[thm]{Corollary}
\newtheorem{conj}[thm]{Conjecture}

\theoremstyle{definition}

\newtheorem{definition}[thm]{Definition}
\newtheorem{example}[thm]{Example}
\newtheorem{notation}[thm]{Notation}

\newtheorem{rmk}[thm]{Remark}

\newcommand{\f}[1]{\mathbb{#1}}

\newcommand{\Q}{\mathbb{Q}} 
\newcommand{\bT}{\mathbb{T}} 
\newcommand{\bA}{\mathbb{A}} 
\newcommand{\bH}{\mathbb{H}} 
\renewcommand{\k}{\mathbf{k}} 
\newcommand{\A}{\mathscr{A}}

\newcommand{\cO}{\mathcal{O}}
\newcommand{\cV}{\mathcal{V}}

\newcommand{\cE}{\mathcal{E}}
\newcommand{\cB}{\mathcal{B}}
\newcommand{\sE}{\mathscr{E}}

\newcommand{\cW}{\mathcal{W}}
\newcommand{\cH}{\mathcal{H}}
\newcommand{\cF}{\mathcal{F}}
\newcommand{\cK}{\mathcal{K}}
\newcommand{\cR}{\mathcal{R}}

\newcommand{\bZ}{\Bbb Z}

\newcommand{\bP}{\Bbb P}
\newcommand{\bL}{\Bbb L}
\newcommand{\bK}{\Bbb K}
\newcommand{\bF}{\Bbb F}

\newcommand{\bb}{\mathfrak{b}}

\newcommand{\rk}{\mathop{\hbox{\rm{rk}}}}

\newcommand{\Spec}{\mathop{\hbox{\rm{Spec}}}}
\newcommand{\Ext}{\mathop{\hbox{\rm{Ext}}}}
\newcommand{\Mat}{{\mathop{\hbox{\rm{Mat}}}}}
\newcommand{\Hom}{{\mathop{\hbox{\rm{Hom}}}}}
\newcommand{\RHom}{{\mathop{\hbox{\rm{RHom}}}}}
\newcommand{\End}{{\mathop{\hbox{\rm{End}}}}}
\newcommand{\Def}{{\mathop{\hbox{\rm{Def}}}}}
\DeclareMathOperator{\IF}{if}

\newcommand{\Ker}{\mathrm{Ker}}
\newcommand{\Perf}{\mathrm{Perf}}

\title[Deformations of Kalck--Karmazyn algebras via  Mirror Symmetry]{Deformations of Kalck--Karmazyn algebras \\ via  Mirror Symmetry}

\author{Yank\i\ Lekili}
\address{Department of Mathematics, Imperial College, London, UK}
\email{y.lekili@imperial.ac.uk}

\author{Jenia Tevelev}
\address{Department of Mathematics \& Statistics, University of Massachusetts, Amherst, USA.}
\email{tevelev@umass.edu}

\begin{document}

\begin{abstract}
As observed by Kawamata \cite{K21}, a $\mathbb{Q}$-Gorenstein smoothing of a Wahl singularity gives rise to a one-parameter flat degeneration of a matrix algebra.
A~similar result holds for a general smoothing of any two-dimensional cyclic quotient singularity, where the matrix algebra
is replaced by a hereditary algebra \cite{TU22}.
From a categorical perspective, these one-parameter families of finite-dimensional algebras "absorb" the singularities of the threefold total spaces of smoothings. These results were established using abstract methods of birational geometry, making the explicit computation of the family of algebras challenging.
Using mirror symmetry for genus-one fibrations \cite{LPol}, we identify a remarkable immersed Lagrangian with a bounding cochain in the punctured torus. The endomorphism algebra of this Lagrangian in the relative Fukaya category corresponds to this flat family of algebras. This enables us to compute Kawamata's matrix order explicitly.
\end{abstract}

\maketitle

\section{Introduction}

The notion of  singularity category \cite{Buchw,Orlov} provides a direct way to compare module categories of rings of dissimilar nature, such as the local rings of singular algebraic varieties and finite-dimensional algebras.
Sometimes one  can even find % one can prove that 
an algebra $\cR$ that ``absorbs'' 
\cite{KuSh}
singularities of an algebraic variety $\cW$, i.e.,
there exists a semi-orthogonal decomposition $D^b(\cW)=\langle D^b(\cR),\cB\rangle$ such that $\cB\subset\Perf(\cW)$.
The algebra $\cR$ is typically presented as the endomorphism algebra of some object in $D^b(\cW)$, such as a vector bundle on $\cW$. The goal of this paper is to 
demonstrate that homological mirror symmetry can help to compute the algebra $\cR$ explicitly.

Consider a cyclic quotient singularity $\bA^2 / \mu_r$, where the primitive root of unity $\zeta \in \mu_r$ acts on $\bA^2$ with weights $(\zeta, \zeta^a)$, and $a$ and $r$ are coprime. This singularity, denoted by ${1 \over r}(1, a)$, is absorbed, after an appropriate compactification, by an $r$-dimensional algebra $R_{r,a}$ called the Kalck--Karmazyn algebra \cite{KK17, KKS}.

For example, the singularity ${1 \over 4}(1, 1)$ (the cone over the rational normal curve of degree $4$)  is absorbed by the $4$-dimensional algebra $R_{4,1} = k[x, y, z]/(x, y, z)^2$.
In~general, we show that $R_{r,a}$ has a  simple multiplication table; see Corollary~\ref{aethdthat}.

A singularity of dimension $n+1$ can be viewed as the total space of a deformation of an $n$-dimensional singularity. The notion of categorical absorption is modified so that $\cR$ is now a $B$-algebra, where $\Spec B$ is the base of the deformation.
The algebra $\cR$ was constructed for general deformations of ${1\over r}(1, a)$ in~\cite{TU22}. It~is flat over $B$ and has the Kalck--Karmazyn algebra $R_{r,a}$ as the special fiber. Its~general fiber is Morita-equivalent to the path algebra of an acyclic quiver.
The proof is based on \cite{K21}, which studied the following special case: the singularity is Wahl, i.e., $r=n^2$ and $a=nq-1$ for coprime $n$ and $q$, and the smoothing is $\mathbb{Q}$-Gorenstein, meaning the relative canonical divisor is $\mathbb{Q}$-Cartier.
In this case, it turns out that the deformation is absorbed by a matrix order. Recall that a matrix order over, say, $B = k[t]$, is a flat $B$-algebra $\cR$ such that $\cR \otimes_B K = \Mat_n(K)$, where $K = k(t)$.

\begin{example}
Let  $\cR \subset  \Mat_2(k[t])$ be a subalgebra given by elements of the form
\begin{equation}\label{sfvfbdfad}\left[\small\begin{matrix}
a_0&
       t a_{3}\\
       {t a_{1}}&
       t a_{2}+a_{0}\\
\end{matrix}\right]
\end{equation}
where $a_i \in k[t]$ for $i=0,1,2,3$.
It can be verified that $\cR$ is an order which gives a flat deformation of the $4$-dimensional algebra  $R_{4,1} = k[x,y,z]/(x,y,z)^2$.
In~fact, this order corresponds to the $\Q$-Gorenstein smoothing of the singularity ${1\over 4}(1,1)$.
\end{example}

One of our main results is the calculation of a matrix order $\cR_{n^2,nq-1}$, which corresponds to the $\mathbb{Q}$-Gorenstein smoothing of any Wahl singularity ${1\over n^2}(1, nq-1)$. Before describing this order, we first explain our geometric approach.

In Section~\ref{sfbdfbdhzdtjts}, we compactify the singularity ${1\over r}(1, a)$ using a projective surface $W$ and review the construction of a remarkable vector bundle $F$ on $W$ of rank~$r$, called the Kawamata vector bundle. The Kalck--Karmazyn algebra $R_{r,a}$ is defined as the endomorphism algebra $\mathrm{End}(F)$. It is an $r$-dimensional $k$-algebra.
The Kawamata vector bundle deforms to the vector bundle $\cF$ on the total space $\cW$ of any deformation of $W$, producing a flat family $\cR = \mathrm{End}(\cF)$ of $r$-dimensional algebras.

To apply mirror symmetry, we choose a compactification $W$  that contains an anticanonical divisor $E = A \cup B$, a curve of arithmetic genus $1$. The components $A$ and $B$, both isomorphic to $\bP^1$, intersect at the singular point $P$ of $W$ as the orbifold coordinate axes of $\bA^2 / \mu_r$ and at a smooth point $Q$ transversally (see Figure~\ref{zxfbxzfbzfnzdf}). We~reduce the computation to $E$ by showing, in Lemma~\ref{sGSGsgsr}, that $R_{r,a} \cong \mathrm{End}(F|_E)$.

\begin{figure}[htbp]
\begin{center}
\includegraphics[width=\textwidth]{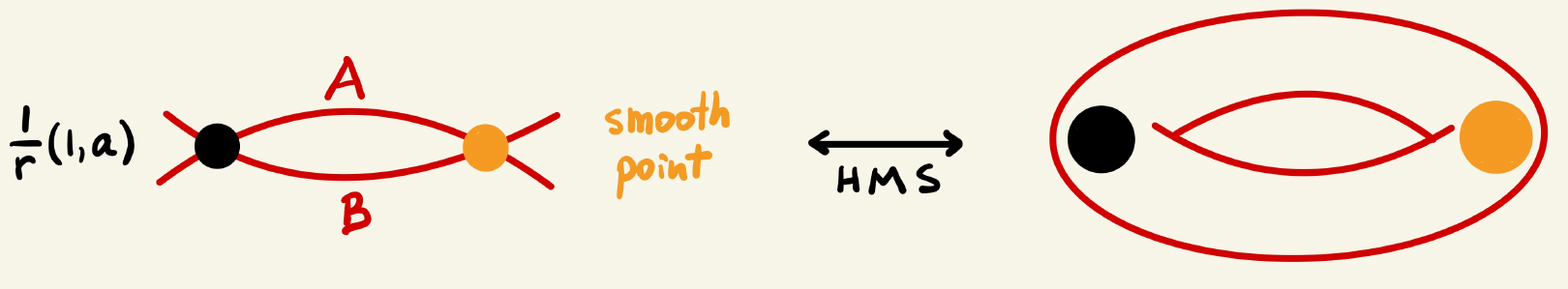}
\caption{The divisor $E$ and its mirror, the two-punctured torus $\bT_2$}
\label{zxfbxzfbzfnzdf}
\end{center}
\end{figure}

Instead of computing $\mathrm{End}(F|_E)$ directly in the perfect derived category $\Perf(E)$, we use homological mirror symmetry. By our construction, $E$ is a cycle of two projective lines, where the irreducible components meet transversely at two points.
We write $E_n$ for the  curve with $n$ irreducible components, each isomorphic to $\mathbb{P}^1$, such that the intersection complex is an $n$-gon. Homological mirror symmetry for $E_n$ was proven in \cite{LPol} (and alternatively in \cite{LPolAus}) as an explicit quasi-equivalence between the split-closed derived compact Fukaya category of the $n$-punctured torus $\mathbb{T}_n$ and the perfect derived category of $E_n$:\begin{equation} \label{hms2} \mathcal{F}(\mathbb{T}_n) \simeq \mathrm{Perf}(E_n). \end{equation}

\begin{definition}
 The Kawamata Lagrangian $\Bbb K_{r,a} \in \mathcal{F}(\mathbb{T}_2)$ is the mirror Lagrangian of the vector bundle $F|_{E} \in \mathrm{Perf}(E_2)$ under homological mirror symmetry. 
\end{definition}

In our calculations with Fukaya categories, we will always assume that $n \geq 1$. Since the symplectic surface $\mathbb{T}_n$ is punctured at least once, its symplectic form is exact, $\omega = d\lambda$, and we use the exact Fukaya category as defined in~\cite{Seidelbook}.
The objects of $\mathcal{F}(\mathbb{T}_n)$ are connected, compact, and exact Lagrangians $\mathbb{L}$ (i.e., $\lambda|\mathbb{L}$ is exact) with a choice of brane data (spin structure, a $U(1)$-local system, and grading data). Changing the spin structure or local system on $\mathbb{L}$ in $\mathcal{F}(\mathbb{T}_n)$ corresponds to tensoring the corresponding complex in $\mathrm{Perf}(E_n)$ with a topologically trivial line bundle (equivalently, one of degree $0$ on all irreducible components of $E_n$), while changing the grading data corresponds to a shift in the triangulated category.
In our illustrations of Lagrangians, we choose a closed curve from every homotopy class to represent the unique (up to Hamiltonian isotopy) exact Lagrangian in that class.

The Kawamata Lagrangian $\bK_{r,a}$ is computed in Theorem~\ref{argargerhwetb} and illustrated in Figure~\ref{fbadhdthtsthstr}, where $b$ is the inverse of $a$ modulo $r$. We represent a $2$-torus as a rectangle with opposite sides identified. Note the two punctures near the NW corner.

\begin{figure}[hbtp]
\begin{center}
\begin{tikzpicture}[scale=0.5]
    % Draw the square with black edges
    \draw[very thick] (0,0) rectangle (16,16);
    
    \foreach \i in {0,...,10}
    {
      \draw[purple, very thick] (0,10.5-\i) .. controls +(0:3) and +(180:3) .. (16, 15.5-\i);
    }

    \foreach \i in {1,2,3,4}
    {
      \draw[purple, very thick] (1.6+3.2*\i,0) .. controls +(30:1) and +(180:1) .. (16, 4.5-\i);
    }

    \foreach \i in {1,2,3,4}
    {
      \draw[purple, very thick] (0,15.5-\i) .. controls +(0:1) and +(190:1) .. (1.6+3.2*\i,16);
    }

   % \draw[purple, very thick] (0,15.5) -- (16, 4.5);
    \draw[purple, very thick] (0,15.5) .. controls +(45:1) and +(270:1) .. (16, 4.5);

    \draw[fill=black!70!black, thick] (0.3,14.95) circle (0.2cm);

    \draw[fill=orange,thick] (1.4,15.6) circle (0.2cm);

    \draw[decorate,decoration={brace,amplitude=10pt},xshift=0pt,yshift=5pt]
    (-0.5,0.5) -- (-0.5,15.5) node[midway,above, xshift=-15pt, yshift=-5pt] {$r$};

   \draw[decorate,decoration={brace,amplitude=10pt},xshift=0pt,yshift=5pt]
    (16.5,15.5) -- (16.5,5.5) node[midway,above, xshift=15pt, yshift=-5pt] {$b$.};

\end{tikzpicture}
\end{center}
\caption{Kawamata Lagrangian $\Bbb K_{r,a}$ (here $r=16$, $a=3$, $b=11$)}
\label{fbadhdthtsthstr}
\end{figure}
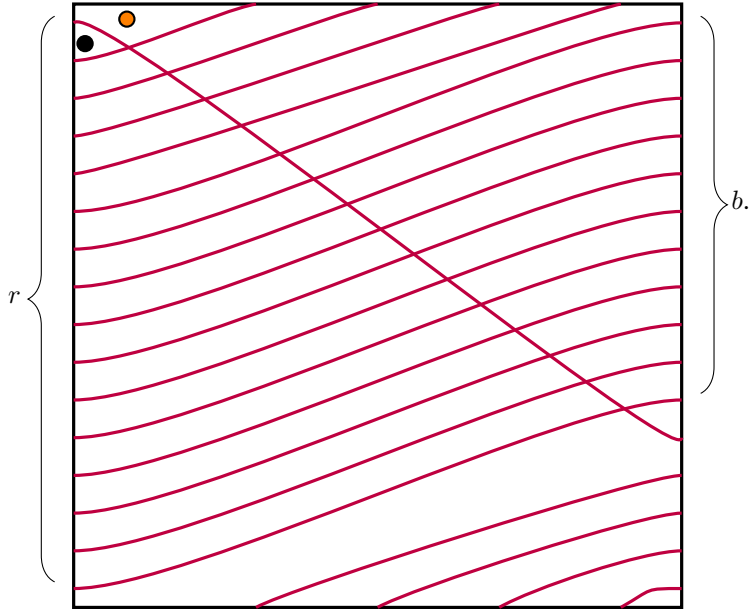

Thus, by the definition of $\bK_{r,a}$, we have that $\mathrm{End}(\bK_{r,a})$ is isomorphic to the Kalck--Karmazyn algebra $R_{r,a}$. 
While we are mostly interested in deformations of $R_{r,a}$, our study also
yields a simple multiplication table for the algebra $R_{r,a}$ itself, which was previously unknown.
We find this description of $R_{r,a}$, provided below, easier to work with than its presentation by generators and relations \cite[p.~3]{KK17}.

\begin{cor}\label{aethdthat}
The Kalck--Karmazyn algebra 
$R_{r,a}$ has basis $w_i$ for $i\in\bZ_r$ and product
\begin{equation}\label{sDBASBASBRQ}
w_jw_i=\begin{cases}
w_{j+i} & \hbox{\rm if a certain condition is satisfied}\cr
0 &  \hbox{\rm otherwise.}
\end{cases}
\end{equation}
To explain the condition, let $\gamma:\,\bZ^2\to\bZ_r$ be a %surjective 
homomorphism
$(i,j)\mapsto j-bi \mod r$ and consider a sublattice 
$\Gamma = \Ker(\gamma)\subset\bZ^2$.
We  plot points of $\Gamma$ as orange dots, as they correspond to the orange puncture (see Figure~\ref{fbadhdthtsthstr})
in the universal cover of the torus.
%be the kernel of $\gamma$.
%Explicitly, 
%\[ \Gamma = \{ (i,j) \subset \mathbb{Z}^2 : j\equiv bi \mod r \}. \]
Consider the biggest Young diagram in the first quadrant with the bottom left corner at $(0,0)$ that does not contain orange dots in its interior.
We~fill every box of this Young diagram with the number $\gamma(i,j)\in\bZ_r$, where $(i,j)$ is the bottom left corner of the box.

To~compute the product $w_jw_i$,
we locate the box filled with $j$ (resp.,with~$i$) in the bottom row (resp., left column) 
of the  Young diagram. If the smallest rectangle containing these boxes is contained in the Young diagram,
then  $w_jw_i=w_{j+i}$. Otherwise,  $w_jw_i=0$.
%In particular, $w_0$ is the identity element.
\end{cor}

This description of the Kalck--Karmazyn algebra $R_{r,a}$
arises from our analysis of holomorphic polygons with boundaries on the Kawamata Lagrangian.

\begin{example}
Suppose $r=9$ and $a=2$. Then $b=5$. From the Young diagram, 
the  non-trivial products in $R_{9,2}$ are 
$w_4w_1=w_5$, $w_4w_2=w_6$, $w_4w_3=w_7$, and $w_4^2=w_{8}$. 
\begin{center}
\begin{tikzpicture}[scale=0.6]
  \pgfmathtruncatemacro{\maxVal}{10} % You can adjust this value
  \pgfmathtruncatemacro{\minVal}{0} % You can adjust this value
  \pgfmathtruncatemacro{\q}{5} % You can adjust this value
  \pgfmathtruncatemacro{\n}{9} % You can adjust this value
    
  % Draw the integer lattice Z^2 with grey dots
    \foreach \x in {\minVal,...,\maxVal}
        \foreach \y in {\minVal,...,\maxVal}
            \fill[gray] (\x,\y) circle (0.05);

   % Loop to generate lattice points
   \foreach \i in {\minVal,...,\maxVal} {
     \foreach \j in {\minVal,...,\maxVal} {
       % Check if qi - j is divisible by n
       \pgfmathtruncatemacro{\result}{mod(\q*\i - \j, \n)}
       \ifnum\result=0
        % Draw lattice point
         \filldraw[orange] (\i,\j) circle (0.15);
        % Hexagon
        % \draw[orange] (\i-1,\j+1) -- (\i,\j+1);
        % \draw[orange] (\i,\j+1) -- (\i+1,\j);
        % \draw[orange] (\i+1,\j) -- (\i+1,\j-1);
        % \draw[orange] (\i+1,\j-1) -- (\i,\j-1);
        % \draw[orange] (\i,\j-1) -- (\i-1,\j);
        % \draw[orange] (\i-1,\j) -- (\i-1,\j+1);         
       \fi
    }
  }

\pgfmathtruncatemacro{\x}{0}  
\foreach \j in {0,...,\numexpr\n-1} {
   \foreach \i in {0,...,\numexpr\n}
   {
       \pgfmathtruncatemacro{\res}{mod(\q*\i - \j+\q, \n)}
       \ifnum \res= 0
           \xdef\x{\i}
           \breakforeach
         \fi
       
       \ifnum \j >0
            \ifnum \i > \numexpr\x-1
            \breakforeach
         \fi
         \fi

         % \fill[blue] (\i,\j) circle (0.2);
         \pgfmathtruncatemacro{\myv}{mod(\q*\i - \j,\n)}
         \ifnum \myv > 0
           \pgfmathtruncatemacro{\myv}{\n-\myv}
         \else
           \pgfmathtruncatemacro{\myv}{-\myv}
         \fi
         
         \ifnum \numexpr\i + \numexpr\j = 0
         \else  
         \fill[white] (\i,\j) circle (0.1);
         \fi
         \node at (\i,\j) {\myv};

       %\fi

  }
}

\node at (\maxVal+3, \minVal) {$r=9$, $a=2$}; 
\end{tikzpicture}

\end{center}
\end{example}

\begin{example}\label{akjbcjaksbdakjs}
$a=b=r-1$ or  $a=b=1$ are the only cases when $R_{r,a}$ is a commutative algebra. Below, 
the first case is illustrated on the left and the second case on the right (for $r=7$).
In the first case, $w_jw_i=w_{j+i}$ if and only if $j+i<r$
and so $R_{r,r-1}\cong\bZ[t]/t^r$ via an isomorphism $w_i\mapsto t^i$.
In the second case,  $w_jw_i=w_{j+i}$ if~and only if $i=0$ or $j=0$ and so
$R_{r,1}\cong\bZ[w_1,\ldots,w_{r-1}]/(w_1,\ldots,w_{r-1})^2$.

\begin{center}
\begin{tikzpicture}[scale=0.5]
  \pgfmathtruncatemacro{\maxVal}{8} % You can adjust this value
  \pgfmathtruncatemacro{\minVal}{0} % You can adjust this value
  \pgfmathtruncatemacro{\q}{6} % You can adjust this value
  \pgfmathtruncatemacro{\n}{7} % You can adjust this value
    
  % Draw the integer lattice Z^2 with grey dots
    \foreach \x in {\minVal,...,\maxVal}
        \foreach \y in {\minVal,...,\maxVal}
            \fill[gray] (\x,\y) circle (0.05);

   % Loop to generate lattice points
   \foreach \i in {\minVal,...,\maxVal} {
     \foreach \j in {\minVal,...,\maxVal} {
       % Check if qi - j is divisible by n
       \pgfmathtruncatemacro{\result}{mod(\q*\i - \j, \n)}
       \ifnum\result=0
        % Draw lattice point
         \filldraw[orange] (\i,\j) circle (0.15);
        % Hexagon
        % \draw[orange] (\i-1,\j+1) -- (\i,\j+1);
        % \draw[orange] (\i,\j+1) -- (\i+1,\j);
        % \draw[orange] (\i+1,\j) -- (\i+1,\j-1);
        % \draw[orange] (\i+1,\j-1) -- (\i,\j-1);
        % \draw[orange] (\i,\j-1) -- (\i-1,\j);
        % \draw[orange] (\i-1,\j) -- (\i-1,\j+1);         
       \fi
    }
  }

\pgfmathtruncatemacro{\x}{0}  
\foreach \j in {0,...,\numexpr\n-1} {
   \foreach \i in {0,...,\numexpr\n}
   {
       \pgfmathtruncatemacro{\res}{mod(\q*\i - \j+\q, \n)}
       \ifnum \res= 0
           \xdef\x{\i}
           \breakforeach
         \fi
       
       \ifnum \j >0
            \ifnum \i > \numexpr\x-1
            \breakforeach
         \fi
         \fi

         % \fill[blue] (\i,\j) circle (0.2);
         \pgfmathtruncatemacro{\myv}{mod(\q*\i - \j,\n)}
         \ifnum \myv > 0
           \pgfmathtruncatemacro{\myv}{\n-\myv}
         \else
           \pgfmathtruncatemacro{\myv}{-\myv}
         \fi
         
         \ifnum \numexpr\i + \numexpr\j = 0
         \else  
         \fill[white] (\i,\j) circle (0.1);
         \fi
         \node at (\i,\j) {\myv};

       %\fi

  }
}

%\node at (\maxVal+3, \minVal) {r=\n, a=\q}; 
\end{tikzpicture}
\qquad\qquad
\qquad\qquad
\begin{tikzpicture}[scale=0.5]
  \pgfmathtruncatemacro{\maxVal}{8} % You can adjust this value
  \pgfmathtruncatemacro{\minVal}{0} % You can adjust this value
  \pgfmathtruncatemacro{\q}{1} % You can adjust this value
  \pgfmathtruncatemacro{\n}{7} % You can adjust this value
    
  % Draw the integer lattice Z^2 with grey dots
    \foreach \x in {\minVal,...,\maxVal}
        \foreach \y in {\minVal,...,\maxVal}
            \fill[gray] (\x,\y) circle (0.05);

   % Loop to generate lattice points
   \foreach \i in {\minVal,...,\maxVal} {
     \foreach \j in {\minVal,...,\maxVal} {
       % Check if qi - j is divisible by n
       \pgfmathtruncatemacro{\result}{mod(\q*\i - \j, \n)}
       \ifnum\result=0
        % Draw lattice point
         \filldraw[orange] (\i,\j) circle (0.15);
        % Hexagon
        % \draw[orange] (\i-1,\j+1) -- (\i,\j+1);
        % \draw[orange] (\i,\j+1) -- (\i+1,\j);
        % \draw[orange] (\i+1,\j) -- (\i+1,\j-1);
        % \draw[orange] (\i+1,\j-1) -- (\i,\j-1);
        % \draw[orange] (\i,\j-1) -- (\i-1,\j);
        % \draw[orange] (\i-1,\j) -- (\i-1,\j+1);         
       \fi
    }
  }

\pgfmathtruncatemacro{\x}{0}  
\foreach \j in {0,...,\numexpr\n-1} {
   \foreach \i in {0,...,\numexpr\n}
   {
       \pgfmathtruncatemacro{\res}{mod(\q*\i - \j+\q, \n)}
       \ifnum \res= 0
           \xdef\x{\i}
           \breakforeach
         \fi
       
       \ifnum \j >0
            \ifnum \i > \numexpr\x-1
            \breakforeach
         \fi
         \fi

         % \fill[blue] (\i,\j) circle (0.2);
         \pgfmathtruncatemacro{\myv}{mod(\q*\i - \j,\n)}
         \ifnum \myv > 0
           \pgfmathtruncatemacro{\myv}{\n-\myv}
         \else
           \pgfmathtruncatemacro{\myv}{-\myv}
         \fi
         
         \ifnum \numexpr\i + \numexpr\j = 0
         \else  
         \fill[white] (\i,\j) circle (0.1);
         \fi
         \node at (\i,\j) {\myv};

       %\fi

  }
}

%\node at (\maxVal+3, \minVal) {r=\n, a=\q}; 
\end{tikzpicture}
\end{center}
\end{example}

%\begin{rmk}
%In subsequent sections, we will give a geometric proof that $R_{r,a}$ is isomorphic to the 
% Kalck-Karmazyn algebra defined as an endomorphism algebra of the Kawamata vector bundle.
%Then Corollary~\ref{asfgbzdfbzdh} will give an independent proof that 
%the Kalck--Karmazyn algebra satisfies relations found in \cite{KK17}.
%\end{rmk}

%
%
%
%
%

%
%
%

\begin{rmk}\label{adfvadfar}
The lattices $\Gamma$ of orange dots for singularities ${1\over r}(1,a)$ and  ${1\over r}(1,b)$, where, as above, $b$ is the inverse of $a$
modulo $r$, are clearly symmetric with respect to the diagonal $y=x$. It follows that the algebras $R_{r,a}$ and $R_{r,b}$ are opposite algebras. This was also observed in \cite[paragraph after Prop.~6.7]{KK18}.
\end{rmk}

In Section~\ref{sFBASBAHATH}, we study deformations of the Kalck--Karmazyn algebra $R_{r,a}$ by endowing the Kawamata Lagrangian $\bK_{r,a}$ with appropriate bounding cochains and computing the endomorphism algebra in the relative Fukaya category.

An obvious way to obtain a deformed object from $\bK_{r,a}$ is to view it as an object of the relative Fukaya category $\mathcal{F}(\mathbb{T}_1, {s})$, which deforms $\mathcal{F}(\mathbb{T}_2)$, where the black puncture  becomes a compactification divisor ${s}$. This relative category has objects represented by the same Lagrangians as in $\mathcal{F}(\mathbb{T}_2)$; however, the $A_\infty$ structure is deformed due to new contributions from holomorphic polygons passing through~${s}$.

Because of the existence of  bigons with boundary on $\bK_{r,a}$ passing through $s$, this naive attempt does not yield a deformation that keeps the rank of $\mathrm{End}(\bK_{r,a})$ constant. The idea is salvaged by equipping the Lagrangian
$\bK_{r,a}$ with a bounding cochain $\mathfrak{b} \in \mathrm{CF}^1(\bK_{r,a}, \bK_{r,a})$ in $\mathcal{F}(\mathbb{T}_1, {s})$. The coefficients of $\mathfrak{b}$ and the deformation parameter $s$ must satisfy a non-trivial equation to achieve a constant rank deformation. One of our main observations is that different choices of bounding cochains 
satisfying the constant rank condition
correspond to the irreducible components of the versal deformation space of the cyclic quotient singularity ${1\over r}(1, a)$.
This idea is made precise in Conjecture~\ref{svSfgasgasrg} below.

On the algebro-geometric side, we study deformations of the algebra $R_{r,a}$ to a flat $A$-algebra $\cR = \End(\cF)$, where $\cF$ is a deformation of the Kawamata vector bundle $F$ to a deformation $\cW$ of an algebraic surface $W$ over the base $\Spec A$.
To apply mirror symmetry for genus-one fibrations, we assume that the divisor $E \subset W$ deforms to a divisor $\cE \subset \cW$ such that a general fiber $\cE_t$ of the genus-one fibration $\cE$ has one node. Starting from the divisor $E \cong E_2$ on $W$ (see Figure~\ref{zxfbxzfbzfnzdf}), we obtain the divisor $\cE_t \cong E_1$ by smoothing the black node $P$, where $W$ is singular, while retaining the orange node $Q$, where $W$ is smooth.

We show that $\cR \cong \End(\cF|_\cE)$. This suggests the following strategy: first, compute deformations of the vector bundle $F_E$ to the total space $\cE$ of the fibration. The versal deformation space $\Def_{F_E/\cE}$ is smooth over $\Spec A$ with fibers  isomorphic to $\Ext^1(F_E, F_E)$.
In contrast to deformations of the Kawamata vector bundle on the algebraic surface, the endomorphism algebra of an arbitrary  deformation of the vector bundle $F_E$  is typically not $A$-flat, since a deformed vector bundle will typically
have an endomorphism algebra of smaller dimension than $R_{r,a} = \End(F_E)$.

\begin{definition}\label{sRBASRGARGHA}
Let $\cV$ be a universal vector bundle on $\cE\times_A\Def_{F_E/\cE}$.
We consider a closed subset $\Def^0_{F_E/\cE}\subset\Def_{F_E/\cE}$
(with a natural subscheme structure) such that $p\in \Def^0_{F_E/\cE}$ if and only if
$\dim\End(\cV_p)=\dim\End(F_E)$ (the maximal possible). The family of algebras  
$\End(\cV_p)$ is a flat family of finite-dimensional algebras over  $\Def^0_{F_E/\cE}$ providing a deformation
of the Kalck--Karmazyn algebra $R_{r,a}=\End(F_E)$.
\end{definition}

\begin{rmk}
One can formulate a more general problem, which can be investigated using similar methods: given a flat family of curves $\cE$ of arithmetic genus $1$ and a vector bundle $V$ on the special fiber, investigate the locus of deformations $\cV$ of $V$ to the total space $\cE$ such that $\End(\cV)$ provides a flat deformation of $\End(V)$.
\end{rmk}

A small nuisance is that an algebraic surface $\cE$ can have a singularity (of type~$A_\ell$) at the node $P$ of the special fiber $E$, which depends on the deformation of $W$. In~Section \ref{sFBASBAHATH}, we will explain how one can pass to the deformation $\sE$ of $E$ that is smooth at $P$. Ignoring this minor difference between fibrations $\cE$ and $\sE$, we have a factorization of the map of versal deformation spaces
$$\Def_{(E \subset W)}\to\Def^0_{F_E/\sE}\to\Def_{R_{r,a}}$$
that maps a deformation of an algebraic surface to the deformation $\cF|_\cE$ of the restriction of the Kawamata vector bundle $F|_E$,
which in turn is mapped to the deformation $\cR=\End(\cF|_\cE)$ of the Kalck--Karmazyn algebra $R_{r,a}$.
The~versal deformation space $\Def_{(E \subset W)}$ has several irreducible components
indexed by $P$-resolutions of singularity ${1\over r}(1,a)$ (Koll\'ar--Shepherd-Barron correspondence \cite{KSB}).
If $(\cE\subset\cW)$ is a general deformation
within a fixed irreducible component of $\Def_{(E \subset W)}$, then the general fiber $\cR_t$ of the family of algebras
$\cR$ is
a hereditary algebra by \cite{TU22}.
In particular, irreducible components of $\Def_{(E \subset W)}$ are mapped to uniquely defined components of
$\Def_{R_{r,a}}$. However, even in the simplest examples, $\Def_{R_{r,a}}$ has many other irreducible components.
In contrast, we believe that
\begin{conj}\label{svSfgasgasrg}
The map $\Def_{(E \subset W)}\to \Def^0_{F_E/\sE}$ is an isomorphism.
In particular, these deformation spaces have the equal number of irreducible components.
\end{conj}

We have verified this conjecture for $r \le 32$.
It shows that every deformation of the Kalck--Karmazyn algebra $R_{r,a}$
corresponds to some deformation of an algebraic surface~$W$ as long as the deformation of $R_{r,a}$ is
captured by a deformation of the vector bundle $F|_E$
to the genus $1$ fibration.
Thus, our dimension reduction from $W$ to $E$ does not lose any information.
We compute the subscheme $\Def^0_{F_E/\sE}$ and the flat family of finite-dimensional algebras
$\End(\cV_p)$ over it explicitly in Corollary~\ref{argaergqehqeth}.
As mentioned above, we use mirror symmetry for the family of genus one curves $\cE$
given by the relative Fukaya category $\mathcal{F}(\bT_1, \{s\})$, where
we re-interpret the black puncture as a divisor $\{s\} \subset \bT_1$ of the $1$-punctured torus.
%The~$A_\infty$-operations in the relative Fukaya category are given by counting holomorphic polygons with boundaries on Lagrangians, but the contribution of each polygon $u$ comes with a weight $s^{\mathrm{mult}(u, \{s\})}$.
%Mirror objects of deformations of the vector bundle $F|_E$ are provided by the Kawamata Lagrangian $\bK_{r,a}$
%endowed with bounding cochains $\mathfrak{b} \in CF^1(\bK_{r,a},\bK_{r,a})$ 
%of Fukaya--Oh--Ohta--Ono~\cite{FOOO}.
We~finish Section~\ref{sFBASBAHATH} with many explicit examples of the scheme $\Def^0_{F_E/\sE}$. 

Finally, in Section~\ref{adfbdfbsdns}, we compute the bounding cochain
for the Kawamata Lagrangian $\bK_{n^2,nq-1}$ of the Wahl singularity that corresponds to its $\Q$-Gorenstein smoothing.
This allows us to compute the Kawamata matrix order.
We give an explicit formula suitable for computer implementation. 

\begin{thm}\label{wFGSRGARGARE}
Let $\cR=\cR_{n^2,nq-1}$ be the matrix order that absorbs  the $\mathbb{Q}$-Gorenstein smoothing of 
a Wahl singularity ${1\over n^2}(1, na-1)$. The total space of this smoothing is a threefold terminal singularity 
${1\over n}(1,-1,a)$.
The algebra $\cR$ admits a $k[t]$-basis $\{w_i\}_{i\in\bZ_{n^2}}$ and
an embedding $\cR\hookrightarrow \Mat_n(k[t])$, which maps an element
$\sum\limits_{k\in\bZ_{n^2}} a_k w_k\in\cR$ 
(here the coefficients $a_i$, $i\in \bZ_{n^2}$, are in $k[t]$)
%($i\in\bZ_{n^2}$)
 to a matrix $A\in\Mat_n(k[t])$ as follows. % with the following matrix entries $A_{ij}$.
If $i<j$ then
$$A_{ij}=\!\!\!\sum_{\scriptsize \begin{matrix} r=0,\ldots,n-1\cr rn\le [inq]\cr 
 [inq]+i< rn+\min\limits_{k=i+1,\ldots,j}([knq]+k)\end{matrix}} 
 \!\!\!\!\!\! t^ra_{j-i+rn}+\qquad\qquad\qquad\qquad\qquad\qquad{}$$
 $${}\qquad\qquad\qquad\qquad\qquad\qquad\sum_{\scriptsize \begin{matrix} r=1,\ldots,n\cr rn> [inq]\cr 
 [inq]+i> rn-n^2+\max\limits_{k=1,\ldots,i-1}([knq]+k)\cr
 [inq]+i> (r-1)n-n^2+\max\limits_{k=j,\ldots,n}([knq]+k)\cr
 \end{matrix}} 
 \!\!\!\!\!\! t^{r-1}a_{j-i+rn-n}
 .$$
If $i>j$ then
$$A_{ij}=-\!\!\!\sum_{\scriptsize \begin{matrix} r=1,\ldots,n\cr rn> [inq]\cr 
 [inq]+i>rn-n^2+\max\limits_{k=j,\ldots,i-1}([knq]+k)\end{matrix}} t^ra_{j-i+rn}.$$
Finally, if $i=j$ then 
$$A_{ii}=a_0-\sum_{\scriptsize \begin{matrix} r=1,\ldots,n-1\cr rn> [inq]\end{matrix}} t^ra_{rn}.$$
%We use the following notation from Section~\ref{sFGASARHHETJ}:  for every $x\in\bZ_{n^2}$, we let $[x]\in\bZ$ be such that $0\le[x]<n^2$ and $x\equiv [x] \mod n^2$.
\end{thm}

The formula in Theorem~\ref{wFGSRGARGARE}  appears complicated, but we will show that it encodes  simple manipulations with rectangles in $\bZ^2$. Investigating them further  reveals interesting symmetries of the order, for example the following fact.

\begin{prop}\label{dfbsbF}
The order over $\bA^1$ from  Theorem~\ref{wFGSRGARGARE} extends to the order over $\bP^1$
(with an underlying vector bundle $\cO_{\bP^1}\oplus\cO_{\bP^1}(-n)^{n^2-1}$.) such that the fiber of the order
over $\infty$ is also isomorphic to the Kalck--Karmazyn algebra $R_{n^2,nq-1}$.
\end{prop}

We write down Kawamata's order embedded into the matrix algebra for small values of $n$ and $q$. For $n=2$ and $q=1$, see \eqref{sfvfbdfad}.

\begin{example}[$n=3$, $q=1$]
\[\left[\small\begin{matrix}
-t^{2} a_{6}+a_{0}&
       t a_{4}+a_{1}&
       t^{2} a_{8}+t a_{5}\\
       {-t^{3} a_{8}}&
       a_0&
       t^{2} a_{7}\\
       {-t a_{1}}&
       {-t a_{2}}&
       -t^{2} a_{6}-t a_{3}+a_{0}\\
\end{matrix}\right]\]
\end{example}

\begin{example}[$n=3$, $q=2$]
\[\left[\small\begin{matrix}
a_0&
       t^{2} a_{7}+t a_{4}&
       t^{2} a_{8}\\
       {-t^{2} a_{5}}&
       -t^{2} a_{6}+a_{0}&
       t a_{4}\\
       {-t a_{1}}&
       -t^{2} a_{5}-t a_{2}&
       -t^{2} a_{6}-t a_{3}+a_{0}\\\end{matrix}\right]\]
\end{example}

\begin{example}[$n=4$, $q=1$]
\[\left[\scriptsize\begin{matrix}
-t^{3} a_{12}-t^{2} a_{8}+a_{0}&
       t a_{5}+a_{1}&
       t a_{6}+a_{2}&
       t^{3} a_{15}+t^{2} a_{11}+t a_{7}\\
       -t^{4} a_{15}-t^{3} a_{11}&
       -t^{3} a_{12}+a_{0}&
       t^{2} a_{9}+t a_{5}+a_{1}&
       t^{3} a_{14}+t^{2} a_{10}\\
       {-t^{4} a_{14}}&
       {-t^{4} a_{15}}&
       a_0&
       t^{3} a_{13}\\
       {-t a_{1}}&
       {-t a_{2}}&
       {-t a_{3}}&
       -t^{3} a_{12}-t^{2} a_{8}-t a_{4}+a_{0}\\
\end{matrix}\right]\]
\end{example}

\begin{example}[$n=4$, $q=3$]
\[\left[\scriptsize\begin{matrix}
a_0&
       t^{3} a_{13}+t^{2} a_{9}+t a_{5}&
       t^{3} a_{14}+t^{2} a_{10}&
       t^{3} a_{15}\\
       {-t^{3} a_{11}}&
       -t^{3} a_{12}+a_{0}&
       t^{2} a_{9}+t a_{5}&
       t^{2} a_{10}\\
       {-t^{2} a_{6}}&
       -t^{3} a_{11}-t^{2} a_{7}&
       -t^{3} a_{12}-t^{2} a_{8}+a_{0}&
       t a_{5}\\
       {-t a_{1}}&
       -t^{2} a_{6}-t a_{2}&
       -t^{3} a_{11}-t^{2} a_{7}-t a_{3}&
       -t^{3} a_{12}-t^{2} a_{8}-t a_{4}+a_{0}\\
\end{matrix}\right]\]
\end{example}

\begin{example}[$n=5$, $q=1$]$\ $

\noindent
\scalebox{0.77}{\parbox{\linewidth}{%
\[\left[\tiny\begin{matrix}
-t^{4} a_{20}-t^{3} a_{15}-t^{2} a_{10}+a_{0}&
       t a_{6}+a_{1}&
       t a_{7}+a_{2}&
       t a_{8}+a_{3}&
       t^{4} a_{24}+t^{3} a_{19}+t^{2} a_{14}+t a_{9}\\
       -t^{5} a_{24}-t^{4} a_{19}-t^{3} a_{14}&
       -t^{4} a_{20}-t^{3} a_{15}+a_{0}&
       t^{2} a_{11}+t a_{6}+a_{1}&
       t^{2} a_{12}+t a_{7}+a_{2}&
       t^{4} a_{23}+t^{3} a_{18}+t^{2} a_{13}\\
       -t^{5} a_{23}-t^{4} a_{18}&
       -t^{5} a_{24}-t^{4} a_{19}&
       -t^{4} a_{20}+a_{0}&
       t^{3} a_{16}+t^{2} a_{11}+t a_{6}+a_{1}&
       t^{4} a_{22}+t^{3} a_{17}\\
       {-t^{5} a_{22}}&
       {-t^{5} a_{23}}&
       {-t^{5} a_{24}}&
       a_0&
       t^{4} a_{21}\\
       {-t a_{1}}&
       {-t a_{2}}&
       {-t a_{3}}&
       {-t a_{4}}&
       -t^{4} a_{20}-t^{3} a_{15}-t^{2} a_{10}-t a_{5}+a_{0}\\
\end{matrix}\right]\]
}}
\end{example}

\begin{example}[$n=5$, $q=2$]$\ $

\noindent
\scalebox{0.8}{\parbox{\linewidth}{%
\[\left[\tiny\begin{matrix}
-t^{4} a_{20}-t^{3} a_{15}+a_{0}&
       t^{2} a_{11}+t a_{6}+a_{1}&
       t^{3} a_{17}+t^{2} a_{12}+t a_{7}&
       t^{3} a_{18}+t^{2} a_{13}+t a_{8}&
       t^{4} a_{24}+t^{3} a_{19}+t^{2} a_{14}\\
       {-t^{5} a_{24}}&
       a_0&
       t^{4} a_{21}+t^{3} a_{16}&
       t^{4} a_{22}+t^{3} a_{17}&
       t^{4} a_{23}\\
       {-t^{2} a_{8}}&
       {-t^{2} a_{9}}&
       -t^{4} a_{20}-t^{3} a_{15}-t^{2} a_{10}+a_{0}&
       t a_{6}+a_{1}&
       t a_{7}\\
       {-t^{4} a_{17}}&
       {-t^{4} a_{18}}&
       -t^{5} a_{24}-t^{4} a_{19}&
       -t^{4} a_{20}+a_{0}&
       t^{3} a_{16}\\
       {-t a_{1}}&
       {-t a_{2}}&
       -t^{2} a_{8}-t a_{3}&
       -t^{2} a_{9}-t a_{4}&
       -t^{4} a_{20}-t^{3} a_{15}-t^{2} a_{10}-t a_{5}+a_{0}\\\end{matrix}\right]\]
}}
\end{example}

\begin{example}[$n=5$, $q=3$]$\ $

\noindent
\scalebox{0.84}{\parbox{\linewidth}{%
\[\left[\tiny\begin{matrix}
-t^{4} a_{20}+a_{0}&
       t^{3} a_{16}+t^{2} a_{11}&
       t^{3} a_{17}+t^{2} a_{12}&
       t^{4} a_{23}+t^{3} a_{18}+t^{2} a_{13}&
       t^{4} a_{24}+t^{3} a_{19}\\
       -t^{3} a_{14}-t^{2} a_{9}&
       -t^{4} a_{20}-t^{3} a_{15}-t^{2} a_{10}+a_{0}&
       t a_{6}+a_{1}&
       t^{2} a_{12}+t a_{7}+a_{2}&
       t^{2} a_{13}+t a_{8}\\
       {-t^{5} a_{23}}&
       {-t^{5} a_{24}}&
       a_0&
       t^{4} a_{21}+t^{3} a_{16}+t^{2} a_{11}&
       t^{4} a_{22}\\
       {-t^{3} a_{12}}&
       {-t^{3} a_{13}}&
       {-t^{3} a_{14}}&
       -t^{4} a_{20}-t^{3} a_{15}+a_{0}&
       t^{2} a_{11}\\
       {-t a_{1}}&
       {-t a_{2}}&
       {-t a_{3}}&
       -t^{3} a_{14}-t^{2} a_{9}-t a_{4}&
       -t^{4} a_{20}-t^{3} a_{15}-t^{2} a_{10}-t a_{5}+a_{0}\\\end{matrix}\right]\]
}}
\end{example}

\begin{example}[$n=5$, $q=4$]$\ $

\noindent
\scalebox{0.79}{\parbox{\linewidth}{%
\[\left[\tiny\begin{matrix}
a_0&
       t^{4} a_{21}+t^{3} a_{16}+t^{2} a_{11}+t a_{6}&
       t^{4} a_{22}+t^{3} a_{17}+t^{2} a_{12}&
       t^{4} a_{23}+t^{3} a_{18}&
       t^{4} a_{24}\\
       {-t^{4} a_{19}}&
       -t^{4} a_{20}+a_{0}&
       t^{3} a_{16}+t^{2} a_{11}+t a_{6}&
       t^{3} a_{17}+t^{2} a_{12}&
       t^{3} a_{18}\\
       {-t^{3} a_{13}}&
       -t^{4} a_{19}-t^{3} a_{14}&
       -t^{4} a_{20}-t^{3} a_{15}+a_{0}&
       t^{2} a_{11}+t a_{6}&
       t^{2} a_{12}\\
       {-t^{2} a_{7}}&
       -t^{3} a_{13}-t^{2} a_{8}&
       -t^{4} a_{19}-t^{3} a_{14}-t^{2} a_{9}&
       -t^{4} a_{20}-t^{3} a_{15}-t^{2} a_{10}+a_{0}&
       t a_{6}\\
       {-t a_{1}}&
       -t^{2} a_{7}-t a_{2}&
       -t^{3} a_{13}-t^{2} a_{8}-t a_{3}&
       -t^{4} a_{19}-t^{3} a_{14}-t^{2} a_{9}-t a_{4}&
       -t^{4} a_{20}-t^{3} a_{15}-t^{2} a_{10}-t a_{5}+a_{0}\\
\end{matrix}\right]\]
}}
\end{example}

\subsection*{Acknowledgements}

We thank Martin Kalck, Sasha Kuznetsov, Daniil Mamaev, Evgeny Shinder,
and Giancarlo Urz\'ua
for useful discussions. 

The first author was supported by the EPSRC grant EP/W015889/1. 
The second author was supported by the NSF grant DMS-2401387.

\section{Kawamata Lagrangian}\label{sfbdfbdhzdtjts}

%We start by reviewing definition of the Kawamata vector bundle.

Fix coprime integers $0 < a < r$ and consider a cyclic quotient singularity ${1\over r}(1, a)$. We will compactify it by a projective algebraic surface $W$ with a unique singular point $P$ that contains an anticanonical divisor $E = A \cup B$, a curve of arithmetic genus $1$, such that $A \cong B \cong \bP^1$ intersect at $P$ as orbifold coordinate axes of $\bA^2 / \mu_r$ and at an additional smooth point $Q$ transversally (see Figure~\ref{zxfbxzfbzfnzdf}). There are other compactifications~\cite{TU22}, but $W$ is the most convenient one for explicit calculations.
See Remark~\ref{sfvsrgqr} below for an explicit construction of~$W$.

The projective algebraic surface $W$ carries a remarkable vector bundle $F$ defined as the maximal iterated extension of the ideal sheaf $\cO_{W}(-A)$ by itself (\cite{Kaw}). We call $F$ the Kawamata vector bundle. To wit, consider a sequence of sheaves $\{ F^i \}_{i \geq 0}$ defined iteratively as follows: First, let $F^0 = \cO_{W}(-A)$ and then construct non-trivial extensions
$0 \to \cO_{W}(-A) \to F^i \to F^{i-1} \to 0$
until we arrive at $F = F^m$ such that $\mathrm{Ext}^1(F^m, \cO_{W}(-A)) = 0$. Kawamata showed that if  a maximal iterated extension exists, then the resulting sheaf $F$ is the versal noncommutative deformation of $\cO_{W}(-A)$ in a certain sense, hence is unique. Existence of $F$ was proved in \cite{KKS}. % (based on a different construction of $F$ from \cite{KK17}). 
Furthermore, it is known that $F$ is locally free of rank $r$ \cite[Prop.~6.7]{KKS}. % and $\mathrm{End}(F)$ is the base of the noncommutative versal family.

\begin{definition}
The algebra $R_{r,a}=\mathrm{End}(F)$  is called the Kalck--Karmazyn algebra.
\end{definition}

One can reconstruct $R_{r,a}$ from the restriction of the vector bundle $F$ to $E$:

\begin{lem}\label{sGSGsgsr}
$R_{r,a}=\mathrm{End}(F)\cong\mathrm{End}(F|_E)$ via the restriction $\alpha\mapsto \alpha|_E$.
\end{lem}

\begin{proof}
Equivalently,  we claim that 
$H^0(W, F^* \otimes F) \cong H^0(E, F^* \otimes F|_E)$ via the restriction.
Indeed, this follows from the short exact sequence 
$$0 \to F^* \otimes F(-E) \to F^* \otimes F \to F^* \otimes F|_E \to 0$$ by applying the long exact sequence of cohomology and using an isomorphism $F^* \otimes F(-E) \cong F^* \otimes F \otimes \omega_W$, Serre duality for $F^* \otimes F \otimes \omega_W$, and the formula $\Ext^k(F, F) = 0$ for $k = 1, 2$, which was proved in \cite{KKS}.
\end{proof}

By Lemma~\ref{sGSGsgsr} and homological mirror symmetry \eqref{hms2}, the Kalck-Karmazyn algebra $R_{r,a}$ is isomorphic to the endomorphism algebra $\End(\bK_{r,a})$ in the Fukaya category $\cF(\mathbb{T}_2)$.
Here, $\mathbb{T}_2$ is a symplectic torus with two punctures: a black puncture that corresponds to the singular point $P \in W$, and an orange puncture that corresponds to the singular point $Q$ of $E_2$. The latter is a smooth point of $W$ (see~ Figure~\ref{zxfbxzfbzfnzdf}).
The algebra $\End(\bK_{r,a})$ and its deformations will be studied in later sections.
The goal of this section is to compute the Lagrangian $\bK_{r,a}$ explicitly.

\begin{notation}
Let $b$ be the inverse of $a$ modulo $r$.
\end{notation}

\begin{thm}\label{argargerhwetb}
The Kawamata Lagrangian $\Bbb K_{r,a}$ is shown in  Figure~\ref{zfmbvasfhbgS} in two equivalent ways.
\begin{figure}[hbtp]
\begin{center}
\includegraphics[height=0.22\textheight]{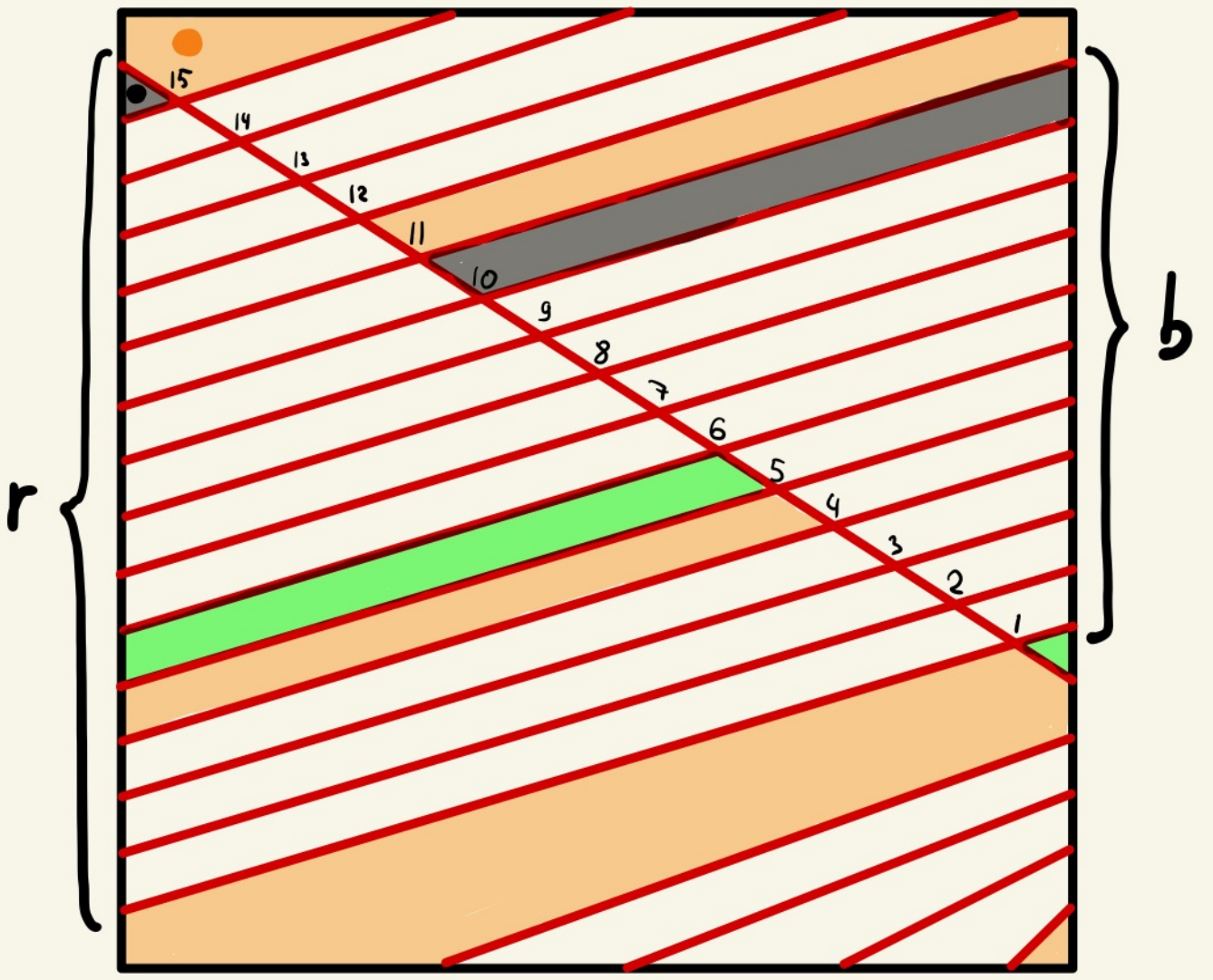}\ 
\includegraphics[height=0.22\textheight]{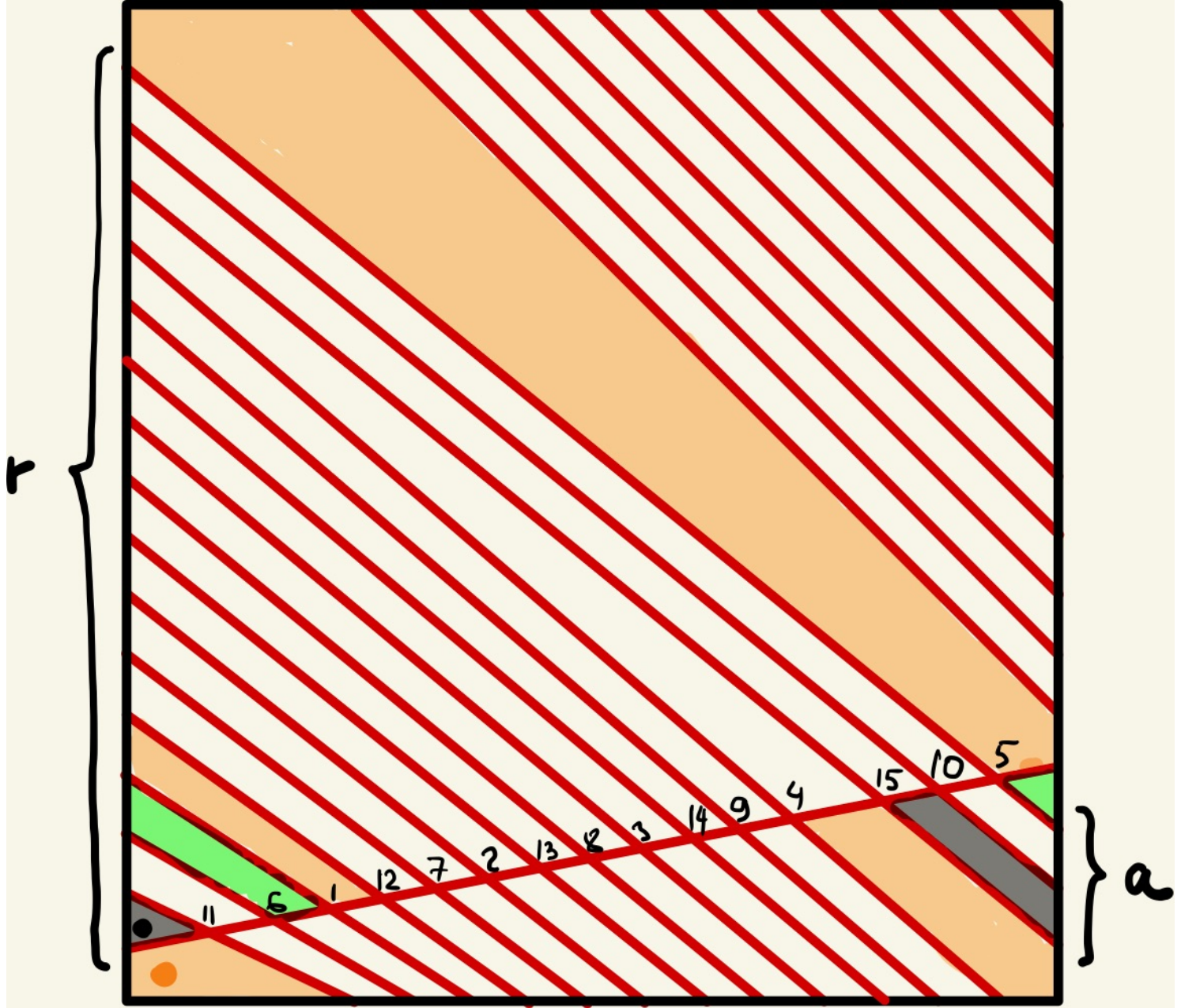}
\end{center}
\caption{Kawamata Lagrangian $\Bbb K_{r,a}$ (here $r=16$, $a=3$, $b=11$)}
\label{zfmbvasfhbgS}
\end{figure}
It has $r - 1$ self-intersection points, which we label by elements of $\Bbb Z_r\setminus\{0\}$.
 \end{thm}

%\begin{figure}[hbtp]
%\begin{center}
%\includegraphics[width=0.6\textwidth]{2_2.pdf}

The proof occupies the rest of this section. We will use a different construction of the Kawamata bundle $F$
from \cite{KK17}, which we are going to recall.
Consider the minimal resolution $W^{min}$ of the projective surface $W$.
The preimage of the singular point  is a chain of rational curves $\Gamma_1,\ldots,\Gamma_t\subset W^{min}$ 
with self-intersection numbers $-b_1,\ldots,-b_t$ such that $b_1,\ldots,b_t\ge 2$
(see Figure~\ref{svsgwrG}.)

\begin{figure}[htbp]
\begin{center}
\includegraphics[width=0.6\textwidth]{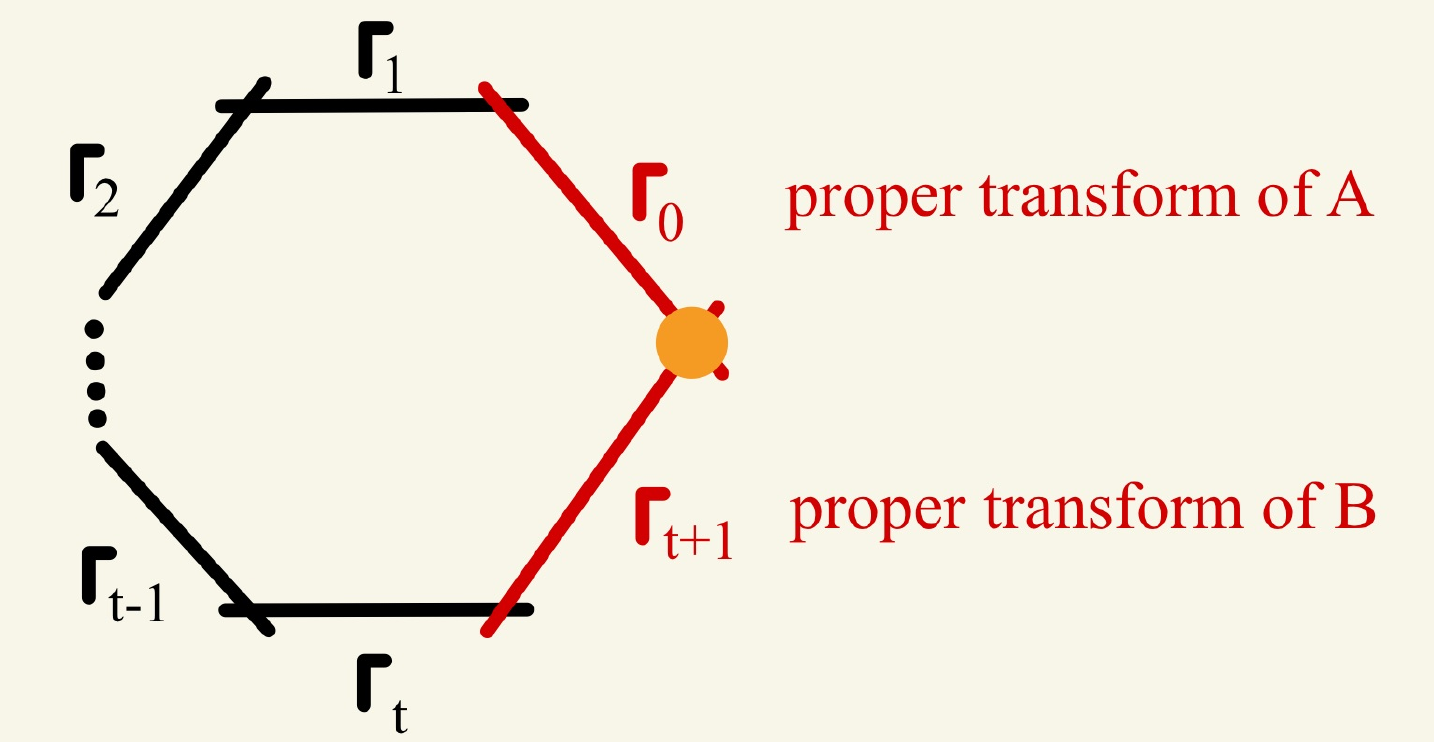}\quad
\end{center}
\caption{Minimal resolution $W^{min}$ of the projective surface $W$}\label{svsgwrG}
\end{figure}

\begin{rmk}\label{sfvsrgqr}
An explicit model of $W^{min}$ can be constructed as follows: start with a rational elliptic fibration with a $1$-nodal fiber,
blow up the node $t+1$ times to create a cycle of $t+2$ projective lines, then blow-up disjoint smooth points on $t$
of the irreducible components to create a chain $\Gamma_1,\ldots,\Gamma_t$ as above.
Contracting the chain gives a projective surface $W$ that satisfies  Assumptions~1.10 of \cite{TU22}.
\end{rmk}

\begin{definition}
There is an exceptional collection on $W^{min}$ of the line bundles
\begin{equation}\label{abadfhadhdt}
L_t=\cO(-\Gamma_0-\Gamma_1-\ldots-\Gamma_t),\ \ldots,\ 
L_1=\cO(-\Gamma_0-\Gamma_1), \ L_0=\cO(-\Gamma_0).
\end{equation}
\end{definition}

\begin{lem}\label{adfvadfbaba}
Let $E_{t+2}=\Gamma_0\cup\ldots\cup\Gamma_{t+1}\subset W^{min}$.
Let  $L$ be a line bundle  on $E_{t+2}$ of multi-degree 
$(1-\xi,0,0,\ldots,0,1)$, where $\xi=\Gamma_0^2$.
Under  mirror symmetry, the mirror Lagrangians 
$\bL_0,\ldots,\bL_{t}\in\mathcal{F}(\mathbb{T}_{t+2})$
of the line bundles $L\otimes L_0|_{E_{t+2}},\ldots,L\otimes L_{t}|_{E_{t+2}}$ are 
illustrated in  Figure~\ref{asr.kjfwkrh} (for $t=3$).
Concretely, each Lagrangian $\bL_i$ winds $b_j-2$ (resp., $b_j-1$) times between the punctures $P_{j-1}$ and $P_{j}$ for $j<i$ (resp.,~$j=i$).
We endow these Lagrangians with bounding spin structure, trivial local system and standard grading.
\begin{figure}[hbtp]
\begin{center}
\includegraphics[width=\textwidth]{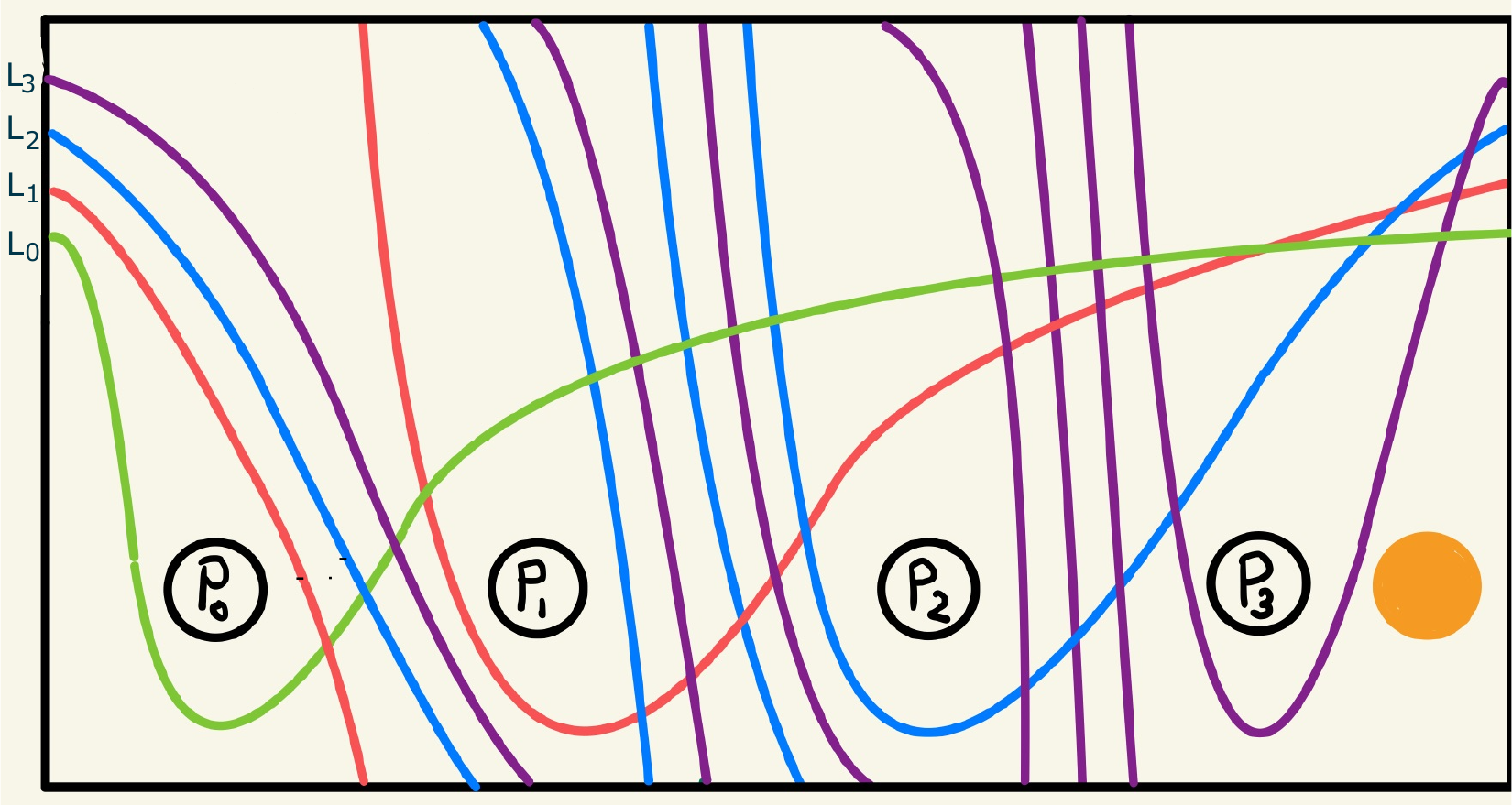}
\end{center}
\caption{Lagrangians $\bL_0,\ldots,\bL_{t+1}\in\mathcal{F}(\mathbb{T}_{t+2})$.}\label{asr.kjfwkrh}
\end{figure}
\end{lem}

\begin{proof}
%We first choose an arbitrary line bundle $L$ of the multi-degree $(1-\xi,0,0,\ldots,0,1)$.
The line bundles 
 $L\otimes L_0|_{E_{t+2}},\ldots,L\otimes L_{t}|_{E_{t+2}}$
have the following degrees on irreducible components $\Gamma_0,\ldots,\Gamma_{t+1}$:
$$(L\otimes L_i)\cdot\Gamma_j=\begin{cases}
0&\IF  j\ge i+2\\
-1&\IF j= i+1\\
b_j-1&\IF i= j\ge 1\\
b_j-2&\IF i>j\ge 1\\\
0&\IF i>j=0\\
-1&\IF i=j=0\\
\end{cases}\qquad\hbox{\rm for}\quad i=0,\ldots,t,\  j=0,\ldots,t+1.
$$
We can choose smooth points $q_i\in\Gamma_i$ for $i=0,\ldots,t+1$ 
so that 
\begin{equation}\label{wfvrfqrqr}
L\otimes L_{i}|_{E_{t+2}}=\cO(-q_{i+1}+\sum_{j\le i}a_{ij}q_j).
\end{equation}
for $i=0,\ldots,t$ and for some integer coefficients $a_{ij}$.

The mirror  of the curve $E_{t+2}$ of arithmetic genus $1$ is a torus $\mathbb{T}_{t+2}$ with $t+2$ punctures
illustrated in Figure \ref{figurehms} along with 
the mirror Lagrangians $[\mathcal{O}_{E_{t+2}}]$
and $[\mathcal{O}_{q_i}]$
corresponding to $\mathcal{O}_{E_{t+2}}$ and the skyscraper sheaves $\mathcal{O}_{q_i}$, where $q_i$ for $i=0,\ldots, t+1$ are choices of smooth points on each $\mathbb{P}^1$ component. We choose these points so that \eqref{wfvrfqrqr} holds.
%These sheaves generate the perfect derived category of $E_{r+2}$. Hence by knowing the images of these sheaves under homological mirror symmetry, we can determine the mirror of any vector bundle on $E_{r+2}$ using exact triangles. In particular, 
Line bundles on $E_{t+2}$ of the form $\cO(\sum a_iq_i)$ are obtained by twisting $\mathcal{O}_{E_{t+2}}$ by twist functors $T_{\mathcal{O}_{q_i}}$ which on the symplectic side are given by Dehn twists around the vertical Lagrangians  $[\mathcal{O}_{q_i}]$. 
For example, a line bundle $\mathcal{O}(q_{t+1})$, which restricts to $\mathcal{O}(1)$ in the $(t+1)$-component and  to $\mathcal{O}$ on the other components fits in to an exact sequence
$0 \to \mathcal{O} \to \mathcal{O}(q_{t+1}) \to \mathcal{O}_{q_{t+1}} \to 0$.\break
The corresponding Lagrangian $[\mathcal{O}(q_{t+1})]$
is obtained by twisting the Lagrangian $[\mathcal{O}]$ around the Lagrangian $[\mathcal{O}_{q_{t+1}}]$ by a right handed Dehn twist as shown on the right side of Figure \ref{figurehms}. 
We use formula \eqref{wfvrfqrqr} and apply
 Dehn twists repeatedly to construct Lagrangians
$\bL_0,\ldots,\bL_{t}$ of the lemma.
\end{proof}

\begin{figure}[htb!]
\centering
\begin{tikzpicture} [scale=0.95]
      \tikzset{->-/.style={decoration={ markings,
        mark=at position #1 with {\arrow[scale=2,>=stealth]{>}}},postaction={decorate}}}
  \draw (0,0) -- (5,0);
   \draw (0,0) -- (6,0);
         \draw[red, ->-=.5]  (0,0.2) -- (6,0.2);
         \draw [red, ->-=.5] (0,4) -- (0,0);
      %   \draw [red, ->-=.5] (1,6) -- (1,0);
         \draw [red, ->-=.5] (4,4) -- (4,0);
         \draw [red, ->-=.5] (5,4) -- (5,0);
         
         \draw (0,4) -- (6,4);
         \draw (6,0) -- (6,4);

\node at (-0.4,3.5)   {\footnotesize $[\mathcal{O}_{q_0}]$};
     %    \node at (0.65,0.5)   {\footnotesize $\mathcal{O}_{p_{n-1}}$};
         \node at (3.6,3.5)   {\footnotesize $[\mathcal{O}_{q_t}]$};
         \node at (4.5,3.5)   {\footnotesize $[\mathcal{O}_{q_{t+1}}]$};
         \node at (3,0.6)   {\footnotesize $[\mathcal{O}_{E_{t+2}}]$};
         
\draw[red, thick, fill=red] (1.7,3) circle(.02);
\draw[red, thick, fill=red] (2.4,3) circle(.02);
\draw[red, thick, fill=red] (3.1,3) circle(.02);

\draw[black, thick, fill=black] (1.8,1.5) circle(.01);
\draw[black, thick, fill=black] (2.2,1.5) circle(.01);
\draw[black, thick, fill=black] (2.6,1.5) circle(.01);

\draw[thick] (0.5,1.5) circle(0.2cm);
%\draw[thick] (1.5,1.5) circle(0.2cm);
\draw[thick] (3.5,1.5) circle(0.2cm);
\draw[thick] (4.5,1.5) circle(0.2cm);
\draw[thick] (5.5,1.5) circle(0.2cm);

%\node at (0.5,0.7) {\footnotesize $z_n$};
%\node at (1.5,0.7) {\footnotesize $z_{n-1}$};
%\node at (3.5,0.7) {\footnotesize $z_{3}$};
%\node at (4.5,0.7) {\footnotesize $z_2$};
%\node at (5.5,0.7) {\footnotesize $z_1$};

\end{tikzpicture}\quad
\begin{tikzpicture} [scale=0.95]
      \tikzset{->-/.style={decoration={ markings,
        mark=at position #1 with {\arrow[scale=2,>=stealth]{>}}},postaction={decorate}}}
        \draw (0,0) -- (6,0);
        \draw (0,0) -- (0,4);
        
         \draw[red]  (0,0.2) -- (4.75,0.2);
         \draw [red] (5.2,0.2) -- (6,0.2);
         \draw [red] (4.75,0.2) -- (4.8,0);
      %   \draw [red, ->-=.5] (1,6) -- (1,0);
         \draw [red, ->-=.5] (4.8,4) -- (5.2,0.2);
                 
         \draw (0,4) -- (6,4);
         \draw (6,0) -- (6,4);
       
\node at (4.1,3.5)   {\footnotesize $[\mathcal{O}({q_{t+1}})]$};

\draw[black, thick, fill=black] (1.8,1.5) circle(.01);
\draw[black, thick, fill=black] (2.2,1.5) circle(.01);
\draw[black, thick, fill=black] (2.6,1.5) circle(.01);

\draw[thick] (0.5,1.5) circle(0.2cm);
%\draw[thick] (1.5,1.5) circle(0.2cm);
\draw[thick] (3.5,1.5) circle(0.2cm);
\draw[thick] (4.5,1.5) circle(0.2cm);
\draw[thick] (5.5,1.5) circle(0.2cm);

%\node at (0.5,0.7) {\footnotesize $z_n$};
%\node at (1.5,0.7) {\footnotesize $z_{n-1}$};
%\node at (3.5,0.7) {\footnotesize $z_{3}$};
%\node at (4.5,0.7) {\footnotesize $z_2$};
%\node at (5.5,0.7) {\footnotesize $z_1$};

\end{tikzpicture}

\caption{Punctured torus $\mathbb{T}_{t+2}$ with oriented Lagrangians corresponding to various perfect sheaves on $E_{t+2}$.}
\label{figurehms}
\end{figure}
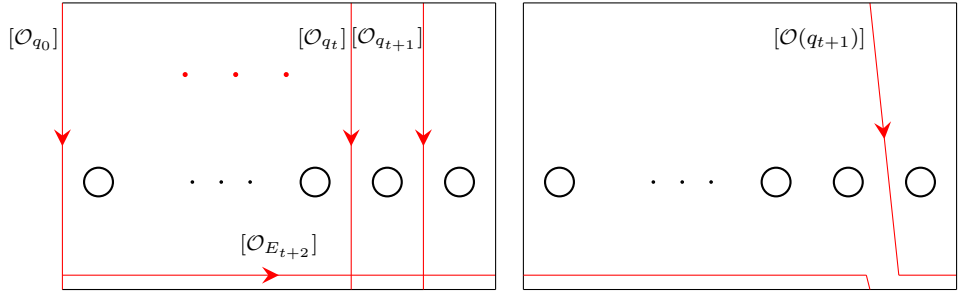

% Next, let's consider the Tate family of curves $T_2 \to \mathrm{Spec} \m
% athbb{Z}[[t_1,t_2]]$ whose special fiber is the N\'eron 2-gon $E$. Theorem \cite[Theorem A]{LPol} establishes a quasi-equivalence over $\mathbb{Z}[[t_1,t_2]]$, 
%\[ \mathcal{F}(T, \{s,o\} ) \simeq \mathrm{Perf}\, T_2 \]
%where the left-hand-side is the relative Fukaya category of the compact torus $\mathbb{T}$ relative to compactification divisor $D=\{s,o\}$ such that $\mathbb{T}\setminus D = \mathbb{T}_0$, where $t_1$ (resp. $t_2$) is a formal parameter keeping track of intersection number of holomorphic polygons with $s$ (resp. $o$). In this paper, our interest is in the subfamily where $t_2=0$. The techniques used in the proof of Theorem \cite{Theorem A}{LPol} apply directly in this case to give the following quasi-equivalence over $\mathbb{Z}[t]$
%\[ \mathcal{F}(\mathbb{\overline{T}}_0, \{s\}) \simeq \mathrm{Perf} \mathscr{E} \]
%where $\mathbb{\overline{T}}_0$ is once-punctured torus, with $D= \{s \}$ is a divisor with respect to which we study the relative Fukaya category and $\mathscr{E} \to \mathrm{Spec} \mathbb{Z}[t]$ is the family of nodal curves where the special fiber is the N\'eron 2-gon and general fiber is the nodal rational curve given by the following equation:
%\[ y^2 \]

\begin{proof}[Proof of Theorem~\ref{argargerhwetb}]
By  \cite{KK17} (based on results of \cite{HillePloog}), the Kawamata vector bundle on $W$ is isomorphic to the push-forward
of a certain vector bundle $F_0$ on $W^{min}$, which  
is the first term in the sequence of vector bundles 
$F_0,\ldots,F_t$ on $W^{min}$.\break 
Concretely,  $F_t=L_t$ and, for $i=0,\ldots,t-1$, the bundle $F_i$ is the universal extension of $F_{i+1}$ by $L_i$,
i.e.~we have a short exact sequence
\begin{equation}\label{sGASGASRGASRG}
0\to{\Ext}^1(F_{i+1},L_i)^*\otimes L_i\to F_i\to F_{i+1}\to0.
\end{equation}

\begin{lem}[\cite{KK17}]\label{sfvSFbafbarn}
We have 
$\dim{\Ext}^1(F_{i+1},L_i)=\rk F_i-\rk F_{i+1}$  and
$$
{\rk F_i\over\rk F_{i+1}}=b_{i+1}-{1\over b_{i+2}-{1\over\ldots-{1\over b_t}}},
$$ 
where $\rk F_i$ and $\rk F_{i+1}$ are coprime,
$-b_1,\ldots,-b_t$ are self-intersections of the exceptional 
curves in the minimal resolution $W^{min}$ of $W$,
and  
$$
{r\over r-a}=b_1-{1\over b_{2}-{1\over\ldots-{1\over b_t}}}.
$$ 
\end{lem}

Let $\bF_i$ be the Lagrangian in $\mathbb{T}_{t+2}$
corresponding to the vector bundle $F_i|_{E_{t+2}}\otimes L$ 
(here $L$ is a line bundle from  Lemma~\ref{adfvadfbaba}).
We investigate  the sequence of Lagrangians $\bF_0,\ldots,\bF_t$ inductively.
Since \eqref{abadfhadhdt} is an exceptional collection,
an argument similar to the proof of Lemma~\ref{sGSGsgsr} shows that 
$$\Ext\nolimits^j(F_{i+1},L_i)\cong\Ext\nolimits^j(F_{i+1}|_{E_{t+2}},L_i|_{E_{t+2}})$$ 
for every $j$. In particular, $F_i|_{E_{t+2}}$ is determined inductively by an exact sequence
\begin{equation}\label{ssdfasrgasrgsrg}
0\to\Ext\nolimits^1(F_{i+1}|_{E_{t+2}},L_i|_{E_{t+2}})^*\otimes L_i|_{E_{t+2}}\to F_i|_{E_{t+2}}\to F_{i+1}|_{E_{t+2}}\to0.
\end{equation}
Tensoring \eqref{ssdfasrgasrgsrg} with a line bundle $L$ of Lemma~\ref{adfvadfbaba}
and applying  mirror symmetry, shows that the Lagrangian $\bF_i$ is determined recursively by
the exact triangle
${\Ext}^1(\bF_{i+1},\bL_i)^*\otimes \bL_i\to \bF_i\to \bF_{i+1}$ in $\mathcal{F}(\mathbb{T}_{t+2})$.

\begin{lem}\label{argargarhaheht}
The Lagrangian $\bF_i$ can be constructed as follows:
repeat the Lagrangian $\bL_i$ as many times as the dimension of $\Ext^1(F_{i+1}, L_i)$ (computed in Lemma~\ref{sfvSFbafbarn}).
Then perform the surgery illustrated at the top of  Figure~\ref{dbarhathaetj}, 
\begin{figure}[htbp]
\begin{center}
\includegraphics[width=0.5\textwidth]{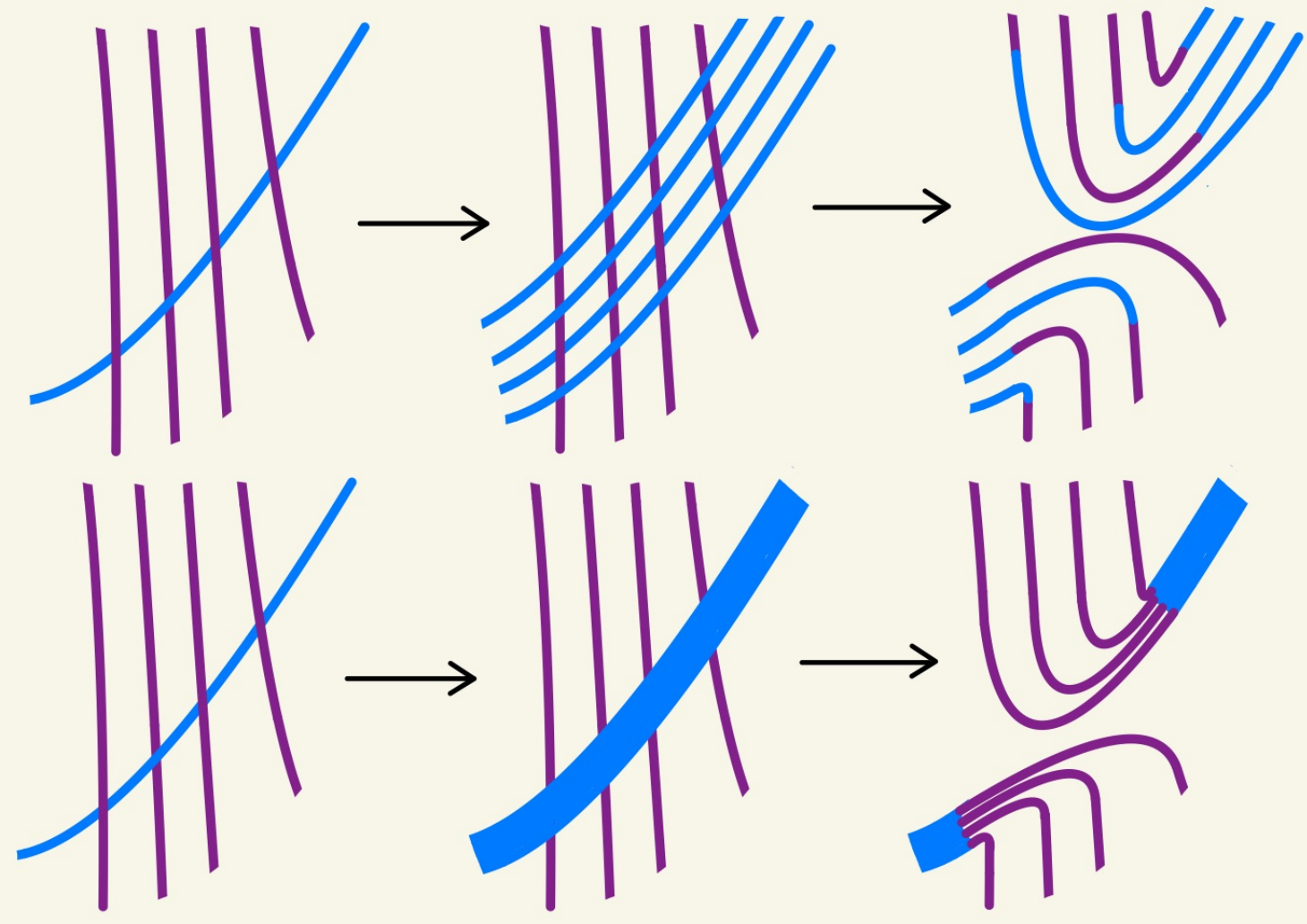}
\end{center}
\caption{Surgery on Lagrangians}\label{dbarhathaetj}
\end{figure}
where the Lagrangian $\bL_i$
is in blue and the Lagrangian  $\bF_{i+1}$ is in magenta.
By aesthetic reasons, instead of repeating the Lagrangian $\bL_i$, we  instead 
draw a thick ``band'' of curves
as at the bottom of  Figure~\ref{dbarhathaetj}.
\end{lem}

\begin{proof}
This follows from \cite[Lemma 5.4]{Abo}.
\end{proof}

To simplify the analysis, we now restrict to the case 
$t=3$ %as in Figure~\ref{asr.kjfwkrh} 
but the algorithm is the same for any $t$. 
We~start with $\bF_3=\bL_3$, draw $\rk\bF_2-\rk\bF_3=b_3-1$ copies of the ``blue'' Lagrangian 
$\bL_2$ and first construct and then simplify the Lagrangian $\bF_2$ 
using Lemma~\ref{argargarhaheht} as illustrated in Figure~\ref{sdvsVsgsrGB}.
\begin{figure}[hbtp]
\begin{center}
\includegraphics[width=0.45\textwidth]{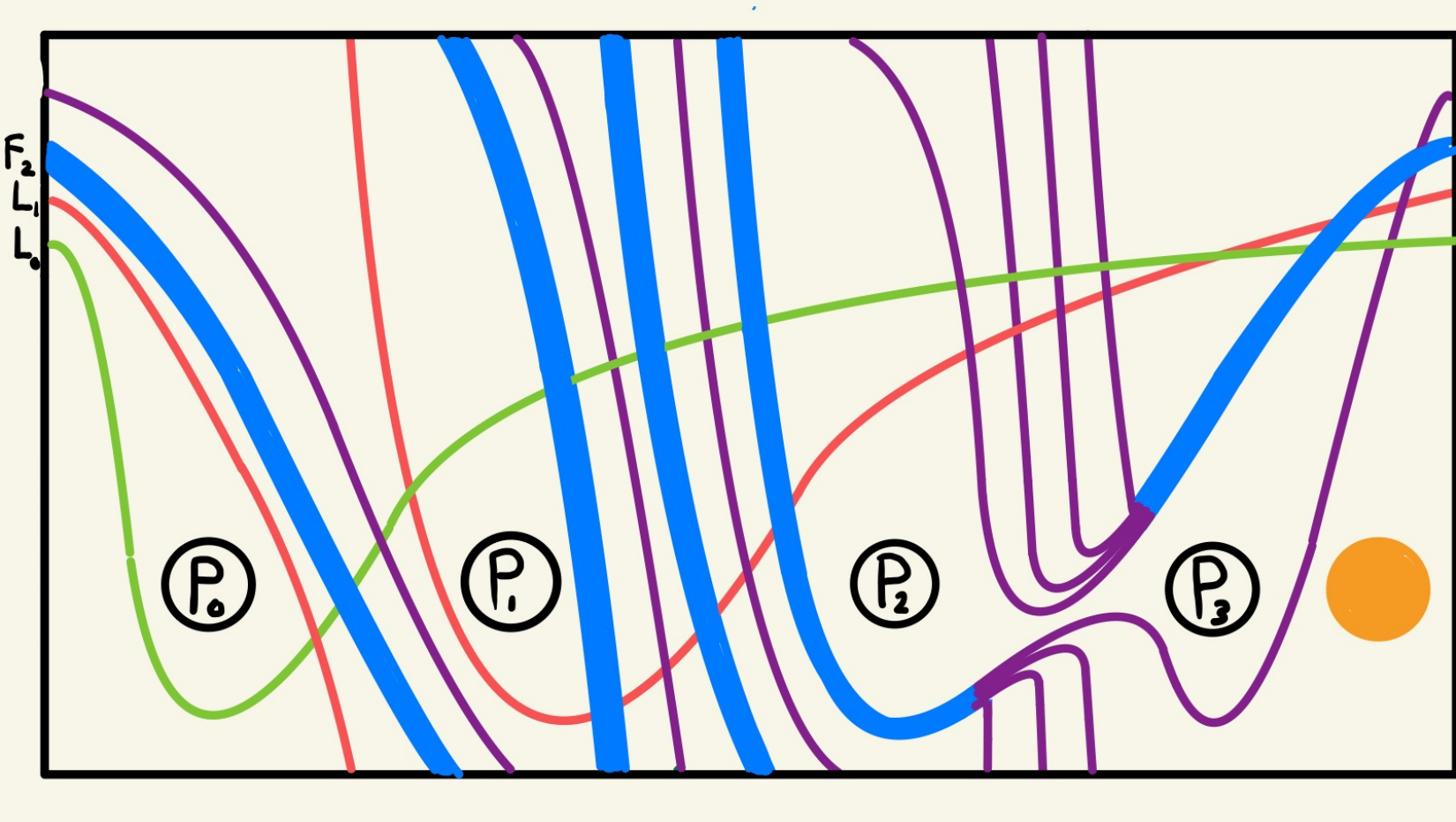}\qquad
\includegraphics[width=0.45\textwidth]{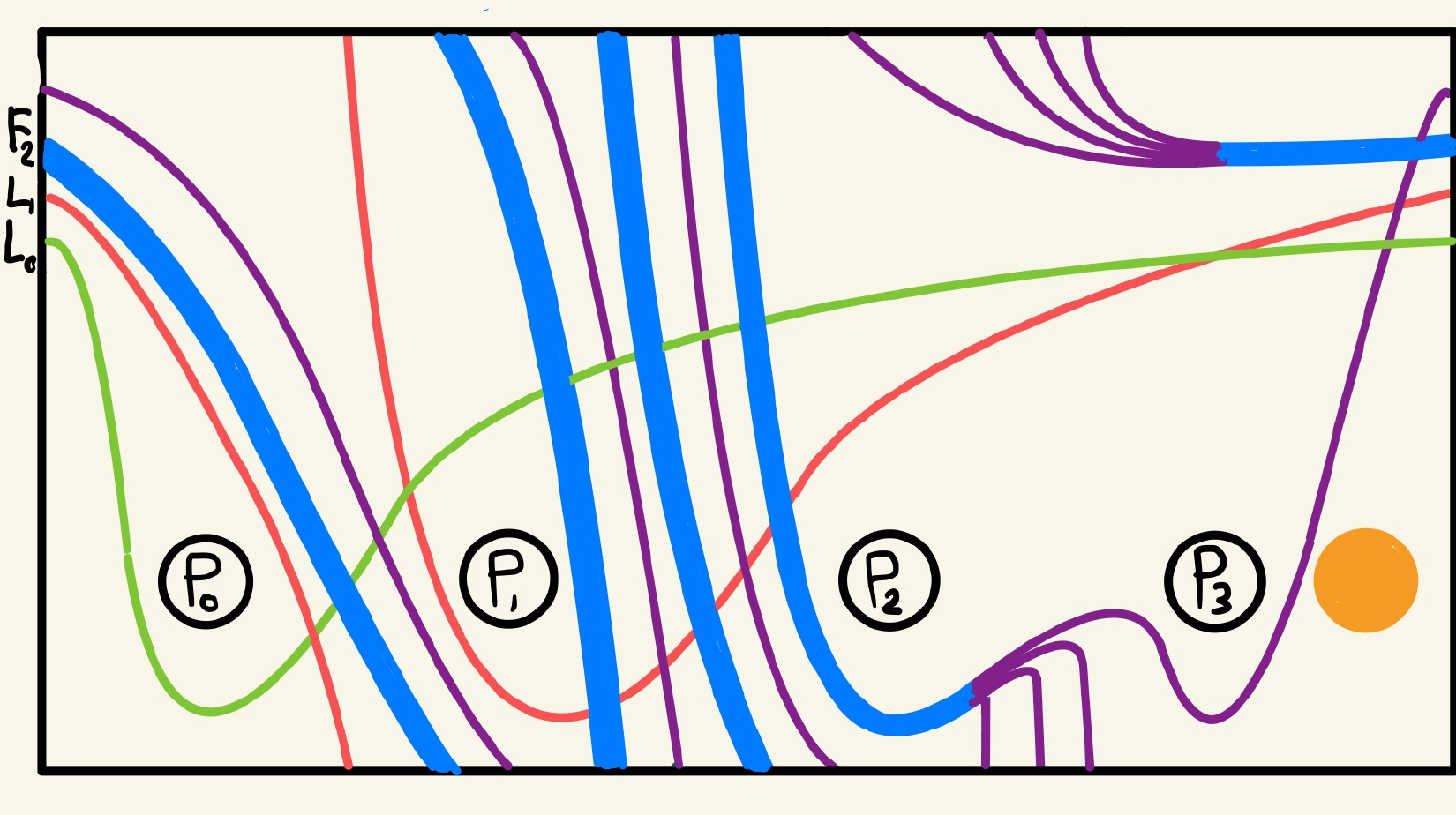}
\end{center}
\caption{The mirror Lagrangian $\bF_2$ of the vector bundle $F_2$.}\label{sdvsVsgsrGB}
\end{figure}
Next, we draw $\rk\bF_1-\rk\bF_2= b_2b_3-b_3-1$ copies of the ``red'' Lagrangian 
$\bL_1$ and 
construct (and simplify) the Lagrangian $\bF_1$ illustrated in Figure~\ref{egaerhaerh} (left side).
\begin{figure}[hbtp]
\begin{center}
\includegraphics[width=0.45\textwidth]{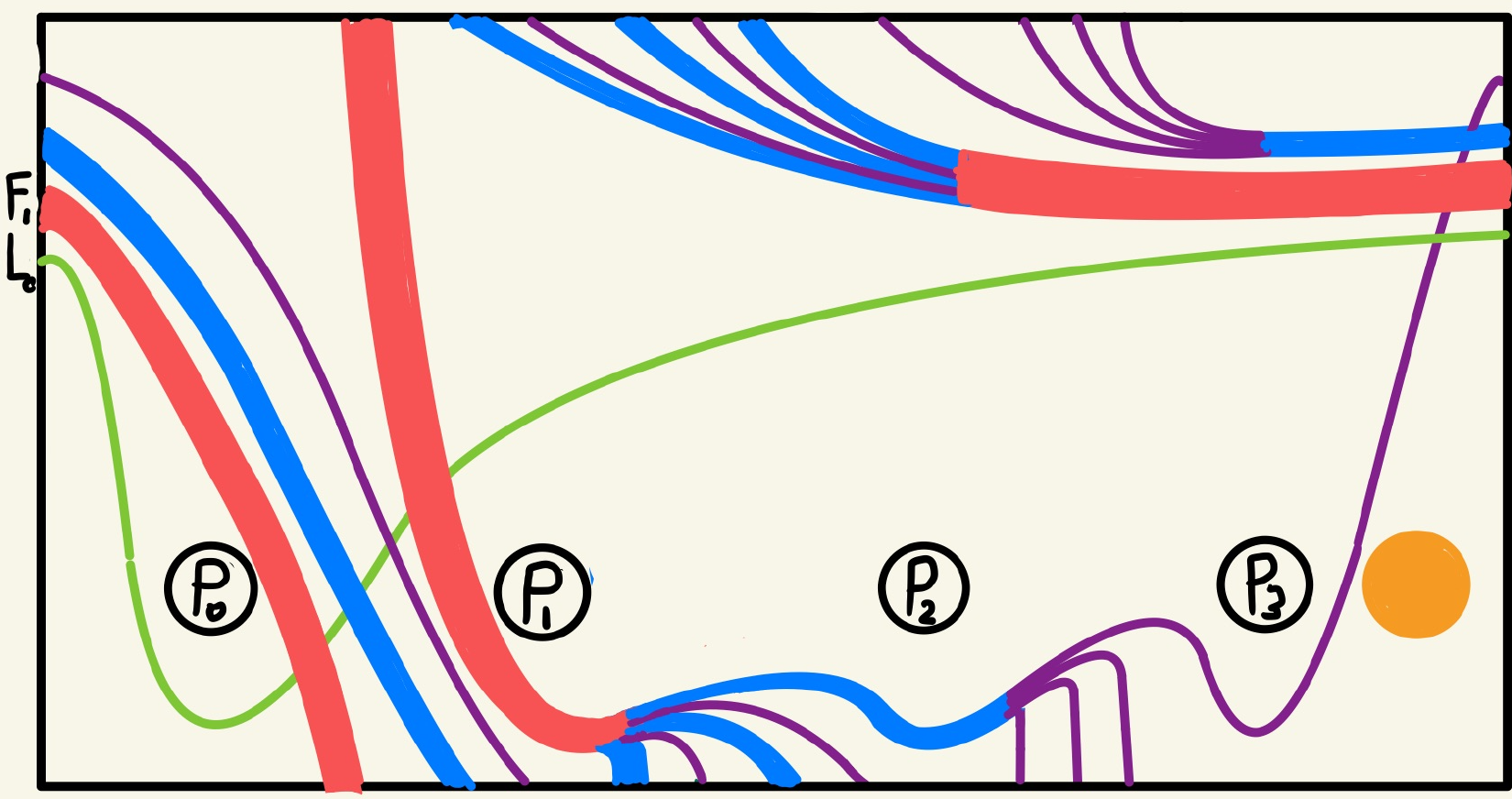}\quad\quad
\includegraphics[width=0.45\textwidth]{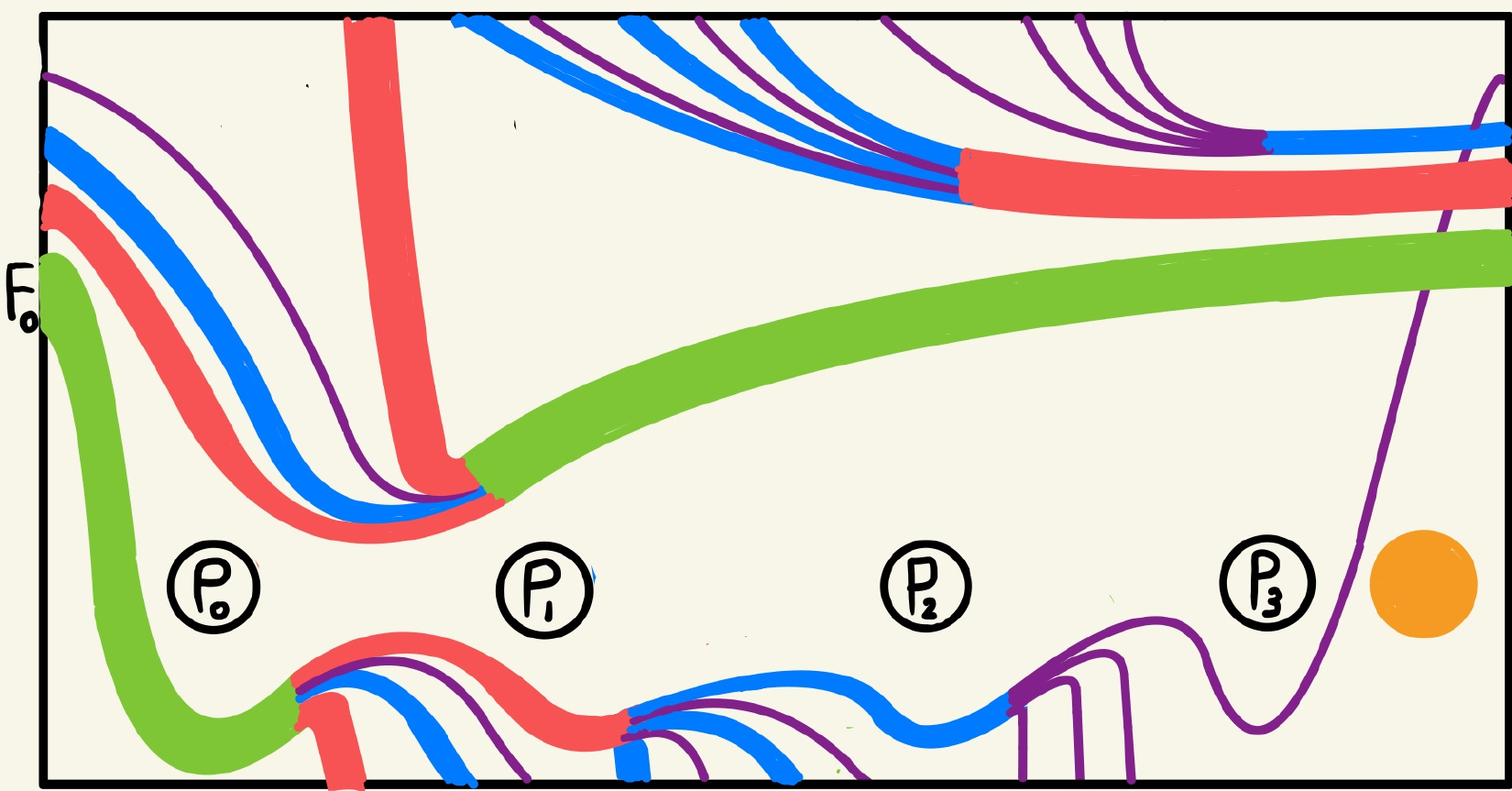}
\end{center}
\caption{The Lagrangians  $\bF_1$ (left) and $\bF_0$ (right.)}\label{egaerhaerh}
\end{figure}
Finally, we draw $\rk\bF_0-\rk\bF_1=b_1b_2b_3-b_2b_3-b_1-b_3+1$ copies of the ``green'' Lagrangian 
$\bL_0$ and  
construct (and simplify) the Lagrangian $\bF_0$ illustrated in Figure~\ref{egaerhaerh} (right side).

%\begin{figure}[hbtp]
%\begin{center}
%\includegraphics[width=0.4\textwidth]{2_11.pdf}
%\end{center}
%\caption{The Lagrangian  $\bF_0$ that corresponds to the universal extension $F_0$ of $F_1$ by $L_0$.}\label{sdgbsdthweth}
%\end{figure}

It remains to observe that the vector bundle $F_0|_{E_{t+2}}$ is a pull-back of the Kawamata vector bundle $F|_E$. Under mirror symmetry, this means that the Lagrangian $\bK_{r,a}$ on the two-punctured torus that corresponds to the Kawamata vector bundle $F|_E$ (tensored by $L$) is obtained from the Lagrangian $\bF_0$ by combining all non-orange punctures into one black puncture. Slightly simplifying Figure~\ref{egaerhaerh} shows that this Lagrangian $\bK_{r,a}$ is the same as the one from the right side of Figure~\ref{zfmbvasfhbgS}. 
\end{proof}

The following corollary is straightforward.

\begin{cor}
The complement of $\Bbb K_{r,a}$ in the torus is the union of rectangular (in other words, four sided)  regions, except for one hexagonal region (which contains an orange puncture) and two triangular regions (one of which contains a black puncture).\break
In Figure~\ref{asrgargqre} we draw the region formed by combining the hexagonal and triangular regions\footnote{
See Definition~\ref{sFBAFBADFNAD} for an explanation of the markings $q$ and $e$.}, the Gauss word
formed by following the Lagrangian until all
 self-intersection points are counted twice, and the universal cover of the torus with the preimage of $\bK_{r,a}$.
 \end{cor}

\begin{figure}[htbp]
\begin{center}
\includegraphics[height=0.26\textheight]{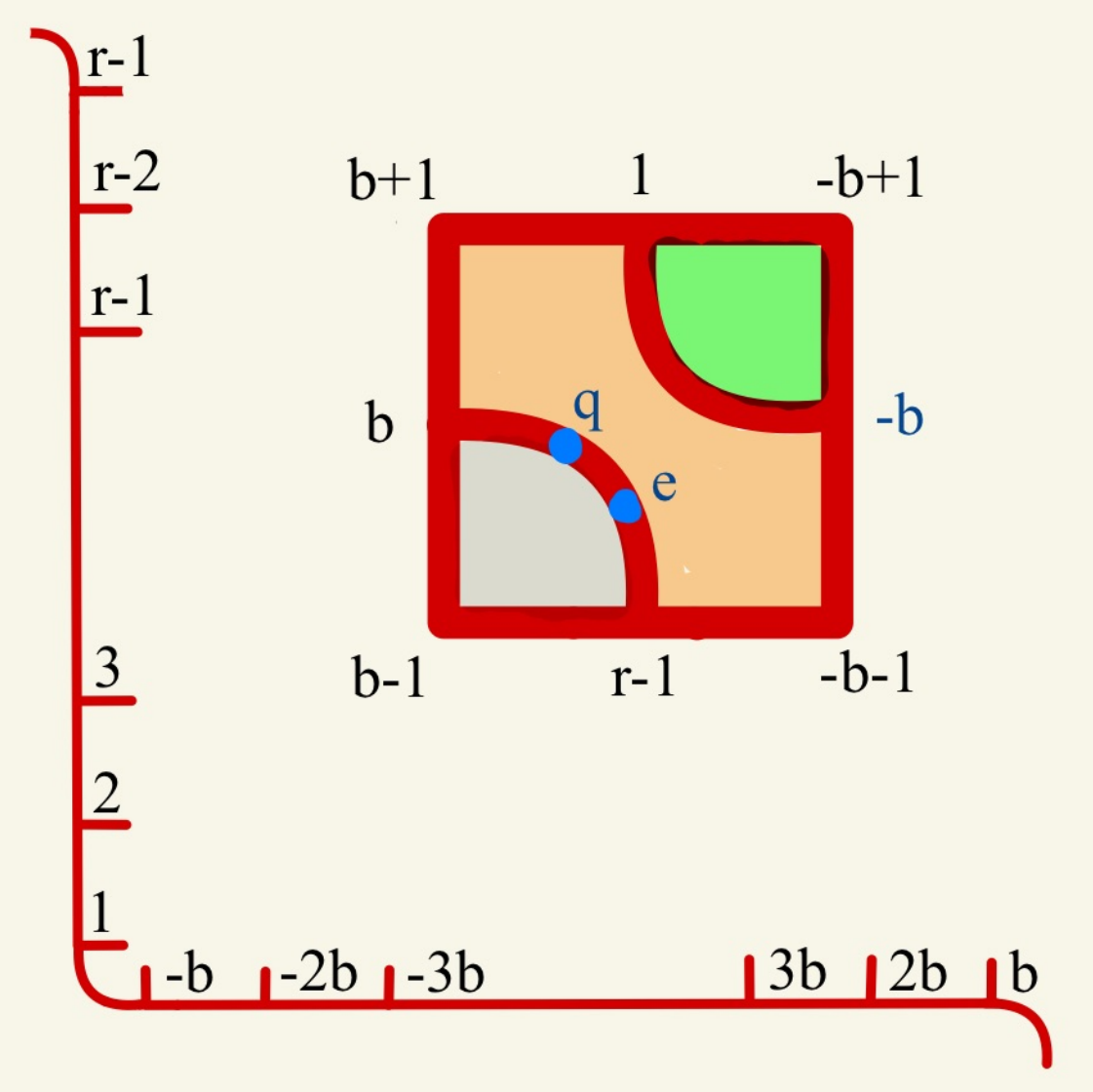}\qquad
\includegraphics[height=0.26\textheight]{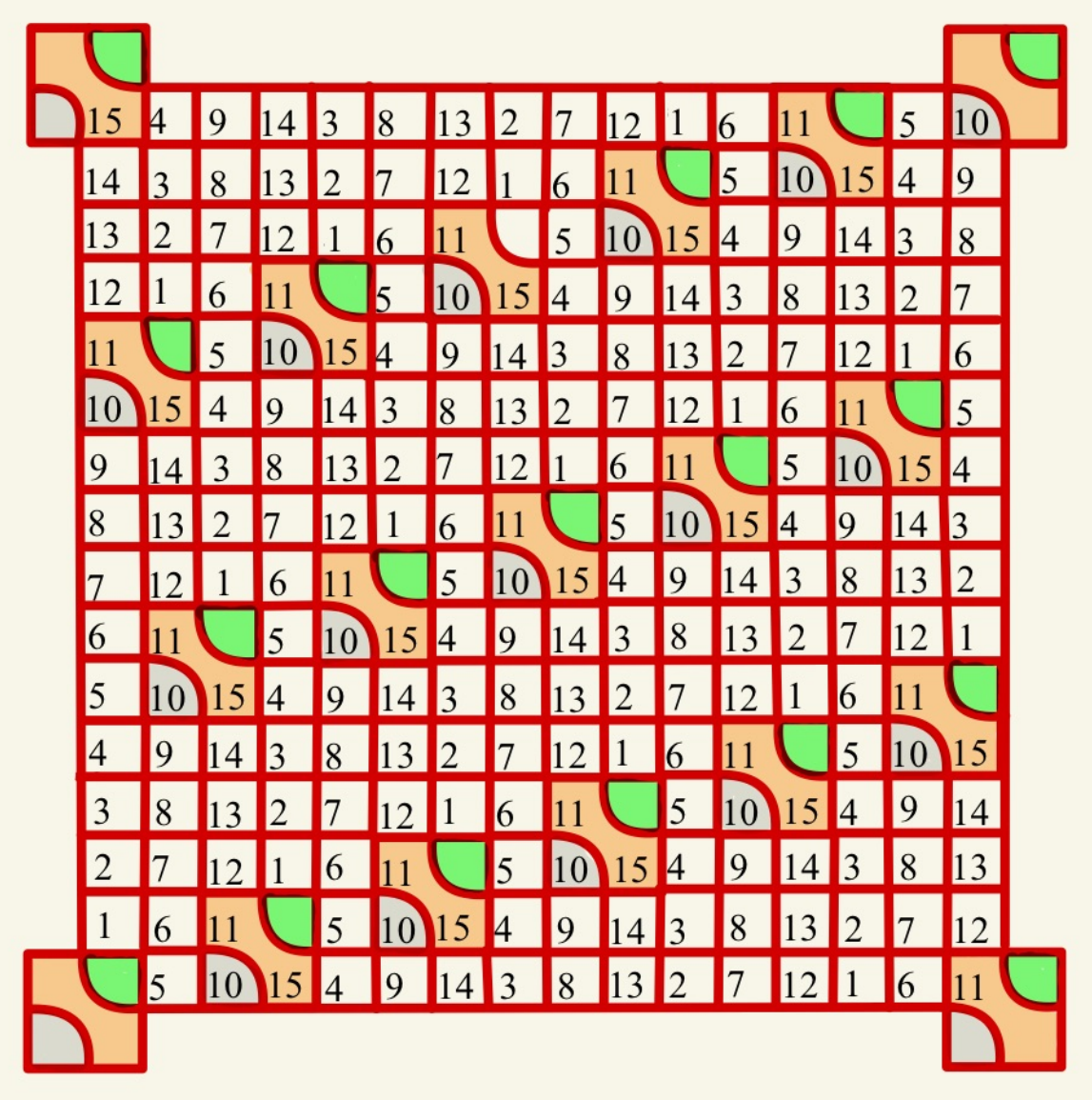}
\end{center}
\caption{The union of non-rectangular regions,
the Gauss word, and the universal cover  (for $r=16$, $a=3$, $b=11$.)}
\label{asrgargqre}
\end{figure}

\section{Deformations of  Lagrangians and their endomorphism rings}\label{sFBASBAHATH}

As in Section~\ref{sfbdfbdhzdtjts}, let $F_E$
be the restriction of the Kawamata 
vector bundle $F$ to the  divisor $E\cong E_2$ of the projective algebraic surface $W$ with 
a cyclic quotient singularity ${1\over r}(1,a)$.
In Theorem~\ref{zfmbvasfhbgS},
we found the mirror Lagrangian $\bK_{r,a}\in \cF(\bT_2)$ of~$F_E$.
We will study  the endomorphism algebra $R_{r,a}$ of $\bK_{r,a}$ and its deformations.

The motivation comes from the study of flat deformations $(\cE\subset\cW)$ of the pair $(E\subset W)$ over  a smooth curve germ $\Spec A$. Let $t\in A$ be a local parameter.\break
We~assume that a general fiber $\cW_t$ is a smooth projective surface but its anti-canonical divisor $\cE_t$ 
has one~node. From the  divisor $E\cong E_2$ of $W$ (see Figure~\ref{zxfbxzfbzfnzdf}), 
we obtain  the  divisor $\cE_t\cong E_1$ of $\cW_t$ by smoothening the black node $P\in E$ but retaining  the orange node~$Q\in E$. 

As~explained in \cite{TU22}, the Kawamata vector bundle $F$ on $W$ deforms  uniquely to a vector bundle $\cF$ on $\cW$ and  the Kalck--Karmazyn algebra $R_{r,a}=\End(F)$ deforms to a $A$-algebra $\cR=\End(\cF)$,
which is a free $A$-module of rank $r$. The $A$-algebra $\cR$ depends on the deformation of $(E\subset W)$.
Concretely, the~versal deformation space $\Def_{(E \subset W)}$ has several irreducible components, 
which are all smooth and classified in \cite{KSB} (Koll\'ar--Shepherd-Barron correspondence).\footnote{These results are usually stated for the versal deformation space $\Def_{W}$
of all deformations of the surface $W$, which generally do not induce a deformation of the anticanonical divisor $E \subset W$. But an extension of these results to deformations of the pair $(E \subset W)$ is  well-known, see e.g.~\cite[Lemma~3.2]{TU22}.}
If $(\cE\subset\cW)$ is a general deformation  
within a fixed irreducible component of $\Def_{(E \subset W)}$, then the general fiber $\cR_t$ of the family of algebras 
$\cR$ is
%gives a deformation of the Kalck--Karmazyn algebra to 
 a hereditary algebra by \cite{TU22}
 (equivalently, $\cR_t$ is Morita-equivalent
to a path algebra of a  quiver without relations). 

We would like to compute the $A$-algebra $\cR$ explicitly.
Our idea is to use the following formula, which can be proved in the same way as Lemma~\ref{sGSGsgsr}:
\begin{equation}
\cR=\End(\cF)\cong\End(\cF|_\cE).
\end{equation}
In this section we focus on computing the closed subscheme $\Def^0_{F_E/\cE}\subset\Def_{F_E/\cE}$
(see Definition~\ref{sRBASRGARGHA}) and the flat family of finite-dimensional algebras  $\End(\cV_p)$ 
over  it that provides a deformation of the Kalck--Karmazyn algebra $R_{r,a}=\End(F_E)$.

\begin{rmk}\label{sfasfbgadrh}
A minor nuisance is that an algebraic surface $\cE$ can have a singularity at 
the  node $P$ 
of the special fiber $E$
of type $A_{\ell-1}$,  $\ell\ge 1$.
We have a finite base change cartesian diagram 

\begin{equation}\begin{CD}
 \cE @>>> \sE \\
@VVV    @VVV\\
\Spec A @>>> \Spec B\\
\end{CD}
\end{equation}
where $B\subset A$ is a subring with local parameter $s=t^\ell$
and $\sE$ is a versal deformation of $P\in E$ (equisingular at $Q\in E$.)
The total space of $\sE$ is smooth at $P$.
Since $\Def^0_{F_E/\cE}\subset\Def_{F_E/\cE}$ is a base change of 
$\Def^0_{F_E/\sE}\subset\Def_{F_E/\sE}$, in this section we will do all calculations on $\sE$
and worry about the finite base change later.
\end{rmk}

\begin{rmk}
Note that irreducible components of $\Def^0_{F_E/\sE}$ are not necessarily smooth over $\Spec B$.
So $1$-parameter deformations of $R_{r,a}$ contained in these components are not necessarily parametrized by $s\in B$ but may require a finite base change (such as $t\in A$).
\end{rmk}

In the remainder of this section, we will describe 
the closed subscheme $\Def^0_{F_E/\sE}$ and 
 the family of algebras $\End(\cV_p)$ over it
using homological mirror symmetry for the family of genus one curves~$\sE$. The answer is given in Corollary~\ref{argaergqehqeth}.

Recall that that the mirror of $F_E$ is an immersed oriented Lagrangian $\bK_{r,a}$ 
on a symplectic torus $\bT_2$ with two punctures (orange and black)  equipped %. It has $r - 1$ self-intersection points and is equipped 
with brane data (bounding spin structure, trivial local system, and standard grading).\break 
In the previous section, we worked with the exact Fukaya category of $\bT_2$ as defined in \cite{Seidelbook}, which provides a mirror for $E_2$.
In this section, we re-interpret the black puncture as a divisor $\{s\} \subset \bT_1$ 
on a one-punctured torus. The computations will take place in the relative exact Fukaya category $\mathcal{F}(\bT_1, \{s\})$, which provides a mirror for the family of genus $1$ curves $\sE \to \Spec B$. Indeed, \cite{LPol} establishes mirror symmetry for the Tate family of curves $T_2 \to \mathrm{Spec}\, \mathbb{Z}[[t_1,t_2]]$, which is the total space of the versal formal deformations of the special fiber $E_2$ (see \cite[Section 2]{LPol}). 
The main result \cite[Theorem~A]{LPol} establishes a quasi-equivalence 
$\mathcal{F}(\mathbb{T}, \{s,o\} ) \simeq \mathrm{Perf}\, T_2$
over $\mathbb{Z}[[t_1,t_2]]$, 
where the left-hand-side is the  split-closed derived   Fukaya category of the compact torus $\mathbb{T}$ relative to compactification divisor given by 2 points $\{s,o\}$.  Note that $\mathbb{T}\setminus \{s, o \} = \mathbb{T}_2$. In the relative Fukaya category $t_1$ (resp. $t_2$) is a formal parameter keeping track of the intersection number of holomorphic polygons with $s$ (resp. $o$). The family of genus 1 curves $\sE \to \Spec B$ corresponds to the subfamily  $(t_2=0)$, where the curve $E_2$  deforms to a curve $E_1$ by smoothing a black node $P$ and retaining an orange node $Q$ in Figure~\ref{zxfbxzfbzfnzdf}. The techniques used in the proof of Theorem \cite[Theorem A]{LPol} apply directly in this case to give the %following
 quasi-equivalence over $B$,
\begin{equation} \mathcal{F}(\mathbb{T}_1, \{s\}) \simeq \mathrm{Perf}\, \mathscr{E}. \end{equation}
Here $\mathbb{T}_1$ is once-punctured torus,  $D= \{s \}$ (formerly known as a black puncture) is a divisor with respect to which we study the relative Fukaya category and $\mathscr{E} \to \Spec B$ is the family of nodal curves where the special fiber is the $E_2$ and general fiber is $E_1$, 
the nodal rational curve. 
The $A_\infty$-category $ \mathcal{F}(\mathbb{T}_1, \{s\})$ is $B$-linear. % and its definition  uses a local parameter $s \in B$. %Recall that, in the relative Fukaya category, 
The~$A_\infty$-operations are given by counting holomorphic polygons with boundaries on Lagrangians, but the contribution of each polygon $u$ comes with a weight $s^{\mathrm{mult}(u, \{s\})}$. (See \cite{LPol} for more background on relative Fukaya categories.) 

We can view the Kawamata Lagrangian $\bK:=\bK_{r,a}$ as an object of $\mathcal{F}(\bT_1, \{s\})$, 
which is a mirror of some deformation of the Kawamata vector bundle $F_E$
to a vector bundle on~$\sE$.
We start by computing the $A_\infty$-algebra $(\mathscr{A}_{\bK}, \{\mathfrak{m}_i\}_{i \geq 1})$ of endomorphisms of $\mathbb{K}$ as an object in $\mathcal{F}(\mathbb{T}_1,\{s\})$. 
We follow the sign conventions as given in \cite[Ch. 1]{Seidelbook}. Recall that an $A_\infty$-algebra $\mathscr{A}$ over a commutative ring $B$ is a $\f{Z}$-graded $B$-module with a collection of $B$-linear maps $\mathfrak{m}_i : \mathscr{A}^{\otimes i} \to \mathscr{A}[2-i]$ for $i \geq 1$, where the notation $\mathscr{A}[2-i]$ means that $\mathfrak{m}_{i}$ lowers the degree by $i - 2$. These~maps are required to satisfy the $A_\infty$-relations:
\[ \sum_{j,k} (-1)^{|a_1|+\ldots + |a_j|-j} \mathfrak{m}_{i-k+1}(a_i, \ldots, a_{j+k+1}, \mathfrak{m}_k (a_{j+k}, \ldots, a_{j+1}), a_j, \ldots, a_1) = 0. \]
The cohomology with respect to $\mathfrak{m}_1$ is an associative algebra with the product:
\[ a_2 a_1 = (-1)^{|a_1|} \mathfrak{m}_2(a_2, a_1). \]
The underlying complex of $(\mathscr{A}_{\bK}, \{\mathfrak{m}_i\}_{i \geq 1})$ is the Floer cochain complex
%Let us introduce some general notation. Given a Lagrangian immersion $L$ with $(r - 1)$ self-intersection points, 
given as a $B$-module by
\[ CF(\bK, \bK) = \mathrm{hom}^0(\bK, \bK) \oplus \mathrm{hom}^1(\bK, \bK), \quad \hbox{\rm where} \]
\[ 
  \mathrm{hom}^0(\bK, \bK) = B w_0 \oplus \bigoplus_{i=1}^{r-1} B w_i \quad \hbox{\rm and} \quad
  \mathrm{hom}^1(\bK, \bK) = B \bar{w}_0 \oplus \bigoplus_{i=1}^{r-1} B \bar{w}_i.
\]

\begin{definition}\label{sFBAFBADFNAD}
For $i \neq 0$, we associate a pair of generators ${w_i , \bar{w}_i}$ with each self-intersection point of the Lagrangian
~$\bK$. %For the Kawamata Lagrangian $\Bbb K_{r,a}$ (see Theorem~\ref{argargerhwetb}), 
%We place  $e$ and $q$ as illustrated in Figure~\ref{argargargar}.
The generators $e = w_0$ and $q = \bar{w}_0$, placed 
as illustrated in Figure~\ref{argargargar},
 correspond to the minimum and the maximum of a Morse function chosen on the domain of the immersion of the Lagrangian.
\end{definition}

\begin{thm}\label{sGsgrsherheq}
The $A_\infty$-algebra 
$(\mathscr{A}_{\bK_{r,a}},\{\mathfrak{m}_i\}_{i\geq 1})$
has the following products:
\begin{enumerate}
\item[(1)] For $i=0,\ldots, r-1$, 
%Here we place the strict unit $w_0= e$ and its dual counit $\overline{w}_0=q$ as in Figure~\ref{argargargar}.
%We have
%\[ \mathscr{A}_0 = \mathrm{hom}^0(\bK,\bK) \oplus \mathrm{hom}^1(\bK,\bK), \]
%\begin{align*}
%  \mathrm{hom}^0(\bK,\bK) = \bigoplus_{i\in\bZ_r} B w_i , \ \ \ \ \mathrm{hom}^1(\bK,\bK) = \bigoplus_{i\in\bZ_r} B\bar{w}_i,
%\end{align*}
%where the generators $\{w_i,\bar{w}_i\}$ for $i\ne0$ correspond to self-intersection points of the Lagrangian.
\begin{align*}
\mathfrak{m}_2(w_i,w_0) &= w_i = \mathfrak{m}_2(w_0,w_i) \\
\mathfrak{m}_2(\bar{w}_i,w_0) &= \bar{w}_i = -\mathfrak{m}_2(w_0,\bar{w}_i) \\
\mathfrak{m}_2(\bar{w}_i,w_i) &= \bar{w}_0 = - \mathfrak{m}_2(w_i,\bar{w}_i) 
\end{align*}
\item[(2)] 
For each $i\neq 0$, %we have the triple products
  \begin{align*} \mathfrak{m}_3(\bar{w}_i, w_i, \bar{w}_i) &= -\bar{w}_i \\
    \mathfrak{m}_3(\bar{w}_i,w_i, \bar{w}_0) &= -\bar{w}_0 \\
    \mathfrak{m}_3(w_i,\bar{w}_i, \bar{w}_0) &= \bar{w}_0  
    \end{align*} 
% which was denoted by $q$ in the previous examples. 
%We always have the following $\mathfrak{m}_2$ products:
%\begin{align*}
%\mathfrak{m}_2(w_i,w_0) &= w_i = \mathfrak{m}_2(w_0,w_i) \\
%\mathfrak{m}_2(\bar{w}_i,w_0) &= \bar{w}_i = -\mathfrak{m}_2(w_0,\bar{w}_i) \\
%\mathfrak{m}_2(\bar{w}_i,w_i) &= \bar{w}_0 = - \mathfrak{m}_2(w_i,\bar{w}_i) 
%\end{align*}
%for $i=0,\ldots, r-1$, coming from the fact that we arrange $w_0$ to be a strict unit, and arrange $w_i, \bar{w}_i$ to be related by Poincar\'e duality. 
%\end{align*}
\item[(3)] 
For each  sub-interval $(x,y)$ of the {\it Gauss word}
(see Figure~\ref{asrgargqre})
\[ w_{r-1},w_{r-2}, \ldots, w_1, w_{-b}, w_{-2b}, \ldots, w_{-(r-1)b}, \]
%For each sub-interval $(x,y)$ on this chain we have the following triple products:
\begin{align*}
  \mathfrak{m}_3(y,\bar{x},x)= y,\ 
  \mathfrak{m}_3(\bar{x},x,\bar{y})= -\bar{y}\quad
  &\hbox{\rm 
  if both $x$ and $y$ are in $\{-b,-2b,\ldots, -(r-1)b\}$,}\\
 \mathfrak{m}_3(x, \bar{x}, \bar{y})= \bar{y},\ 
  \mathfrak{m}_3(y,x,\bar{x})= - y\quad
  &\hbox{\rm 
if  $x\in \{r-1,\ldots,1\}$ and $y\in \{-b,\ldots, -(r-1)b\}$,}\\
  \mathfrak{m}_3(\bar{y},x, \bar{x}) = -\bar{y},\ 
  \mathfrak{m}_3(x,\bar{x}, y)= -y\quad
  &\hbox{\rm 
if both $x$ and $y$ are in $\{r-1,r-2,\ldots,1\}$.}
\end{align*}
\item[(4)] 
Products that correspond to `visible'  polygons are
described in Figure~\ref{SfvsbfsbFDnD}.
\begin{figure}[hbtp]
\begin{center}
\includegraphics[width=\textwidth]{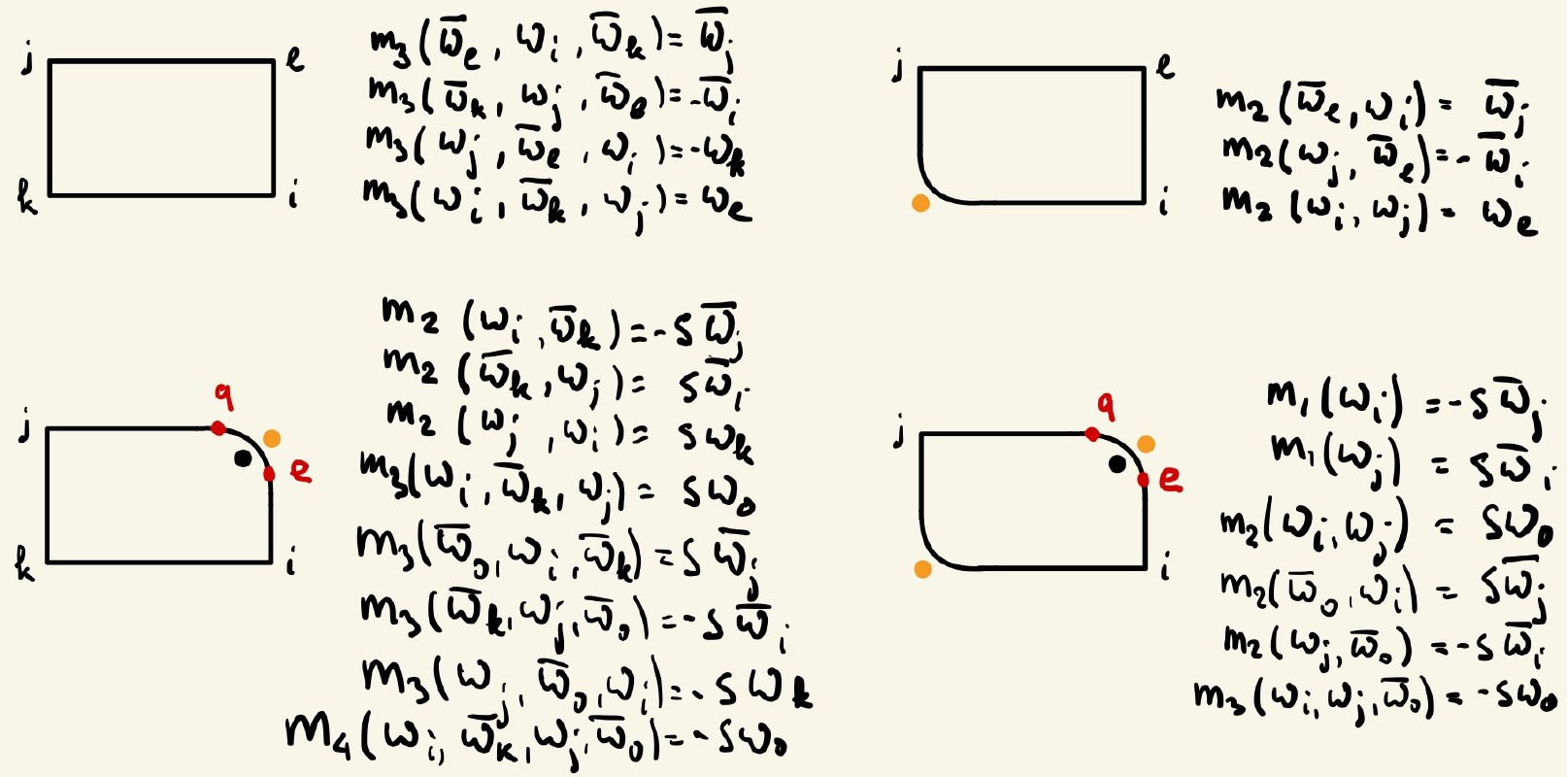}\qquad
\end{center}
\caption{Contributions to $\mathscr{A}_{\bK_{r,a}}$ from visible holomorphic polygons}
\label{SfvsbfsbFDnD}
\end{figure}
\end{enumerate}
\end{thm}

\begin{proof} Let $\bK:=\bK_{r,a}$.
The $A_\infty$ structure of $\mathscr{A}_\bK$ is defined via counts of holomorphic polygons with boundary on~$\bK$. We stick to the conventions laid out in \cite[Section 7]{Seidelgenus2}. When working with Fukaya categories of surfaces, the holomorphic curve contributions come in two flavors. There are ``visible'' immersed polygons with boundary on the Lagrangians, and ``virtual'' polygons which only become visible after successive perturbations of the Lagrangians. Before perturbing, the corresponding moduli spaces are in general not regular around constant maps for polygons with more than three edges (see \cite{LPerutz} for an illustration), hence to count these contributions correctly one has to use virtual fundamental chains or else perturb. A concrete way to deal with this issue in the case of surfaces is by taking successive push-offs of the Lagrangian using a small Hamiltonian perturbation. Fortunately, in this paper, other than the bigons and triangles that contribute to the differential $\mathfrak{m}_1$ and the product $\mathfrak{m}_2$ for which regularity can be arranged (even before perturbing, see \cite[Section 7]{Seidelgenus2}), we will only deal with holomorphic rectangles to compute the contributions to $\mathfrak{m}_3$ and the perturbations by push-offs still remain manageable.
Taking only the virtual contributions into account, one gets a model for the Fukaya category of the Weinstein neighborhood of the immersed Lagrangian (a plumbing) which can be described with an $A_\infty$-algebra with $\mathfrak{m}_i \neq 0$ only for $i=2,3$. We call this algebra {\it a hidden $A_\infty$  algebra $\mathscr{A}_0$}. The~contributions (1), (2), (3) to the $A_\infty$  algebra $\mathscr{A}_\bK$ come from this hidden algebra.

%\begin{definition}
%We place the counit $q$ and the unit $e$ consequently according to the orientation of the Lagrangian $\bL$.
%Continuing along the Lagragian and recording all self-intersection points until each point is counted twice, gives the Gauss word.
%In addition, we decorate a letter in the Gauss word with an apostrophe if, at this intersection point, the second branch of the Lagrangian approaches from the right, see Figure~\ref{asfvasasvasvas}, where we also illustrate a small Hamiltonian perturbation of the Lagrangian.
%\end{definition}
%\begin{figure}
%\begin{center}
%\includegraphics[width=\textwidth]{2_28.pdf}
%\caption{The projective algebraic  surface $W$ with the anticanonical divisor $E=A\cup B$ and the two-punctured torus $\bT_2$ that corresponds to $E$ under homological mirror symmetry}
%\label{asfvasasvasvas}
%\end{center}
%\end{figure}

\begin{figure}[hbtp]
\begin{center}
  \begin{tikzpicture}

    \tikzset{->-/.style={decoration={markings, mark=at position #1 with {\arrow{>}}},postaction={decorate}}}

    \node[] (t1) at (0,0) {};
    \node[] (t2) at (0,4) {};
    \node[] (t3) at (10,4) {};
    \node[] (t4) at (10,0) {};

    \node[] (L0left) at (0,1) {};
    \node[] (L0mid)  at (5,2) {};
    \node[] (L0right) at (10,3) {};

    \node[] (L1left) at (0,3) {};
    \node[] (L1mid)  at (5,2) {};
    \node[] (L1right) at (10,1) {};

    \node[] (x1) at (4.5,2) {\small $x$}; 
    \node[] (x2) at (5, 2.3) {\small $\bar{x}$};

    \fill[black] (6.7,2.58) circle (2pt);    
    \fill[black] (8.2,2.88) circle (2pt);    

    \node[] (q) at (6.7,2.8) {\small $q$};
    \node[] (e) at (8.2,3.1) {\small $e$};
    
     \draw[gray!90, thick] (t1.center) -- (t2.center) -- (t3.center) -- (t4.center) -- (t1.center);

    \draw[fill=black!70!black, thick] (1,2) circle (0.2cm);
    \node (t2) at (0.6,1.7) {$s$};
    \draw[fill=orange,thick] (5,3) circle (0.2cm);
    %\node (t1) at (5,3) {$o$};

    \draw[thick] (L0left.center) edge[bend right=10] (L0mid.center);
    \draw[thick] (L0mid.center) edge[bend left=10] (L0right.center);

    \draw[thick, ->-=.5] (L1left.center) to[bend left=10] (L1mid.center);
    \draw[thick] (L1mid.center) edge[bend right=10] (L1right.center);

  \end{tikzpicture}
\end{center}
\caption{Kawamata Lagrangian of ${1\over2}(1,1)$.}\label{srgarga}
\end{figure}

We will only compute these contributions 
in  the simplest example of ${1\over2}(1,1)$, since the general case is  similar.
The Kawamata Lagrangian is illustrated in Figure~\ref{srgarga}, where we denote $w_1$ by~ $x$.
We have 
$$\mathrm{hom}^0(L, L) = \mathbb{Z} e \oplus \mathbb{Z} x\quad\hbox{\rm and}\quad
\mathrm{hom}^1(L, L) = \mathbb{Z} q \oplus \mathbb{Z} \bar{x}.$$
Following the statement of Theorem~\ref{sGsgrsherheq},
in addition to $m_2$ products from (1), we need to find three $m_3$ products from (2) and
two $m_3$ products from (3) (the Gauss word is $\{1,1\}$ and only the second case of (3) is present).
These triple products are
\begin{align*}
   \mathfrak{m}_3(x, x, \bar{x}) &= -x \\
   \mathfrak{m}_3(x, \bar{x}, \bar{x}) &= \bar{x} = - \mathfrak{m}_3(\bar{x}, x, \bar{x}) \\
   \mathfrak{m}_3(x, \bar{x}, q) &= q = - \mathfrak{m}_3(\bar{x}, x, q)
\end{align*}
They can be checked by perturbing the Lagrangian by taking three push-offs. The next figure  shows the rectangle that gives  the triple product  $\mathfrak{m}_3(x, x, \bar{x}) = -x$.  
The~reader is invited to find rectangles that give  other triple products listed above.

%It is straightforward to check (by a tedious but finite computation) that the above products form an $A_\infty$ algebra (with $\mathfrak{m}_i=0$ for $i \neq 2,3$).

\definecolor{darkgreen}{rgb}{0.0, 0.5, 0.0}

\begin{center}
  \begin{tikzpicture}

      \tikzset{->-/.style={decoration={markings, mark=at position #1 with {\arrow{>}}},postaction={decorate}}}

    \node[] (t1) at (0,0) {};
    \node[] (t2) at (0,4) {};
    \node[] (t3) at (10,4) {};
    \node[] (t4) at (10,0) {};

    \node[] (L0left) at (0,1) {};
    \node[] (L0mid)  at (5,2) {};
    \node[] (L0right) at (10,3) {};

    \node[] (L1left) at (0,3) {};
    \node[] (L1mid)  at (5,2) {};
    \node[] (L1right) at (10,1) {};

    \node[] (K0left) at (0,0.7) {};
    \node[] (K0mid)  at (5,1.7) {};
    \node[] (K0right) at (10,2.7) {};

    \node[] (K1left) at (0,2.7) {};
    \node[] (K1mid)  at (5,1.7) {};
    \node[] (K1right) at (10,0.7) {};

    \node[] (N0left) at (0,0.4) {};
    \node[] (N0mid)  at (5,1.4) {};
    \node[] (N0right) at (10,2.4) {};

    \node[] (N1left) at (0,2.4) {};
    \node[] (N1mid)  at (5,1.4) {};
    \node[] (N1right) at (10,0.4) {};

    \node[] (M0left) at (0,0.1) {};
    \node[] (M0mid)  at (5,1.1) {};
    \node[] (M0right) at (10,2.1) {};

    \node[] (M1left) at (0,2.1) {};
    \node[] (M1mid)  at (5,1.1) {};
    \node[] (M1right) at (10,0.1) {};
    
  %  \node[] (x1) at (4.6,2.03) {\small $x$}; 
  %  \node[] (x2) at (5.48, 1.32) {\small $\bar{x}$};

    \node[] (q) at (6.7,2.8) {};
    \node[] (e) at (8.2,3.1) {};

    \node[] (qp) at (6.7,2.58) {};
    \node[] (ep) at (8.2,2.88) {};

    \node[] (qq) at (7.2,2.68) {};
    \node[] (eq) at (7.8,2.78) {};

    \node[] (Qq) at (7.35,2.88) {};
    \node[] (Eq) at (7.55,2.98) {};
    
    % \node[] at (7.4,2.48) {$\small q$};

     \draw[gray!90, thick] (t1.center) -- (t2.center) -- (t3.center) -- (t4.center) -- (t1.center);

     \draw[fill=gray!50] (4.6,1.56) .. controls +(20:1) and +(-165:1) .. (5.45, 1.80)  .. controls +(-23:1) and +(-177:1) .. (10, 0.98)  .. controls +(-90:1) and +(90:1) .. (10,0.42) .. controls +(190:1) and +(-32:1) .. cycle;
     
    \draw[fill=gray!50] (4.6,1.24) .. controls +(1600:1) and +(-20:1) .. (3.7, 1.54)  .. controls +(205:1) and +(0:1) .. (0, 1)  .. controls +(-90:1) and +(90:1) .. (0,0.4) .. controls +(0:1) and +(210:1) .. cycle;
      
    \draw[fill=black!70!black, thick] (1,1.7) circle (0.15cm);
    \node (t2) at (0.7,1.4) {$s$};
    \draw[fill=orange,thick] (5,3) circle (0.15cm);
    %\node (t1) at (5,3) {$o$};
    
    \draw[thick] (L0left.center) to[bend right=10] (L0mid.center);
    \draw[thick] (L0mid.center) edge[bend left=10] (L0right.center);
    \draw[thick,->-=.5] (L1left.center) to[bend left=10] (L1mid.center);
    \draw[thick] (L1mid.center) edge[bend right=10] (L1right.center);

    \draw[blue, thick] (K0left.center) edge[bend right=10] (K0mid.center);
    \draw[blue, thick] (K0mid.center) edge[bend right=15] (qp.center);
    \draw[blue, thick] (qp.center) edge[bend left=45] (ep.center);
    \draw[blue, thick] (ep.center) edge[bend right=10] (K0right.center);
    \draw[blue, thick, ->-=.5] (K1left.center) to[bend left=10] (K1mid.center);
    \draw[blue, thick] (K1mid.center) edge[bend right=10] (K1right.center);

    \draw[red, thick] (N0left.center) edge[bend right=10] (N0mid.center);
    \draw[red, thick] (N0mid.center) .. controls +(30:1) and +(-100:1) .. (qq.center);
    \draw[red,thick] (qq.center) .. controls +(100:1) and +(110:1) .. (eq.center);
    \draw[red, thick] (eq.center) .. controls +(-50:1) and +(180:1) .. (N0right.center);
    \draw[red, thick, ->-=.5] (N1left.center) to[bend left=10] (N1mid.center);
    \draw[red, thick] (N1mid.center) edge[bend right=10] (N1right.center);

    \draw[darkgreen, thick] (M0left.center) edge[bend right=10] (M0mid.center);
    \draw[darkgreen, thick] (M0mid.center) .. controls +(15:1) and +(-80:1) .. (Qq.center);
    \draw[darkgreen, thick] (Qq.center) .. controls +(100:1) and +(110:1) .. (Eq.center);
    \draw[darkgreen, thick] (Eq.center) .. controls +(-70:1) and +(180:1) .. (M0right.center);
    \draw[darkgreen, thick, ->-=.5] (M1left.center) to[bend left=10] (M1mid.center);
    \draw[darkgreen, thick] (M1mid.center) edge[bend right=10] (M1right.center);

     %\fill[black] (qq) circle (2pt);
     \fill[black] (5.45,1.81) circle (2pt);    
    \fill[black] (4.6,1.55) circle (2pt);    
    \fill[black] (4.6,1.25) circle (2pt);
    \fill[black] (3.7,1.54) circle (2pt);

    \node[] (v1) at (11,1) {};
    \node[] (v2) at (11,2) {} ;
    \node[] (v3) at (12,2) {};
    \node[] (v4) at (12,1) {};

    \draw[thick,darkgreen, ->-=0.6] (v1.center) -- (v2.center);
    \draw[thick,red, ->-=0.6] (v3.center) -- (v2.center);
    \draw[thick,blue, ->-=0.6] (v3.center) -- (v4.center);
    \draw[thick, ->-=-.5] (v4.center) -- (v1.center);

    \node[shift={(0,-.2)}] at (v1) {\small $x$}; 
    \node[shift={(0,.2)}] at (v2) {\small $x$};
    \node[shift={(0,.2)}] at (v3) {\small $x$};
    \node[shift={(0,-.22)}] at (v4) {\small $\bar{x}$};
    
  \end{tikzpicture}
\end{center}

    The contributions via perturbation are computed in the following manner. Let $L$, ${\color{blue} L'}$, ${\color{red}L''}$ and ${\color{darkgreen}L'''}$ are the original Lagrangian and its push-offs. Then, treating these Lagrangians as separate, we compute the triangle that contributes to the product
    \[ \mathfrak{m}_3: CF({\color{red} L''}, {\color{darkgreen} L'''}) \otimes CF({\color{blue} L'},{\color{red} L''}) \otimes CF(L, {\color{blue} L'} ) \to CF(L, {\color{darkgreen} L'''}) \]
   
    This means looking for rectangles (in the case of $\mathfrak{m}_3$) whose boundary traces the Lagrangians  $L$, ${\color{blue} L'}$, ${\color{red}L''}$ and ${\color{darkgreen} L'''}$ in the counter-clockwise order. The corner between $L, {\color{darkgreen} L'''}$ is treated as an output and all the others are input. Once we orient our Lagrangians (which we always do in the way indicated), an intersection point $p \in CF(L,K)$ corresponds to a degree 0 generator if the intersection number $L \cdot K = -1$ and a degree 1 generator if the intersection number $L\cdot K = +1$. The sign contribution of a polygon is determined according to whether the orientation of the Lagrangians in its boundary matches with the counter-clockwise orientation of the boundary of the polygon, see \cite[Section 7]{Seidelgenus2} for a detailed explanation. Finally, to get the product defined on $CF(L,L)$, we identify $CF(L,L)$ with  $CF(L, {\color{blue} L'} )$, $CF({\color{blue} L'},{\color{red} L''})$, $CF({\color{red} L''}, {\color{darkgreen} L'''})$ and  $CF(L, {\color{darkgreen} L'''})$.
    In~general, these identifications might be non-trivial to compute but in the case of surfaces the are straightforward.

Finally, we analyze contributions to $\mathscr{A}_\bK$ given by equations
Theorem~\ref{sGsgrsherheq}~(4). They come from the visible holomorphic polygons, for which the corresponding moduli spaces are regular. We illustrate some visible polygons in Figure~\ref{argargargar}.
\begin{figure}[hbtp]
\begin{center}
\includegraphics[width=0.6\textwidth]{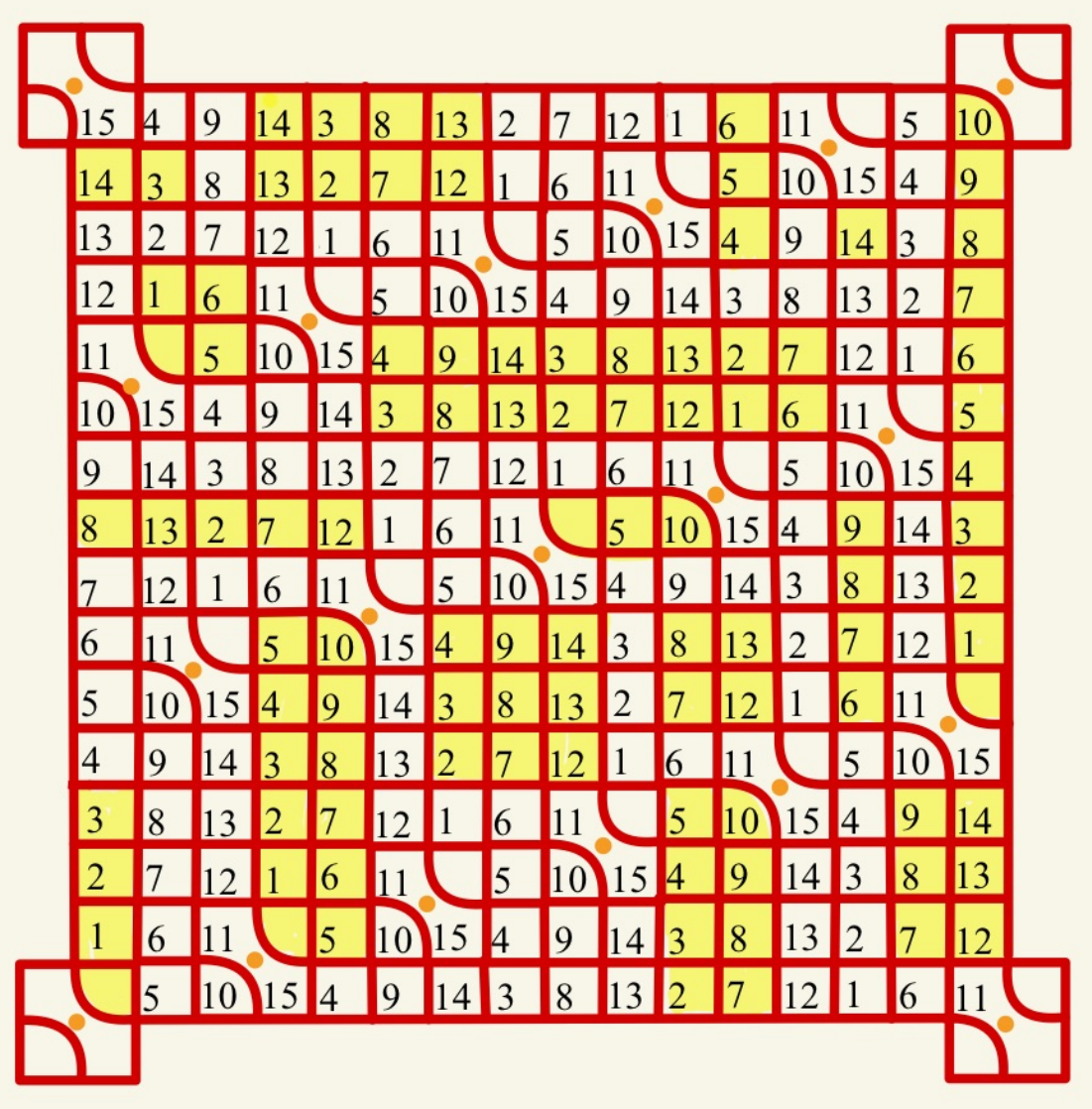}
\end{center}
\caption{Some visible holomorphic polygons}
\label{argargargar}
\end{figure}

We can view these contributions as providing an $A_\infty$ deformation of the hidden $A_\infty$-algebra $\mathscr{A}_0$ to $\mathscr{A}_\bK$. A~special feature of our Lagrangian $\bK$ is its grid-like structure, illustrated in Figure~\ref{argargargar}. It~implies that all visible polygons with boundary on $\bK$ that do not contain the orange puncture (but may contain the black puncture) are either {\it rectangles} or degenerations of rectangles ({\it bigons} or {\it triangles}), where the missing vertices of a rectangle correspond to the curved sections of the Kawamata Lagrangian passing close to the orange puncture. These polygons are given in Figure~\ref{SfvsbfsbFDnD}. 
In the ${1\over2}(1,1)$ example, only the bigon from the bottom right corner of Figure~\ref{SfvsbfsbFDnD} shows up. Since $i = j = 1$ in this case, the contributions to $m_1$ cancel each other out, leaving the products
$$ \mathfrak{m}_2(x, x) = s \cdot e = -\mathfrak{m}_3(x, x, q) \quad \hbox{\rm and} \quad
  \mathfrak{m}_2(q, x) = s \cdot \bar{x} = -\mathfrak{m}_2(x, q) $$
Computing these products for general $(r,a)$ is a routine calculation providing several cases summarized in Figure~\ref{SfvsbfsbFDnD}.
\end{proof}

The Kawamata Lagrangian $\bK=\bK_{r,a}$ 
is the mirror of one possible deformation of the vector bundle $F_E$ to the family of genus $1$ curves $\sE$.
To compute the mirrors of all possible deformations,
%, or more precisely their endomorphism algebras,  which will provide deformations of the Kalck-Karmazyn algebra $R=\End(F_E)$.
we use the formalism of bounding cochains.

Let $\mathfrak{b} \in CF^1(\bK,\bK)$.
By Fukaya--Oh--Ohta--Ono \cite{FOOO}, the new products 
 \[ \mathfrak{m}^{\mathfrak{b}}_i (x_i,x_{i-1},\ldots, x_1)= \sum_{j \geq i} \mathfrak{m}_j (\mathfrak{b},\ldots, \mathfrak{b},x_i,\mathfrak{b},\ldots, \mathfrak{b},x_{i-1},\mathfrak{b},\ldots,\ldots,\mathfrak{b},x_1,\mathfrak{b},\ldots, \mathfrak{b}) \] 
% These new products $\{ \mathfrak{m}^{\mathfrak{b}}_i\}$ 
give a deformed $A_\infty$-algebra if the {\it bounding cochain} $\mathfrak{b}$ satisfies the Maurer--Cartan equation
   \[ \mathfrak{m}_1(\mathfrak{b}) + \mathfrak{m}_2(\mathfrak{b},\mathfrak{b}) + \ldots = 0.\]
  In our case, this equation is automatic because $CF(\bK,\bK)$ has no generators in degree 2. 
Thus, we have the following immediate corollary of Theorem~\ref{sGsgrsherheq}.

\begin{cor}\label{argaergqehqeth}
Fix a bounding cochain $\mathfrak{b}= \sum\limits_{i \in \bZ_r} t_i \bar{w}_i \in CF^1(\bK,\bK)$.
Contributions to the differentials $dw_i=m_1^\mathfrak{b}(w_i)$ and to the products $w_iw_j=m_2^\mathfrak{b}(w_i,w_j)$
in the relative Fukaya category of $\mathcal{F}(\bT_1, \{s\})$ 
are given in Figure~\ref{sFGSFGADHADTEH}, where 
$s$ needs to be replaced by $s(1-t_0)$
(in practice, we will assume that $t_0=0$, so this doesn't matter.)

\begin{figure}[hbtp]
\begin{center}
\includegraphics[height=0.22\textheight]{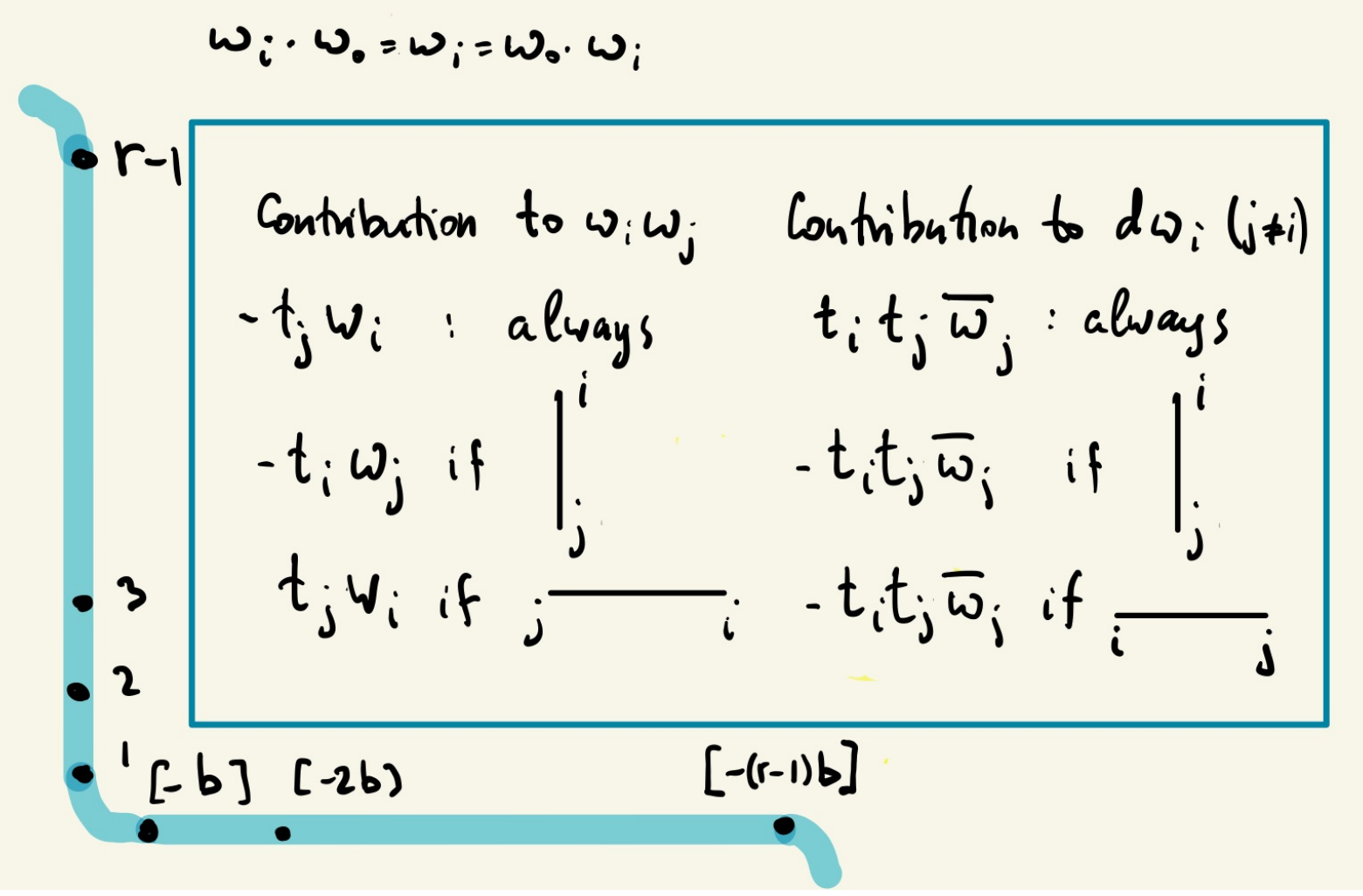}\qquad
\includegraphics[height=0.22\textheight]{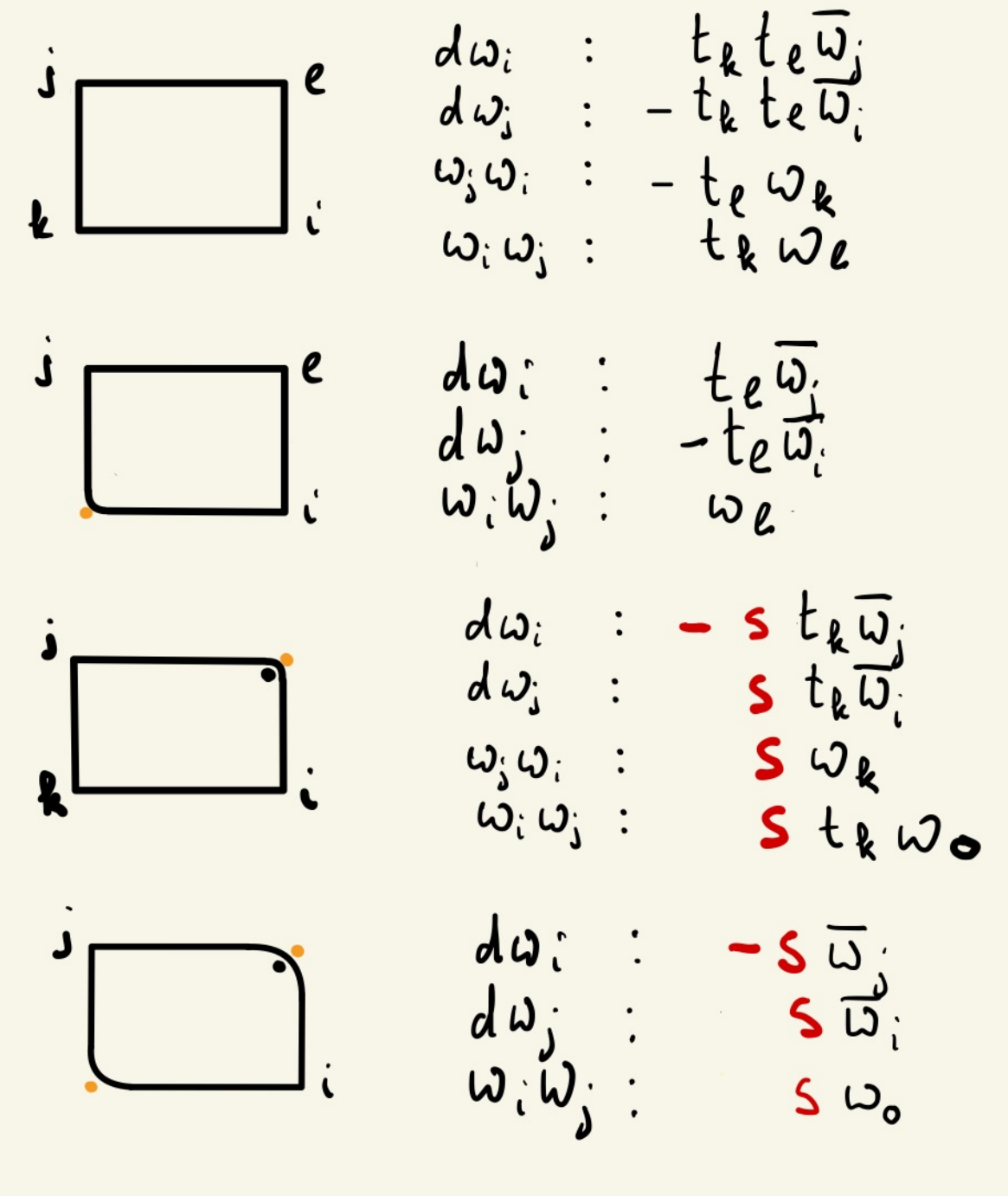}
\end{center}
\caption{Contributions from  virtual (left) and visible (right) polygons}\label{sFGSFGADHADTEH}
\end{figure}

%Contributions from the hidden algebra $\A_0$ are on the left side and from the visible polygons are on the right side.

Write $d w_i=\sum m_{ij}\bar w_j$ for $i,j=1,\ldots,r-1$.
The matrix $\mathfrak{D}=(m_{ij})$ is skew-symmetric. 
The subscheme $\Def^0_{F_E/\sE}\subset\Def_{F_E/\sE}$ of Definition~\ref{sRBASRGARGHA}
is cut out by the ideal in $B[t_0,\ldots,t_{r-1}]$ generated by the matrix entries of $\mathfrak{D}$.
Over~this~subscheme,  %$HF(\bK,\bK)$ remains constant in the relative Fukaya category of $\mathcal{F}(\bT_1, \{s\})$.  This gives 
\[ \cR=\cR_{\mathfrak{b}} = (\mathrm{hom}^0(\bK,\bK), 
\mathfrak{m}_2^\mathfrak{b}).\]
gives a flat deformation
 of the Kalck-Karmazyn algebra $R=\End(F_E)\cong \End(\bK)$. % over the subscheme $\Def^0_{F_E/\sE}$.
 \end{cor}

\begin{cor}\label{GSRGHSRHRH}
The Kalck--Karmazyn algebra $R_{r,a}$ has multiplication given by \eqref{sDBASBASBRQ}.\break
One can  also write a closed expression for this product. 
For every $i\in\bZ_r$, we define $[i]\in\bZ$ such that 
$0 \leq [i] < r$ and $i \equiv [i] \mod r$. 
We  define a   function $m(j) = \min\limits_{k=1,\ldots,[-aj]} [kb]$ for $j \in \bZ_r^*$
and set $m(0) = r$.
Then
$R_{r,a}$ has basis $w_i$ for $i\in\bZ_r$ and product
\begin{equation}\label{dghzdghdtha}
w_jw_i=\begin{cases}
w_{j+i} & \hbox{\rm if $m(j)>[i]$}\cr
0 &  \hbox{\rm otherwise.}
\end{cases}
\end{equation}
\end{cor}

\begin{proof}
When $s=t_i=0$ for all $i$, all differentials in Corollary~\ref{argaergqehqeth} trivially vanish and the only
contributions to the products $w_iw_j$ come from visible triangles (see the second row on the right side of Figure~\ref{sFGSFGADHADTEH}.) But visible triangles
precisely correspond to rectangles in the first quadrant with vertices
$$\begin{matrix}
(0,[i])& \dots & ([-aj],[i])\cr
\vdots &&\vdots\cr
(0,0) &\dots & ([-aj],0)\cr
\end{matrix}$$
does not contain any orange dots except for $(0,0)$.
This condition, appearing in~\eqref{sDBASBASBRQ},
is also equivalent to the inequality
$m(j)>[i]$.
\end{proof}

\begin{figure}[hbtp]
\begin{center}
\includegraphics[height=0.3\textheight]{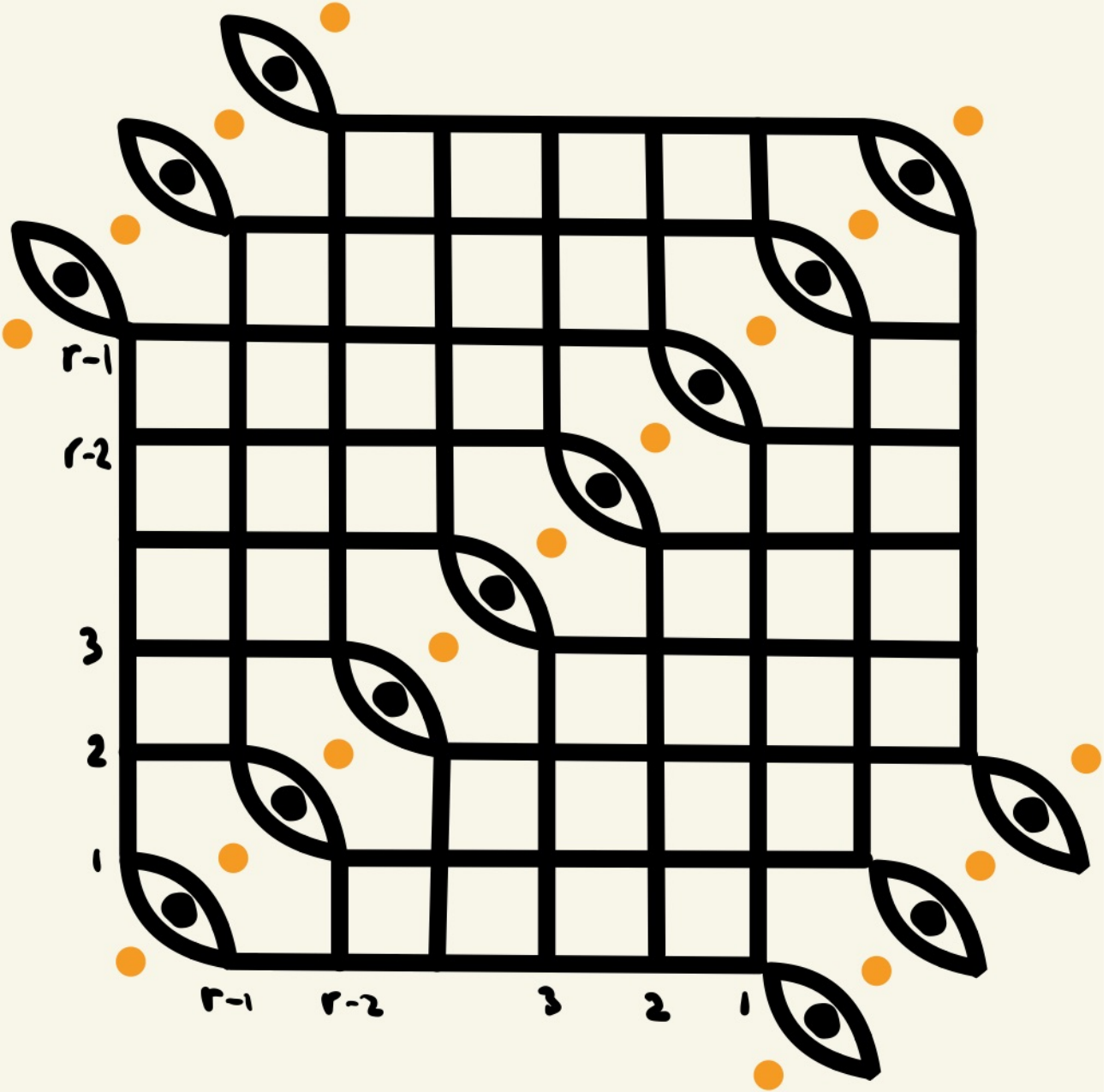}
\end{center}
\caption{Kawamata Lagrangian for ${1\over r}(1,1)$}\label{sdvasfvasfb}
\end{figure}

In the remainder of this section, we will apply Corollary~\ref{argaergqehqeth} in some examples, starting with
the Kawamata Lagrangian that corresponds to 
the cone over a rational normal curve (cyclic quotient singularity~${1\over r}(1,1)$.)

Since $a=b=1$, the Gauss word is
$\{r-1,r-2,\ldots,2,1,r-1,r-2,\ldots,2,1\}$.\break
The Lagrangian  is illustrated in Figure~\ref{sdvasfvasfb},
c.f.~Example~\ref{akjbcjaksbdakjs}.
Take a bounding cochain  $\mathfrak{b} = \sum\limits_{i \in \bZ_r^*} t_i \bar{w}_i$.
The hidden algebra $\mathscr{A}_0$ gives  contributions
$$\hbox{\rm to $dw_i$:\quad $t_it_j\bar w_j$ if $i<j$ and $-t_it_j\bar w_j$ if $i>j$}$$
$$\hbox{\rm to $w_iw_j$:\quad  $0$ if $i<j$,
$-t_jw_i-t_iw_j$ if $i>j$ and $-t_iw_i$ if $i=j$.}$$
There are two types of visible polygons (the top and the bottom row
on the right of Figure~\ref{sFGSFGADHADTEH}).
The first type is given by $1\le i<k<j\le r-1$ (and then $l=i+j-k$).
The contributions are
$$\hbox{\rm
to $dw_i$:\quad $t_kt_l\bar w_j$;\quad
to $dw_j$:\quad $-t_kt_l\bar w_i$;\quad
to $w_iw_j$:\quad 
$t_kw_l$; \quad
and 
to $w_jw_i$:\quad 
$-t_lw_k$.}$$
Finally, there is only one polygon (=bigon) of the second type, contributing
$$\hbox{\rm
to $dw_{r-1}$:\quad $-s\bar w_1$;\quad 
to $dw_1$:\quad $s\bar w_{r-1}$;\quad 
to $w_{r-1}w_1$:\quad $sw_0$.}$$

\begin{cor}
For the singularity ${1\over r}(1,1)$,
the skew-symmetric matrix $\mathfrak{D}$ of 
Corollary~\ref{argaergqehqeth} has entries  (for $i<j$) given by 
$$m_{ij}=\sum_{k=i}^{j-1}t_kt_{i+j-k}$$
 except that $m_{1,r-1}$ (if $r>2$) has an additional term $s$.
 \end{cor}
 
Legtus analyze the vanishing locus of the matrix $\mathfrak{D}$.
There are a few cases depending on~$r$.

\begin{example}[$r=2$] 
In this case $\mathfrak{D}=0$, i.e.~there are no obstructions to deformations of the Kalck--Karmazyn algebra
$R$ over $k[t_1,s]$.
The multiplication is given by 
$w_1^2=-t_1w_1+s$.
So the deformation of the Kalck--Karmazyn algebra $R=k[w_1]/(w_1^2)$ is given by 
$k[w_1]/ (w_1^2 + t_1 w_1 - s)$.
\end{example}

\begin{example}[$r\ge 3$, first component]
One of the solutions of the matrix equation $\mathfrak{D}=0$ is to take
$t_2=t_3=\ldots=t_{r-2}=0$, $s=-t_1t_{r-1}$. This is clearly the only possibility if $r=3$.
We also claim that this is the only possibility when $r\ge5$ (although in this case the ideal 
generated by entries of $\mathfrak{D}$ is not reduced). Indeed, the entries right above the diagonal are
$t_1t_2,\ldots,t_{r-2}t_{r-1}$, so some of the variables have to vanish.
On the other hand, the entries of the form $m_{i,i+2}$ are
$t_it_{i+2}+t_{i+1}^2$. So if $t_i$ (or $t_{i+2}$) vanishes, then so does $t_{i+1}$.
This forces $t_2=t_3=\ldots=t_{r-2}=0$.

Over $k[t_1,t_{r-1},s] / (t_1t_{r-1}+s)\cong k[t_1,t_{r-1}]$, the deformed algebra is given by
  \[
k\langle w_1,\ldots, w_{r-1},t_1,t_{r-1} \rangle /
\left\langle
\begin{aligned}
    &w_1^2 + t_1 w_1, w_2^2, w_3^2, \ldots, w_{r-2}^2, w_{r-1}^2 + t_{r-1} w_{r-1}\\
    &w_i w_j \text{ for } i < j, w_i w_j \text{ for } r-1 > i > j > 1\\
    &w_i w_1 + t_1 w_i, w_{r-1} w_i + t_{r-1} w_i \text{ for } 1 < i < r-1\\
    &w_{r-1} w_1 + t_{r-1} w_1 + t_1 w_{r-1} + t_1 t_{r-1}
\end{aligned}
\right\rangle
\]
Generically (when $t_1, t_{r-1}\ne 0$), this gives a deformation of the Kalck-Karmazyn algebra $R_{r,1} = k \langle w_1,\ldots, w_{r-1} \rangle / (w_1,\ldots, w_{r-1})^2$ to the path algebra of the $(r-2)$-Kronecker quiver
\begin{tikzcd}[ampersand replacement=\&]
e_1 \arrow[r, "a_1", shift left=3] \arrow[r, "\ldots", shift right=0] \arrow[r, "a_{r-2}", shift right=4] \& e_2
\end{tikzcd}
via the isomorphism $w_1 \to a_1 - t_1 e_1, w_2 \to a_2, \ldots, w_{r-2}\to a_{r-2}, w_{r-1} \to - t_{r-1} e_2$. 
\end{example}

\begin{example}[$r=4$, second component]\label{ssGwrgwrg} 
If $r=4$, there is another possibility for vanishing of the matrix $\mathfrak{D}$, namely
$t_1=t_3=s+t_2^2=0$. This gives a deformation of the Kalck--Karmazyn algebra
$R_{4,1} = k \langle w_1, w_2, w_3\rangle / (w_1, w_2, w_3)^2$
over $k[t_2,s]/ (s+t_2^2)\cong k[t_2]$
to the algebra
  \[
k[t_2] \langle w_1,w_2,w_3 \rangle / 
\left\langle
\begin{aligned}
    &w_1w_2, w_2w_3, w_1^2, w_3^2, w_2^2 + t_2 w_2, \\
    &w_1w_3 - t_2 w_2, w_3w_1 - t_2 w_2 + t_2^2, \\
    &w_3w_2 + t_2 w_3, w_2w_1 + t_2 w_1
\end{aligned}
\right\rangle
\]
For $t_2\ne0$, this algebra is isomorphic to a $2\times2$ matrix algebra $\Mat_2(k)$.
\end{example}

In accordance with Conjecture~\ref{svSfgasgasrg}, we see that.
at least in the case of ${1\over r}(1,1)$, all deformations of the Kalck--Karamazyn algebra $R = \End(F_E)$ over $\Def^0_{F_E/\sE}$ are induced by deformations of the algebraic surface~$W$. According to \cite{Pi74}, the versal deformation space of $W$ is irreducible for $r \ne 4$ (although non-reduced for $r > 4$) and corresponds to Artin deformations of $W$ (deformations induced by a deformation of the resolution of singularities of $W$), while for $r = 4$ there is an additional component that corresponds to $\Q$-Gorenstein deformations. According to \cite{TU22}, general Artin deformations of $W$ give deformations of the Kalck--Karamazyn algebra to the path algebra of the Kronecker quiver, while $\Q$-Gorenstein deformations lead to deformations to $\Mat_2(k)$. So~Example~\ref{ssGwrgwrg} gives an explicit presentation for this deformation. In the next section, we will generalize this calculation to arbitrary $\Q$-Gorenstein deformations of Wahl singularities.

\begin{example} 
We wrote a computer code \cite{LTcode} implementation of Corollary~\ref{argaergqehqeth}.
For example, let $r=15$ and $a=4$.
This is the first case when  the versal deformation space $\Def_{(E\subset W)}$ 
of a cyclic quotient singularity
has three irreducible components,
as can be verified by the computer program \cite{zuniga}. 
In accordance with Conjecture~\ref{svSfgasgasrg}, $\Def^0_{F_E/\sE}$
also has three irreducible components given by the  ideals

$I_1=({t}_{13},{t}_{12},{t}_{11},{t}_{10},{t}_{9},{t}_{6},{t}_{
      5},{t}_{4},{t}_{3},{t}_{2},{t}_{1} {t}_{14}+s,{t}_{7}^{2}-{t}_{14},{t}_{1} {t}_{7}-{t}_{8})$,

$I_2=({t}_{13},{t}_{10},{t}_{9},{t}_{8},{t}_{7},{t}_{6},{t}_{5
      },{t}_{2},{t}_{1} {t}_{14}+s,{t}_{3}
      {t}_{11}-{t}_{14},{t}_{1} {t}_{11}-{t}_{12},{t}_{1} {t}_{3}-{t}_{4})$,
      
 $I_3=({t}_{14},{t}_{12},{t}_{10},{t}_{9},{t}_{8},{t}_{7},{t}_{6
      },{t}_{5},{t}_{3},{t}_{1},{t}_{2} {t}_{13}+s,{t}_{2}
      {t}_{11}-{t}_{13},{t}_{2}^{2}-{t}_{4})$.
\end{example}
      
\begin{example}       
Another singularity with $3$ irreducible components is ${1\over 19}(1,7)$, which was  analyzed in \cite{KSB}.
$\Def^0_{F_E/\sE}$ has $3$ irreducible components given by the  ideals

$I_1= ({t}_{4},{t}_{15},{t}_{6},{t}_{13},{t}_{16},{t}_{3},{t}_{1
      },{t}_{18},{t}_{8},{t}_{11},{t}_{5}
      {t}_{7}-{t}_{12},-{t}_{7}^{2}+{t}_{14},{t}_{7} {t}_{12}+s,$
      
      \ \hfill${t}_{2}
      {t}_{5}-{t}_{7},-{t}_{5}
      {t}_{12}+{t}_{17},-{t}_{2}
      {t}_{7}+{t}_{9},-{t}_{5}^{2}+{t}_{10})$
      
$I_2= ({t}_{4},{t}_{15},{t}_{7},{t}_{12},{t}_{16},{t}_{3},{t}_{2
      },{t}_{17},{t}_{1}
      {t}_{13}-{t}_{14},{t}_{1} {t}_{5}-{t}_{6},-{t}_{5}
      {t}_{13}+{t}_{18},$

      \ \hfill${t}_{6} {t}_{13}+s,{t}_{5}
      {t}_{8}-{t}_{13},-{t}_{5} {t}_{6}+{t}_{11},-{t}_{1} {t}_{8}+{t}_{9},-{t}_{5}^{2}+{t}_{10})$
      
$I_3= ({t}_{4},{t}_{15},{t}_{7},{t}_{12},{t}_{6},{t}_{13},{t}_{5
      },{t}_{14},{t}_{3} {t}_{16}+s,{t}_{2}
      {t}_{16}-{t}_{18},{t}_{1} {t}_{2}-{t}_{3},$

      \ \hfill$-{t}_{3} {t}_{8}+{t}_{11},-{t}_{1}
      {t}_{16}+{t}_{17},-{t}_{1}
      {t}_{8}+{t}_{9},-{t}_{2} {t}_{8}+{t}_{10})$.      
\end{example}

\section{ $\Q$-Gorenstein deformation of the Kalck--Karmazyn algebra}\label{adfbdfbsdns}

We fix coprime integers $0<q<n$. A cyclic quotient singularity ${1\over n^2}(1,nq-1)$
is called a {\it Wahl singularity}.
It can also be  described as
$ (xy=z^n)\subset {1\over n}(1,-1,q)$.
A special feature of the Wahl singularity is that it
admits 
 a $1$-dimensional versal $\Q$-Gorenstein\footnote{Recall that a flat deformation of an  algebraic surface over a smooth base is called $\Q$-Gorenstein
if the relative canonical divisor is $\Q$-Cartier.
}
 deformation space, namely
\begin{equation}\label{sFSGSrh}
(xy=z^n+t)\subset {1\over n}(1,-1,q)\times\bA^1_t.
\end{equation}
We compactify the Wahl singularity to a projective surface $W$ as in Section~\ref{sfbdfbdhzdtjts}
and let $\cW$ be the corresponding projective $\Q$-Gorenstein
deformation. After a finite base change, the total space of the deformation carries a 
torsion-free sheaf $\cH$ introduced by Hacking \cite{H13} (we use a version from  \cite{K21})
such that its restriction to a general fiber $\cW_t$ is an exceptional vector bundle.
It was proved by Kawamata \cite{K21}
that the restriction of the Kawamata vector bundle $\cF$  to $\cW_t$ splits as 
\begin{equation}\label{dfdfgadhadh}
\cF_t\cong \cH_t^{\oplus n}
\end{equation}
and so the Kalck--Karmazyn algebra 
$R=\End(F)$
deforms to the matrix algebra
$$\cR_t=\End(\cF_t)\cong\Mat_n(k).$$
By Tsen's theorem,  the flat
$k[t]$-algebra $\cR=\End(\cF)$ 
is an order over $\Spec k[t]$, i.e.~$\cR\otimes_{k[t]} k(t)\cong\Mat_n(k(t))$.
We will use  machinery of Kawamata Lagrangians 
to write down an explicit embedding $\cR\hookrightarrow \Mat_n(k[t])$ of $k[t]$-algebras.

The rest of this section is occupied by the proof of Theorem~\ref{wFGSRGARGARE}.
Our plan is as follows.
The anticanonical divisor $\cE\subset \cW$ is obtained (locally) by setting $z=0$ in \eqref{sFSGSrh}.
It~has an $A_{n-1}$ singularity at the node $P$.
It follows that $\cE$ is obtained from the versal family $\sE$ (see Remark~\ref{sfasfbgadrh}) 
by the base change $s=t^n$. Let $\bK\in \mathcal{F}(\mathbb{T}_1, \{s\})$ be the Kawamata Lagrangian.
In Lemma~\ref{EFwegwrGWR}, we will introduce an ad hoc bounding cochain~$\mathfrak{b}\in CF^1(\bK,\bK)$
and a flat deformation
$\cR_{\mathfrak{b}} = (\mathrm{hom}^0(\bK,\bK), 
\mathfrak{m}_2^\mathfrak{b})$
 of the Kalck-Karmazyn algebra $R=R_{n^2,nq-1}$ over $\bA^1_t$.
 In Lemma~\ref{afdgzdfhadrh}, we will 
 embed $\cR_{\mathfrak{b}}$ into $\Mat_n(k[t])$ and check formulas of Theorem~\ref{wFGSRGARGARE}.
 Finally, in Lemma~\ref{dfvdfbadfbdnadt}, we will show that 
$\cR_{\mathfrak{b}}$ is isomorphic to 
the endomorphism algebra
$\cR=\End(\cF)$
of the Kawamata vector bundle completing the proof of Theorem~\ref{wFGSRGARGARE}.

\begin{lem}\label{EFwegwrGWR}
%In the notation of Corollary~\ref{argaergqehqeth}, 
Consider the locus of bounding cochains  
$\mathfrak{b} = \sum\limits_{i \in \bZ_{n^2}\setminus\{0\}} t_i \bar{w}_i\in CF^1(\bK,\bK)$
given by the following formulas:
$$s=t_n^n,\quad t_{nr}=t_n^r\ \hbox{\rm  for}\ r=1,\ldots,n-1\ \hbox{\rm  and}\ t_i=0\ \hbox{\rm  for}\ i\not\equiv0\mod n.$$
We denote $t_n$  by $t$.
This locus of bounding cochains (isomorphic to $\bA^1_t$)
 is in $\Def^0_{F_E/\sE}$.
In~particular, the algebra
$$\cR_\bb=(\mathrm{hom}^0(\bK,\bK), 
\mathfrak{m}_2^\mathfrak{b})$$
%with multiplication described in Theorem~\ref{argaergqehqeth},
gives a flat deformation of the Kalck-Karmazyn algebra~$R$ over $\bA^1_t$.
\end{lem}

\begin{figure}[hbtp]
\begin{center}
\includegraphics[height=0.25\textheight]{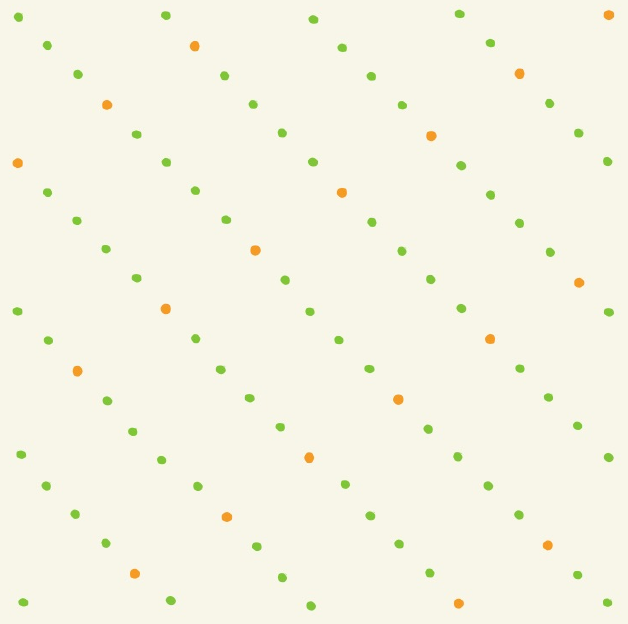}\qquad
\includegraphics[height=0.25\textheight]{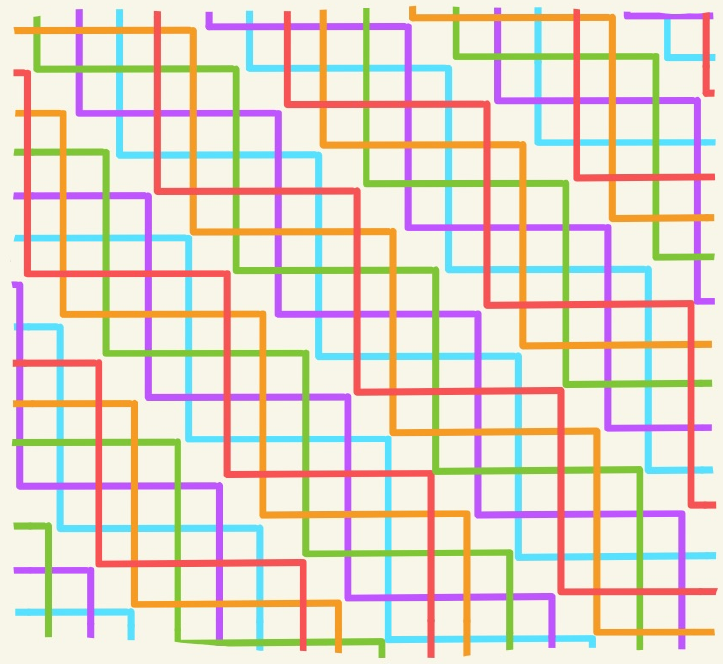}
\end{center}
\caption{$\Q$-Gorenstein deformation of the Kawamata Lagrangian}\label{afbzfbdfnzdgnadtjn}
\end{figure}

\begin{rmk}
The bounding cochain $\mathfrak{b}$ is illustrated on the left side of Figure~\ref{afbzfbdfnzdgnadtjn}, where green dots indicate self-intersection points $w_i$ of the Lagrangian $\bK$ (which is not shown) such that $t_i \ne 0$. The~right side of Figure~\ref{afbzfbdfnzdgnadtjn} explains our interest in these self-intersection points: a naive surgery deformation of $\bK$ as in Figure~\ref{dbarhathaetj} splits $\bK$ into $n$ isotopic Lagrangians, which we will later identify with mirrors of the Hacking vector bundle restricted to the general fiber of the family~$\sE$.
\end{rmk}

\begin{proof}
Since $a = nq - 1$, its inverse $b \equiv -nq - 1 \mod n^2$. The subword of the Gauss word formed by indices divisible by $n$ is
$n(n - 1), \ldots, 2n, n, n, 2n, \ldots, n(n - 1)$. It~follows that contributions to the differential $dw_i$ coming from the hidden algebra $\A_0$ (the left side of Figure~\ref{sFGSFGADHADTEH}) cancel each other out. It remains to show that contributions coming from the visible polygons (the right side of Figure~\ref{sFGSFGADHADTEH}) also cancel each other out. We interpret these polygons as lattice rectangles in $\bZ^2$. The contribution to the differential $dw_\alpha$ from a lattice rectangle is trivial unless both NE and SW corners of the rectangle are green or orange. In other words, these corners should belong to colored anti-diagonals from the left side of Figure~\ref{afbzfbdfnzdgnadtjn}. 
Furthermore, apart from these two corners, the lattice rectangles should not contain any orange points.
We claim that these {\it permitted rectangles} come in pairs as illustrated in Figure~\ref{s.DMNVBs,fbsjdhfb} (left and middle).
\begin{figure}[hbtp]
\begin{center}
\includegraphics[height=0.15\textheight]{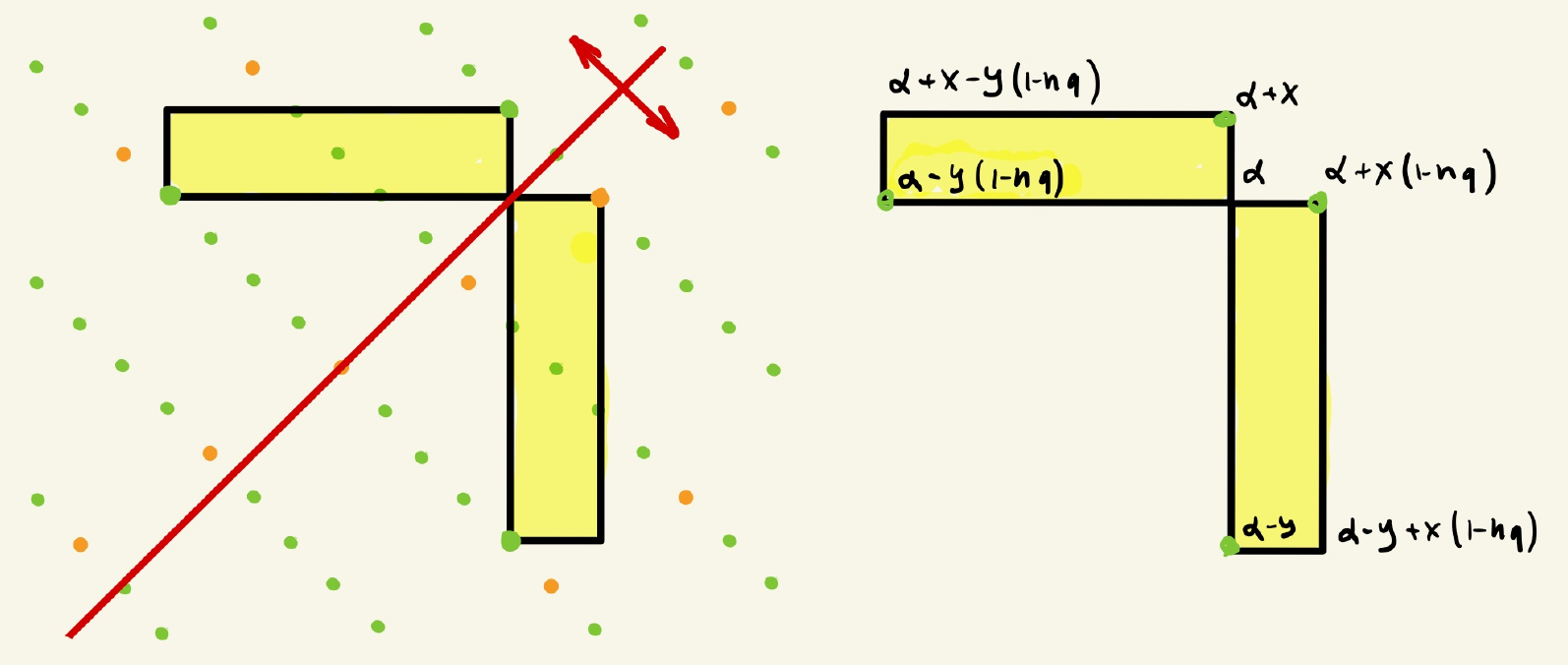}\ 
\includegraphics[height=0.15\textheight]{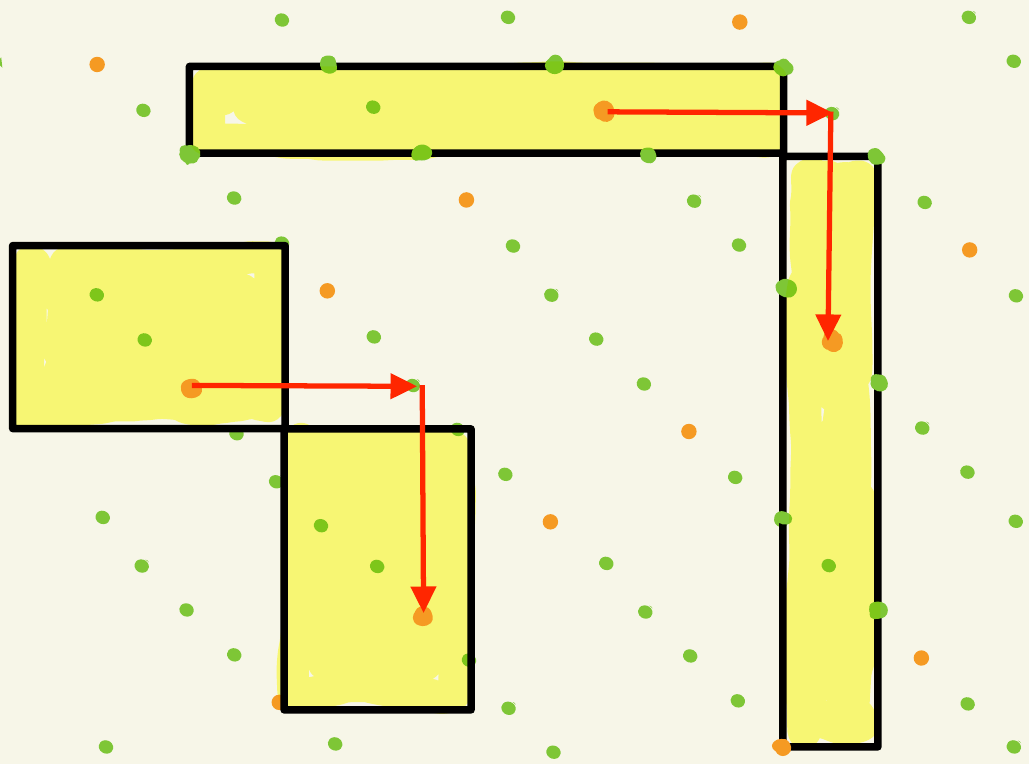}
\end{center}
\caption{Permitted rectangles come in pairs}\label{s.DMNVBs,fbsjdhfb}.
\end{figure}
Here both rectangles have shape $x\times y$.
Indeed, if the NE and SW corners of the top rectangle are green or orange then
$\alpha\equiv y\equiv -x \mod n$, which implies that these corners of the bottom rectangle are green or orange, and vice versa.
Furthermore, 
\begin{equation}\label{wrgargaerg}
\alpha+x-y(1-nq)\equiv \alpha-y+x(1-nq)\mod n^2,
\end{equation} 
 i.e. $w_{\alpha+x-y(1-nq)}=w_{\alpha-y+x(1-nq)}$. 
So both rectangles contribute  $\bar w_{\alpha+x-y(1-nq)}$ to the differential $dw_\alpha$ with coefficients that will be determined below.

Next, we claim that if one of the rectangles is not permitted then the other one is not permitted as well.
In other words, if one of the rectangles
contains orange dots (away from the NE and SW corners) then the other one does as well, 
as illustrated on the right side of Figure~\ref{s.DMNVBs,fbsjdhfb}.
Indeed, suppose the top rectangle contains an orange dot. We slide this point anti-diagonally (in the SE direction)
until it hits the bottom rectangle. We claim that one of the dots on this anti-diagonal within the bottom-right rectangle is orange.
In order to find this orange dot, 
we decompose the SE translation as illustrated on the right side of Figure~\ref{s.DMNVBs,fbsjdhfb}.
Namely, we first move to the right, hopping from one anti-diagonal of green/orange dots
to the next, until we get the point that can be moved down into the  bottom-right rectangle. This point will be orange.

Finally, we have to check that if both rectangles in the pair are permitted then the contributions to the differential $dw_\alpha$ given by the NW corner of the top rectangle and the SE corner of the bottom rectangle cancel each other out. Concretely, we need to check that
\begin{equation}\label{sDVSFBASFBAFD}
[\alpha-y]+[\alpha+x-nqx]+n^2\ (\hbox{\rm where we only add $n^2$ if}\  [\alpha+x-nqx]=0)
\end{equation}
$$=[\alpha-y+nqy]+[\alpha+x]+n^2\ (\hbox{\rm  where we only add $n^2$ if}\  [\alpha+x]=0)$$
Note that, if $y>n$, the condition \eqref{sDVSFBASFBAFD} is invariant under the change $y\mapsto y-n$, which corresponds to 
shortening the rectangles. Indeed, this obviously preserves permissibility of the rectangles.
Furthermore, the left hand side of \eqref{sDVSFBASFBAFD} increases by $n$
(note that $[\alpha-y+n]\ne0$ since otherwise the bottom rectangle is not permissible),
and the same is true for the right hand side.
We also claim that, if $x>n$, we can shorten the rectangles in the other direction, $x\mapsto x-n$.
Again, this obviously preserves permissibility of the rectangles.
We claim that both sides of \eqref{sDVSFBASFBAFD}  decrease by $n$ under this operation.
Indeed, neither $[\alpha+(x-n)]$ nor $[\alpha+(x-n)-nq(x-n)]$ is equal to $0$ by permissibility of the rectangles. 
Furthermore, if either $[\alpha+x]$ nor $[\alpha+x-nqx]$ is equal to $0$ then the formula works because 
$n^2$ is added to the formula to compensate.

By the above, we can assume that $x,y\le n$. Since 
$\alpha\equiv y\equiv -x \mod n$, $x+y=n$ or $2n$. The second case is, however,
impossible because then $x=y=n$ and every $n\times n$ square with green or orange vertices contains orange 
along the anti-diagonal, which is not permitted. So $0<x<n$ and $y=n-x$.
We rewrite \eqref{sDVSFBASFBAFD} as follows:
\begin{equation}\label{szs,fjbvadfhbav}
[\alpha-n+x]+[\alpha+x-nqx]+n^2\ (\hbox{\rm where we only add $n^2$ if}\  [\alpha+x-nqx]=0)
\end{equation}
$$=[\alpha-n+x-nqx]+[\alpha+x]+n^2\ (\hbox{\rm  where we only add $n^2$ if}\  [\alpha+x]=0)$$
But this is clear:
$$
[\alpha-n+x]=[\alpha+x]-n+n^2\ (\hbox{\rm  where we only add $n^2$ if}\  [\alpha+x]=0),$$
and
$$
[\alpha-n+x-nqx]=[\alpha+x-nqx]-n+n^2\ (\hbox{\rm where we only add $n^2$ if}\  [\alpha+x-nqx]=0).$$
This completes the proof.
\end{proof}

%\begin{cor}
%$\cR_\bb$ is a $\bZ_{n^2}$-graded $k[t]$-algebra, where $t$ has weight $n$.
%It has a $k[t]$-basis $\{w_i\}$
%and multiplication of Corollary~\ref{argaergqehqeth}, 
%where $s=t^n$, $t_{nr}=t^r$ for $1\le r\le n-1$ and $t_k=0$ otherwise.
%\end{cor}

\begin{figure}[hbtp]
\begin{center}
\includegraphics[width=0.45\textwidth]{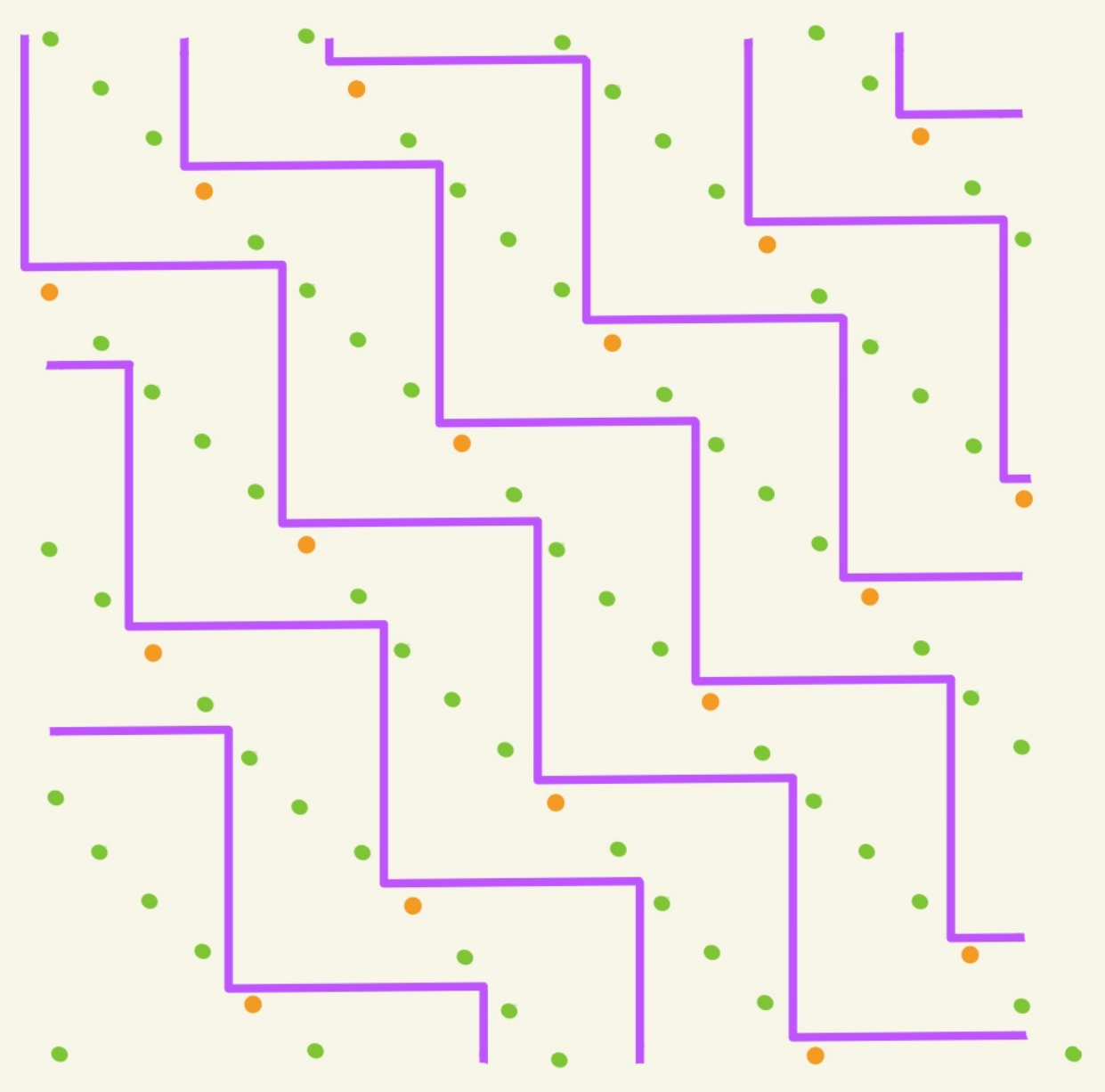}\qquad
\includegraphics[width=0.45\textwidth]{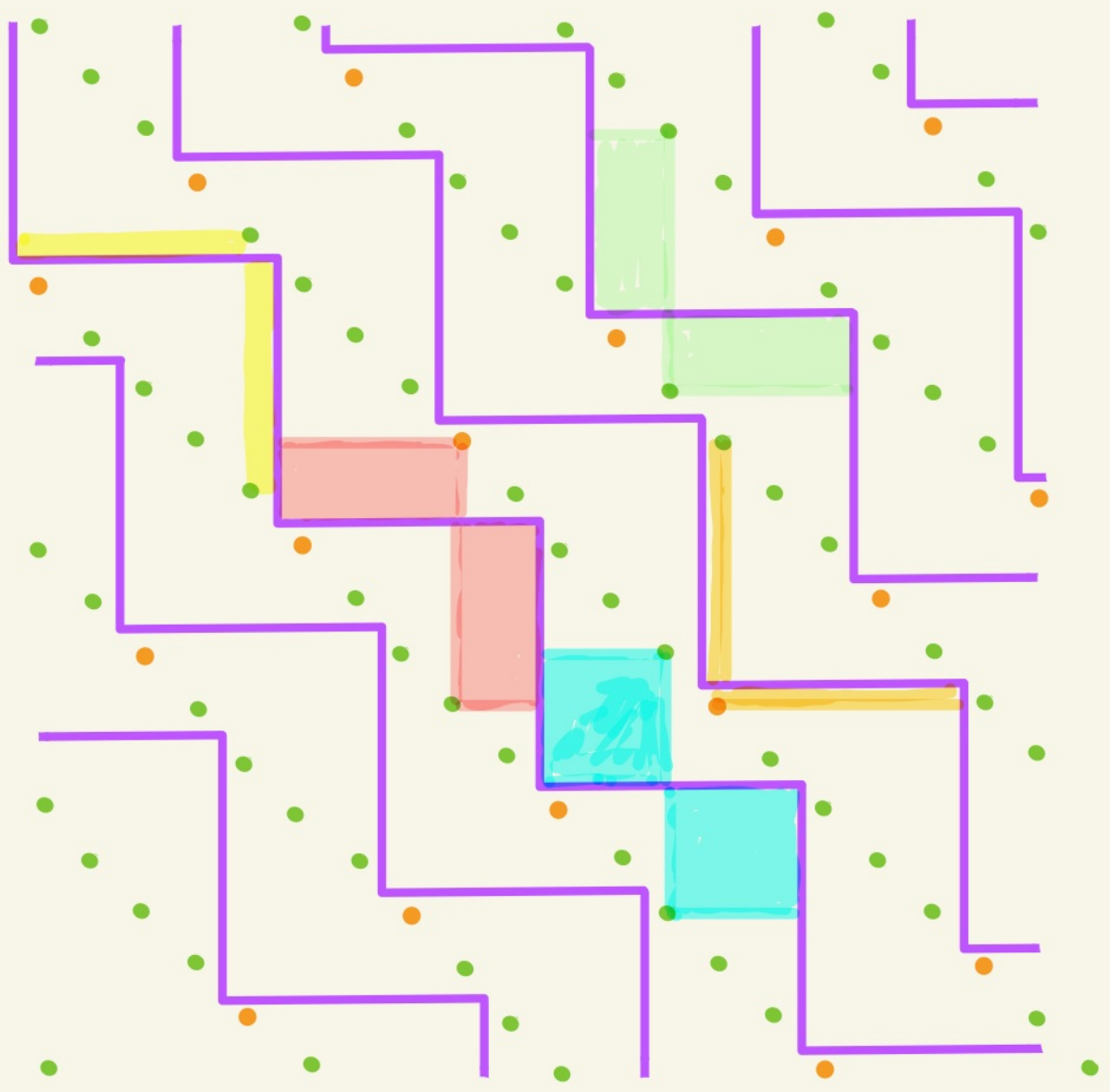}
\end{center}
\caption{The Hacking Lagrangian $\bH_{n,q}$ (left) and 
Lemma~\ref{wrgargaergh} (right).}\label{argargahaeha}\label{sgsFGsgsgsR}
\end{figure}

The relative Fukaya category 
$\mathcal{F}(\mathbb{T}_1, \{s \})$ is s $k[s]$-linear triangulated category.
Its~base change $\mathcal{F}(\mathbb{T}_1, \{s\}) \otimes_{k[s]} k[s^{\pm 1}]$
corresponds under mirror symmetry to the category of perfect complexes on the complement of the special fiber of the family  $\sE \to \bA^1_s$. 
Motivated by Figure~\ref{afbzfbdfnzdgnadtjn}, we  define an object $\bH\in\mathcal{F}(\mathbb{T}_1, \{s\}) \otimes_{k[s]} k[s^{\pm 1}]$. We will show in Lemma \ref{sVsgsrGwrgh} that its specialization to any $s\ne0$
is a mirror of the Hacking vector bundle $\mathcal{H}|_{\sE_s}$ (up to tensoring with a degree $0$ line bundle).

\begin{definition}
Let the {\it Hacking Lagrangian} $\bH := \bH_{n,q}$ be a Lagrangian illustrated (on the universal cover of the torus) in Figure~\ref{argargahaeha}  and equipped with a local system over $k[s^{\pm 1}]$ whose monodromy is given by $s^{-1}$. Note that $\bH$ has  one ``vertical'' and  one ``horizontal'' segment.
When doing computations for a Lagrangian endowed with a local system, one trivializes the local system on 
outside of a specified marked point, 
which we put near the orange dot where $\bH$ bends. 
Holomorphic curve counts are twisted by the monodoromy of the local system whenever the boundary of the holomorphic polygon passes through this marked point. 
 \end{definition}

\begin{lem}\label{sVsgsrGwrgh}
For $s\ne0$, the specialization $\bH_s$ of the Hacking Lagrangian
is the mirror of the restriction of the Hacking vector bundle $\cH|_{\sE_s}$
tensored with a degree $0$ line bundle.
\end{lem}

\begin{proof}
Let $\cH_s \in \Perf(\sE_s)$ be an object that corresponds to $\bH_s$ under homological mirror symmetry. Using Figure~\ref{figurehms}, we compute $\RHom(\cH_s, \cO_q) = k^n$ for every closed point $q \in \sE_s$. 
So~$\cH_s$ is a vector bundle of rank $n$. A surgery illustrated on the right side of Figure~\ref{s.DMNVBs,fbsjdhfb} shows that $c_1(\cH_s) = \frac{1}{n} c_1(\cF|_{\sE_s})$. By \eqref{dfdfgadhadh}, it follows that $\cH_s$ and $\cH|_{\sE_s}$ are vector bundles of the same rank and degree on the $1$-nodal cubic curve $\sE_s$. In fact, using Figure~\ref{figurehms} and homological mirror symmetry, we see that $R\Gamma(\sE_s, \cH_s) = k^q$ and $R\Gamma(\sE_s, \cF|_{\sE_s}) = k^{nq}$. In particular, $\deg \cH_s = \deg \cH|_{\sE_s} = q$ by Riemann-Roch. Since both $\cH|_{\sE_s}$ and $\cH_s$ are simple vector bundles (here we use \eqref{dfdfgadhadh} for $\cH|_{\sE_s}$ and homological mirror symmetry for $\cH_s$), they are both stable and isomorphic up to tensoring with a degree $0$ line bundle \cite{Bu}.
\end{proof}

Next, we work with the category
$\mathcal{F}(\mathbb{T}_1, \{s\}) \otimes_{k[s]} k[t]$
where $s=t^n$. We view the Kawamata Lagrangian $\bK$ endowed with a bounding cochain $\bb$
of Lemma~\ref{EFwegwrGWR} as an object $\bK_\bb$ of this category.

\begin{lem}\label{wrgargaergh}
We can pullback both $\bK_\bb$ and the Hacking Lagrangian $\bH$ to the category
$\mathcal{F}(\mathbb{T}_1, \{s\}) \otimes_{k[s]} k[t^{\pm1}]$. Then 
$\Hom(\bK_{\mathfrak b},\bH)\cong k[t,t^{-1}]^{\oplus n}$. 
\end{lem}

\begin{proof}
In Figure~\ref{afbzfbdfnzdgnadtjn}, the Kawamata Lagrangian (which is not shown) follows the grid (which includes orange and green points), while the Hacking Lagrangian goes halfway between the lines of the grid.
It follows that the Kawamata and Hacking Lagrangians intersect in $n$ points along the vertical segment of the Hacking Lagrangian and $n$ points along its horizontal segment. This will show that $\Hom(\bK_{\mathfrak b}, \bH) \cong k[t, t^{-1}]^{\oplus n}$ if we can prove that the differentials in the $A_\infty$-module $CF(\bK_\bb,\bH)$ vanish. The proof is the same as the proof of Lemma~\ref{EFwegwrGWR}: contributions to each differential come in pairs of permitted rectangles (illustrated on the right side of Figure~\ref{sgsFGsgsgsR}), which cancel each other out.
Concretely, in the notation of the proof of Lemma~\ref{EFwegwrGWR}
(see Figure~\ref{s.DMNVBs,fbsjdhfb}), the permitted rectangles have parameters $\alpha = y - nqy$ and $x + y = n$. Instead of \eqref{sDVSFBASFBAFD}, the cancellation condition becomes
$$[\alpha - y] = -n + [\alpha + x] + n^2\ (\hbox{\rm where we only add $n^2$ if}\ [\alpha + x] = 0.)$$
Two of the terms in the formula  \eqref{sDVSFBASFBAFD} do not appear in our case because the Hacking Lagrangian bends where the Kawamata Lagrangian intersects itself.
Also, a new term of $-n$ appears because the local system of the Hacking Lagrangian
contributes to the top rectangle of each pair. After simplification, the condition becomes
$[-nqy] = -n + [n - nqy] + n^2$, where we only add $n^2$ if $[n - nqy] = 0$. This~identity  is straightforward.
\end{proof}

\begin{lem}\label{afdgzdfhadrh}
Action of $\cR=\End(\bK_{\mathfrak b})$ on $\Hom(\bK_{\mathfrak b},\bH)\cong k[t,t^{-1}]^{\oplus n}$
gives a $k[t]$-linear map $\cR\to\Mat_n(k[t])$ 
that sends an element
$\sum\limits_{k\in\bZ_{n^2}} a_kw_k\in\cR$ 
%($i\in\bZ_{n^2}$)
to a matrix $A\in\Mat_n(k[t])$ with the matrix entries given in Theorem~\ref{wFGSRGARGARE}.
\end{lem}

\begin{proof}
We study the product $\Hom(\bK_{\bb},\bH)\otimes \Hom(\bK_\bb,\bK_\bb)\to \Hom(\bK_\bb,\bH)$
using the  basis $\{e_1,\ldots,e_n\}$ of $\Hom(\bK_{\bb},\bH)$ given by red dots in Figure~\ref{sRGrwgrherqhq} (where only the Hacking Lagrangian 
$\bH$ is shown). The corresponding matrix $A\in\Mat_n(k[t])$ is computed using the  holomorphic polygons illustrated in Figure~\ref{sRGrwgrherqhq}.

\begin{figure}[hbtp]
\begin{center}
\includegraphics[height=0.22\textheight]{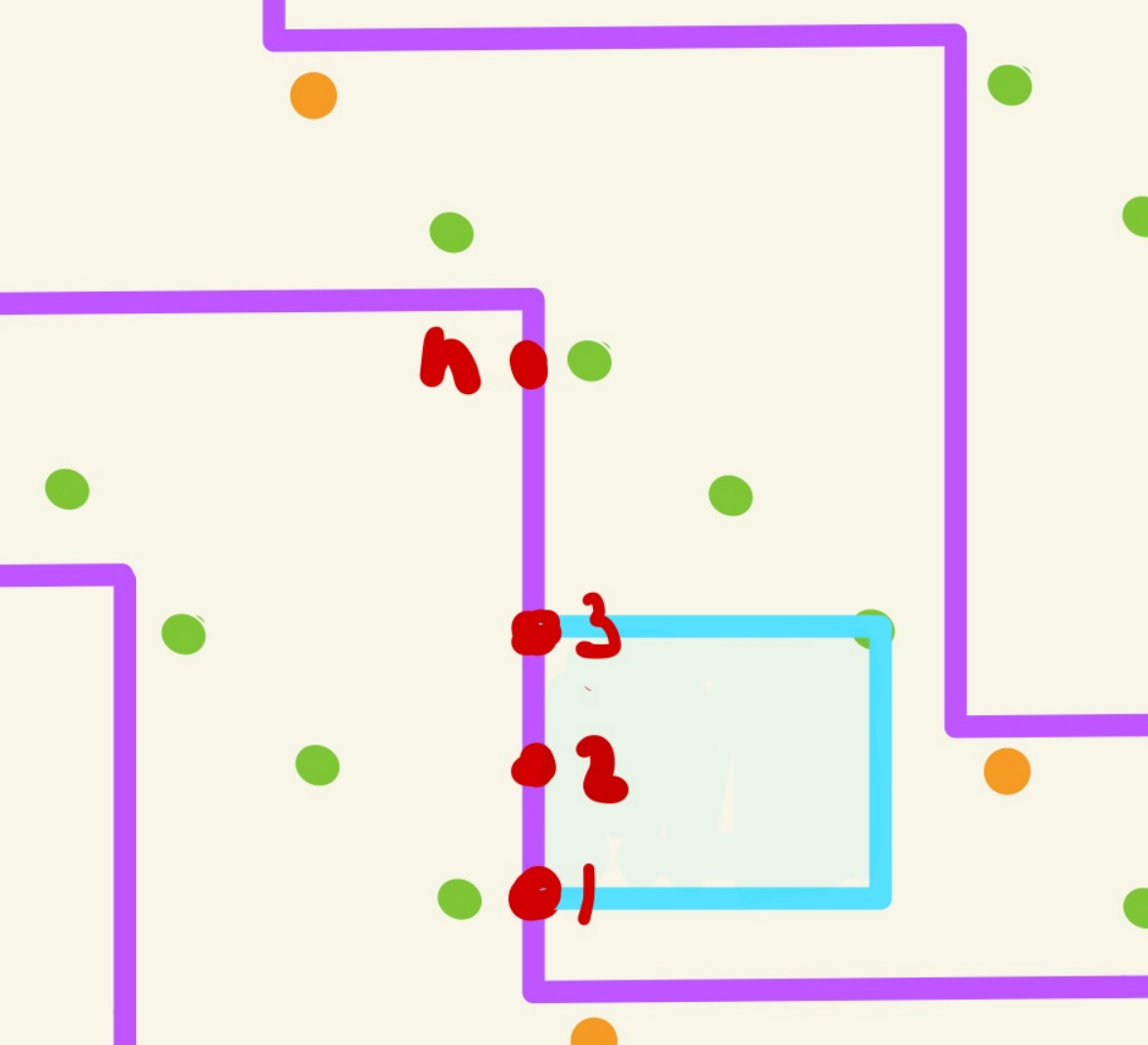}\qquad
\includegraphics[height=0.22\textheight]{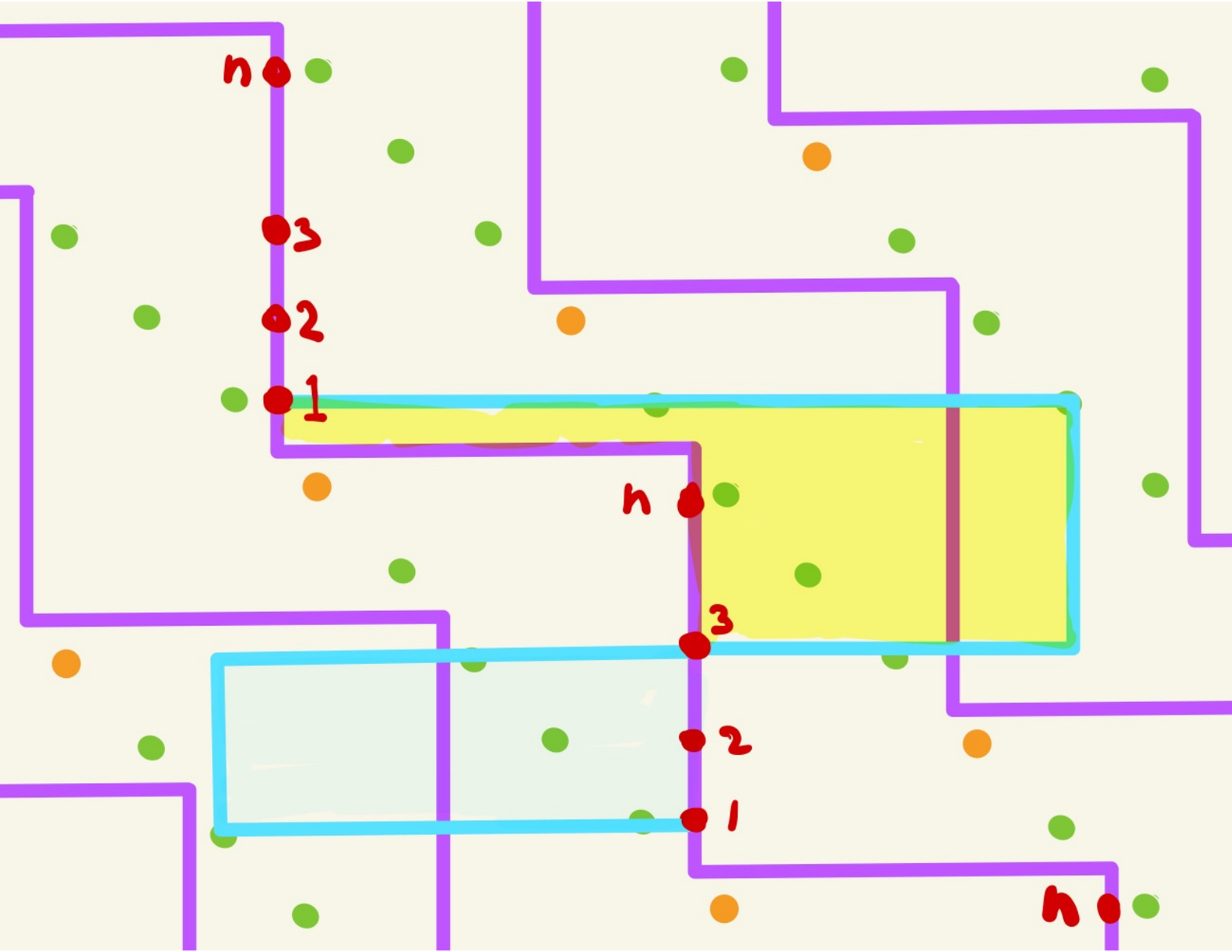}
\end{center}
\caption{Matrix entries $A_{ij}$ for $i>j$ (left) and $i<j$ (right)}\label{sRGrwgrherqhq}.
\end{figure}

When $i>j$, there is only one possibility: the polygon is a rectangle (possibly degenerated into a triangle) 
with the west side given by points $e_i$ and $e_j$, the NE corner either green or orange,
and no other orange points. 
The SE corner  is a basis element $w_k$ of $\Hom(\bK_\bb,\bK_\bb)$.
This is illustrated on the left  of Figure~\ref{sRGrwgrherqhq} (for $A_{31}$).

When $i<j$, there are two possibilities, which explains why the formulas in Theorem~\ref{wFGSRGARGARE}
are more complicated in this case.
 The first possibility, analogous to the case of $i>j$, is to take 
a rectangle (possibly degenerated into a triangle) 
with the east side given by points $e_i$ and $e_j$, the SW corner either green or orange,
and no other orange points.
The NW corner  gives a basis element $w_k$ of $\Hom(\bK_\bb,\bK_\bb)$.
This is illustrated by a blue rectangle on the right  of Figure~\ref{sRGrwgrherqhq} (for $A_{13}$).

The second possibility is to take a rectangle that starts at the point $e_i$, goes to the right along the Kawamata Lagrangian until a green (or orange point), goes down to a self-intersection point $w_k$ and then goes to the left to the point $e_j$. On the universal cover $e_j$
and $e_i$ are located on consecutive vertical parts of the Hacking Lagrangian.
This is illustrated by a yellow rectangle on the right side of Figure~\ref{sRGrwgrherqhq} (for $A_{13}$).
As always, the rectangle should not contain any orange points except, possibly, the NE corner.
The calculation of this composition includes a contribution from the local system on the Hacking Lagrangian.

Finally, contributions to $A_{ij}$ for $i=j$ are given by invisible polygons. The calculation here is entirely analogous to the proof of Corollary~\ref{argaergqehqeth}.
\end{proof}

It remains to prove the following lemma:

\begin{lem}\label{dfvdfbadfbdnadt}
Under homological mirror symmetry, 
the Lagrangian $\bK_\mathfrak{b}$ corresponds to the restriction $\cF|_\cE$ of the Kawamata vector bundle $\cF$
on the $\Q$-Gorenstein deformation $\cW$ of $W$ to the family of genus $1$ curves $\cE\subset\cW$
(up to tensoring with a line bundle).
\end{lem}

\begin{proof}
Let $\cK=\cF|_\cE$ and let 
$\cK'\in\Perf(\cE)$ be an object that corresponds to $\bK_\mathfrak{b}$ under homological mirror symmetry.
Both $\cK'$ and $\cK$ are deformations of the same vector bundle $F|_E$ on the special fiber $E\cong E_2$ of $\cE$.
In particular, $\cK'$ and $\cK$ are vector bundles of the same rank $n^2$ and degree~$nq$. It suffices to show that 
$\cK'_t$ and $\cK_t$ are isomorphic up to tensoring with a line bundle for every $t\ne0$.
As in Lemma~\ref{sVsgsrGwrgh}, let $\cH_t$ be the vector bundle that corresponds to the Lagrangian 
$\bH$ under mirror symmetry.
By~Lemma~\ref{sVsgsrGwrgh} and \eqref{dfdfgadhadh}, it suffices to prove that 
$\cK'_t\cong\cH_t^{\oplus n}$. By~Lemma~\ref{wrgargaergh} and mirror symmetry, we have 
$\Hom(\cK'_t,\cH_t)\cong k^{\oplus n}$. 
It follows that $\Hom(\cH_t, \cK'_t)\cong\Ext^1(\cK'_t,\cH_t)^*\cong k^{\oplus n}$
since $\cK'_t$ and $\cH_t$ have the same slope.
By~Lemma~\ref{svsFgasgr}, it suffices to prove that 
the bilinear pairing
$$
\Hom(\cK'_t,\cH_t)\otimes \Hom(\cH_t,\cK'_t)\to k, \quad e\otimes f\mapsto e\circ f\in\Hom(\cH_t,\cH_t)=k
$$
is non-degenerate.
Let $e_1,\ldots,e_n\in\Hom(\bK_\mathfrak{b},\bH)$ be the basis used in the proof of Lemma~\ref{afdgzdfhadrh}
(represented by  red dots in Figure~\ref{adfbafgasf}).
Let $\bH'$ is a small perturbation of the Hacking Lagrangian $\bH$
 and let 
$f_1,\ldots,f_n\in\Hom(\bH', \bK_\mathfrak{b})$ be a basis
represented by  blue dots in Figure~\ref{adfbafgasf}.
As this picture illustrates, these bases are dual under the bilinear pairing
since the only rectangles contributing to the composition are the obvious black rectangles illustrated in this picture.
\end{proof}

\begin{figure}[hbtp]
\begin{center}
\includegraphics[width=0.5\textwidth]{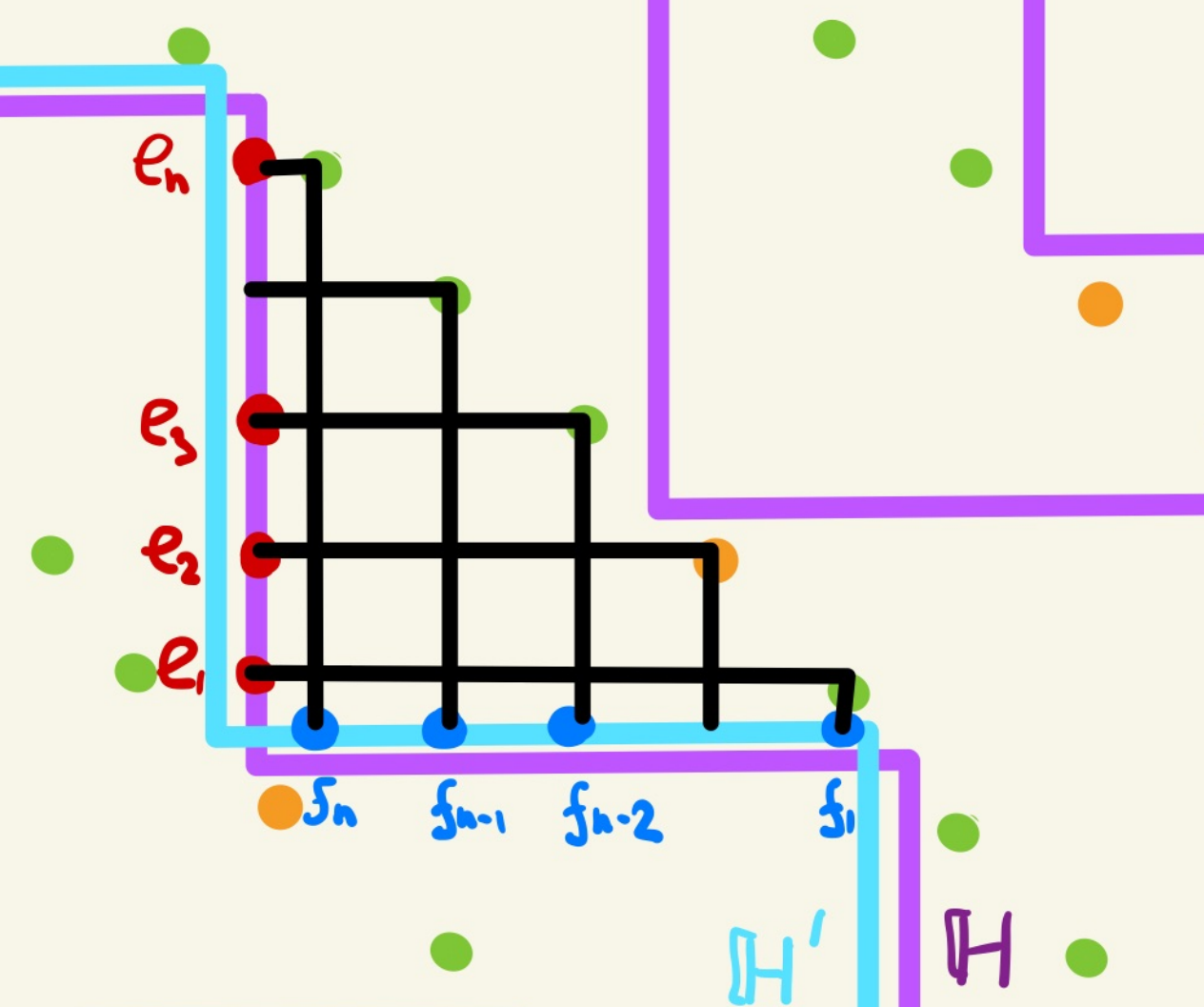}
\end{center}
\caption{Dual bases in $\Hom(\bK,\bH)$ and $\Hom(\bH',\bK)$}\label{adfbafgasf}.
\end{figure}

\begin{lem}\label{svsFgasgr}
Let $H$ be a simple vector bundle on a $k$-scheme $X$. Let $K$ be a vector bundle such that 
$\rk K=n\rk H$ for some integer $n$. Then $K$ is isomorphic to $H^{\oplus n}$ if and only if
$\Hom(K,H)$ and $\Hom(H,K)$ are $n$-dimensional vector spaces and the bilinear pairing
$$
\Hom(K,H)\otimes \Hom(H,K)\to k, \quad e\otimes f\mapsto e\circ f\in\Hom(H,H)=k
$$
is non-degenerate.
\end{lem}

\begin{proof}
The condition is certainly necessary. To show that it is sufficient, we choose dual bases
$e_1,\ldots,e_n\in \Hom(K,H)$ and $f_1,\ldots,f_n\in \Hom(H,K)$. Then the morphism
$H^{\oplus n}\to H^{\oplus n}$ given by the matrix $e_i\circ f_j$ for $i,j=1,\ldots,n$ is an identity isomorphism.
But it factors through the morphism $e:\,K\to H^{\oplus n}$ given by $e=(e_1,\ldots,e_n)$.
It follows that $e$ is a surjective morphism of locally free sheaves of the same rank. Therefore, $e$ is an isomorphism.
\end{proof}

\begin{proof}[Proof of Proposition~\ref{dfbsbF}]
Instead of using explicit formulas, we use presentation of the order from Corollary~\ref{argaergqehqeth}.
Recall that $s=t^n$.
Setting $\tilde w_i=w_i/s$ for $i\ne 0$ and $\tilde w_0=w_0$ in the formulas from Corollary~\ref{argaergqehqeth}
shows that the only products that survive in the limit as $t'=1/t\to0$ are the products 
$w_jw_i=sw_k$ appearing in the third polygon from the top, 
which in the limit become $\tilde w_j\tilde w_i=\tilde w_k$.\break
This algebra is isomorphic to $R_{n^2,nq-1}$ via an isomorphism 
$\tilde w_i\mapsto w_{-i}$.
\end{proof}

\end{document}